\RequirePackage[hyphens]{url}
\documentclass{article}

\pdfoutput=1

\usepackage{endnotes}
\let\footnote=\endnote

\usepackage{tikz}
\usetikzlibrary{positioning}
\usetikzlibrary{calc}
\usetikzlibrary{fit}

\usepackage{subcaption, multirow}
\usepackage{amsmath, mathtools}

\usepackage{xltabular}
\usepackage[ruled,linesnumbered]{algorithm2e}
\PassOptionsToPackage{hyphens}{url}\usepackage[hypertexnames=false]{hyperref}

\usepackage{enumitem}

\newcommand{\R}{\mathbb{R}}
\newcommand{\lBrack}{[\![}
\newcommand{\rBrack}{]\!]}

\DeclareMathOperator{\proj}{proj}
\DeclareMathOperator{\conv}{conv}
\DeclareMathOperator{\cone}{cone}

\usepackage{natbib}
 \bibpunct[, ]{(}{)}{,}{a}{}{,}
 
 \def\bibsep{\smallskipamount}

\usepackage{multibib}
\newcites{appendix}{References}

\usepackage{amsthm,amssymb}
\usepackage{setspace}
\usepackage[margin=0.8in]{geometry}
\usepackage{apptools}

\allowdisplaybreaks

\newtheorem{theorem}{Theorem}
\newtheorem{lemma}{Lemma}
\newtheorem{proposition}{Proposition}
\newtheorem{corollary}{Corollary}
\newtheorem{assumption}{Assumption}
\newtheorem{repeattheorem}{Theorem}
\newtheorem{repeatlemma}{Lemma}
\newtheorem{repeatproposition}{Proposition}

\AtAppendix{\counterwithin{lemma}{section}}
\AtAppendix{\counterwithin{corollary}{section}}
\AtAppendix{\counterwithin{figure}{section}}
\AtAppendix{\counterwithin{table}{section}}
\AtAppendix{\counterwithin{equation}{section}}
\AtAppendix{}
\hypersetup{
           breaklinks=true,
           colorlinks=false, 
           pdfusetitle=true,
        }
        
\title{BattOpt: Optimal Facility Planning for Electric Vehicle Battery Recycling}
\author{Matthew Brun\thanks{Operations Research Center, Massachusetts Institute of Technology, Cambridge, MA 02139. Email: \href{mailto:brunm@mit.edu}{brunm@mit.edu}} \and Xu Andy Sun\thanks{Sloan School of Management, Massachusetts Institute of Technology, Cambridge, MA 02139.  Email: \href{mailto:sunx@mit.edu}{sunx@mit.edu}}}
\date{\vspace{-3em}}

\begin{document}

\maketitle

%%Abstract
\begin{abstract}
The electric vehicle (EV) battery supply chain will face challenges in sourcing scarce and expensive minerals required for manufacturing and in disposing of hazardous retired batteries.  Integrating recycling technology into the supply chain has the potential to alleviate these issues; however, players in the battery market must design investment plans for recycling facilities. In this paper, we propose a multistage stochastic optimization model for computing minimum cost recycling capacity decisions, in which retired batteries are recycled and recovered materials are used to manufacture new batteries. To construct a realistic and high-fidelity model, we transform EverBatt, a leading evaluation framework for battery recycling cost and environmental impact, into a prescriptive decision-making tool for determining optimal investment strategies. Our model is a separable concave minimization subject to linear constraints, a class for which we describe a finitely convergent global optimization algorithm based on piecewise linear approximation that solves up to 14x faster than comparable algorithms. We propose an equivalent reformulation of the model that reduces the total number of variables by introducing integrality constraints.  The reformulation can also be solved by our global algorithm with drastically reduced solve times.  To produce a set of operational scenarios, we generate time-series projections for new battery demand, retired battery supply, and material costs, leveraging state-of-the-art econometric models for critical metal prices and EV demand.  We apply our model to compute optimal investment plans and quantify the impact of optimal decision-making on the supply chain.  The optimal solutions show that effective investment in recycling can reduce battery manufacturing costs by 22\% and environmental impacts by up to 26\%.  We measure the operational benefits of different recycling technologies, including material recovery efficiency and remanufacturing flexibility, and demonstrate how high uncertainty in material costs and EV adoption motivates greater investment in recycling capacity.  Finally, we analyze the impact of policy instruments to reveal key insights for policy makers: while recycling capacity grants and production credits both effectively encourage domestic investment in recycling, production credits are more cost effective under smaller policy budgets, whereas grants become the more cost-effective instrument as total policy expenditure increases.  
\end{abstract}


%%Introduction
\section{Introduction}

Electric vehicles (EVs) are a principal technology in the energy transition that will play an important role in reducing emissions.  The demand for the lithium-ion batteries that power EVs is expected to grow by over 20\% year-over-year through 2030, reaching a market valuation in excess of \$360 billion \citep{campagnol2022valuechain}.  Manufacture of these batteries requires \emph{critical minerals}, materials defined by a vulnerability to supply chain disruption, including lithium (Li), cobalt (Co), nickel (Ni), and manganese (Mn) \citep{usgs2022critical}.  As demand for batteries grows significantly, so will the demand for these materials; for instance, lithium requirements are expected to grow 6-fold by 2030 under current climate pledges \citep{iea2022supplychain}.

EV batteries generally are retired after a 20\% capacity reduction.  Batteries typically last between 10 and 20 years, although this lifespan depends strongly on charging patterns and environmental factors \citep{fioriti2023battery}.  As large numbers of batteries are removed from vehicles, attention must be paid to their safe disposal.  EV batteries contain toxic and heavy metals, so improper disposal may pose health and environmental risks \citep{kang2013potential}.  Further, recent European Union regulation sets recycled material composition requirements for new batteries and material recovery targets from retired batteries \citep{eu2023recyclingrequirement}.

In light of these challenges, a promising solution is the recycling of retired batteries, allowing for the recovery of component materials for reuse.  Lithium-ion battery packs are made up of many individual cells, each of which contain a graphite-based anode, lithium-based cathode, electrolytes, binding polymers, and copper and aluminum foils, inside a plastic or steel shell.  Common cathode compounds include lithium iron phosphate (LFP), lithium nickel cobalt aluminum oxide (NCA), and lithium nickel manganese cobalt oxide (NMC).  Many battery materials can be recovered via recycling, although the most valuable are the metals contained in cathode compounds.

Current technologies for lithium-ion battery recycling include pyrometallurgical, hydrometallurgical, and mechanical processes \citep{harper2019recycling}.  \textit{Pyrometallurgical recycling} exposes battery packs to high temperatures, incinerating extraneous battery components and inducing chemical reactions to recover compounds from the cathode material \citep{makuza2021pyrometallurgical}.  \textit{Hydrometallurgical recycling} reacts cathode materials with a leaching solution, from which compounds are precipitated \citep{yao2018hydrometallurgical}.  A newer process, referred to as \textit{direct recycling}, uses mechanical processes such as shredding and crushing to separate and recover constituent materials, followed by ``relithiation'' of cathode material to restore the chemical properties of a new material \citep{wu2023direct}.

EverBatt \citep{dai2019everbatt} is a recently proposed closed-loop model for analyzing the cost and environmental impacts of battery recycling and manufacturing with detailed, high-fidelity data for recycling process yields, facility costs, and emissions calculations.  In a \textit{closed-loop} model, recycled materials are used for the manufacturing of new products.  This differs from \textit{open-loop} models, in which recovered materials are sold or disposed.  Currently, Tesla is developing in-house battery recycling at its manufacturing facilities to perform closed-loop recycling \citep{forbes2021tesla}.  Other EV companies, including Audi, BMW, and Volkswagen, are partnering with external recyclers like Redwood Materials under an open-loop model \citep{forbes2022redwood,bmw2024partnership}.

In the EverBatt model, retired batteries are first recycled.  If pyrometallurgical or hydrometallurgical recycling is used, the extracted cathode precursors are remanufactured into cathode material in a step called \textit{cathode production}.  Under direct recycling, as cathode material is recovered in its original form, this step is not necessary.  Then, new batteries are manufactured using the recycled cathode material and other purchased materials.  Given fixed facility capacities that handle specific retired battery input and new battery output quantities, EverBatt computes the cost and environmental impact of recycling, cathode production, and manufacturing facilities across this closed-loop supply chain.
 
As players in the battery industry look towards incorporation of recycling technologies \citep{basf2023partnership}, important questions surrounding investment and capacity planning will arise.  Firms will decide how much recycling capacity to construct, which technologies to invest in, and when to use recycled materials to produce new batteries.  In this paper, we address these questions by converting the EverBatt framework into an optimization model.  
We summarize our contributions in modeling, algorithmic design and computation, and analysis as follows:
\begin{enumerate}[label=\textbf{\arabic*}., noitemsep]
    \item \textbf{Modeling:} 
        \begin{itemize}
            \item We propose a multistage stochastic optimization model for the design and operation of a closed-loop EV battery recycling supply chain. At the strategic scale, a multi-year capacity decision is made for recycling and cathode production facilities.  Then, at the operational scale, decisions are made over a number of multistage scenarios on how many batteries to recycle and how to use recycled materials to manufacture new batteries. This model is a prescriptive decision-making tool built on the detailed, high-fidelity evaluation framework of battery supply chain cost and environmental impact proposed in the EverBatt tool.
            \item We design a methodology that leverages state-of-the-art econometric models for EV vehicle stock and critical mineral price projections to build a variety of operational scenarios for retired battery supply, new battery demand, and material cost.  Our strategy leverages underlying projections of the energy transition to capture the correlation between EV adoption and critical mineral prices.
        \end{itemize}
    
    \item \textbf{Algorithmic design and computation:}
    \begin{itemize}
    \item We describe an adaptive piecewise linear approximation algorithm (aPWL) for mixed-integer linearly constrained separable concave minimization and prove a novel result on its finite convergence to a globally optimal solution.  In computational tests, we show that the aPWL algorithm compares favorably to benchmark algorithms, solving up to 14x faster on some instances. 
    \item We develop an equivalent reformulation of our model that reduces the total number of variables by introducing integrality constraints.  The reformulation reduces solve times from hours to minutes.
    \item We leverage a Benders' decomposition to separate the operational scale from the strategic scale.  We evaluate cut aggregation strategies within the Benders' decomposition and show that a grouping of similar scenarios outperforms traditional single-cut and multi-cut implementations.
    \end{itemize}
    
    \item \textbf{Managerial insights and policy analysis:} 
    \begin{itemize}
        \item We conduct extensive computational studies to analyze the impact of closed-loop recycling on new battery manufacturing. Our analysis shows that an optimal recycling investment plan in the United States reduces the expected cost of battery production by 22\%, the expected energy consumption by 6\%, and the expected greenhouse gas (GHG) emissions by 7\% through 2050, relative to a solution that does not utilize recycling. When including the disposal of unrecycled batteries, GHG emissions are reduced by up to 26\%.  We show similar reductions on a supply chain that includes the United States and China.
        \item We reveal a hidden benefit of hydrometallurgical recycling over direct recycling: the former allows greater flexibility in the types of new batteries that can be manufactured from recycled materials. We quantify the operational value of this flexibility when the cathode chemistries of retired and new batteries do not align.
        \item We analyze the impact of policy instruments, including capacity grants and production tax credits, on optimal investment and operational decisions.  Our analysis reveals key insights for policy makers. We show that both recycling capacity grants and production credits encourage domestic investment in recycling capacity.  However, production credits are more cost effective under smaller policy budgets, whereas grants become the more cost-effective instrument with greater total policy expenditure.
        \item We describe the response of firms to different levels of uncertainty in the distributions of new battery demand, retired battery supply, and material cost.  We demonstrate that greater uncertainty results in more recycling capacity construction, indicating that investment in recycling is motivated by the extremes of these distributions.
    \end{itemize}
\end{enumerate}

This work considers questions on optimal decision-making for the planning and operation of EV battery recycling supply chains.  These questions, and the models and algorithms we introduce to answer them, generalize beyond EV battery recycling to other recycling and remanufacturing industries.  These industries include recycling of solar photovoltaics \citep{choi2010design}, personal electronics \citep{sodhi2001models}, and plastics \citep{chaudhari2021systems}, among others.
Our framework can support investment planning for a broad range of recycling industries through the energy transition and into a sustainable future.

The remainder of the paper is organized as follows.  In Section~\ref{sec:litReview}, we review relevant literature.  In Section~\ref{sec:model}, we introduce our models for investment planning in the EV battery recycling supply chain.  In Section~\ref{sec:algorithm}, we describe an algorithm for solving linearly constrained separable concave minimization problems and propose a scale decomposition scheme.  In Section~\ref{sec:data}, we describe our sources of data and a process for multistage scenario construction.  In Section~\ref{sec:computation}, we present the computational performance of our algorithmic approaches.  In Section~\ref{sec:insight}, we analyze representative optimal investment decisions, describe the benefits of different recycling technologies, discuss the impact of policy instruments on the supply chain, and examine how firms respond to uncertainty.  In Section~\ref{sec:conclusion}, we conclude the paper.

%%Literature Review
\section{Literature Review}
\label{sec:litReview}

\subsection{Optimal Capacity Planning for Recycling Supply Chains}
Capacity planning problems for manufacturing applications have been studied extensively in the literature.  For thorough reviews, see \cite{van2003commissioned} and \cite{martinez2014review}.  Many studies consider a discrete set of feasible capacity decisions.  We instead focus on formulations with concave costs and continuous decisions.  In this problem class, \cite{lee1987multifacility} and \cite{rajagopalan1998capacity} simplify the continuous space to a discrete set of potentially optimal decisions and solve the problem via dynamic programming.  \cite{li1994dynamic} similarly compute a finite set of candidate capacities and propose heuristic algorithms.  In this work, we introduce a reformulation by identifying an underlying structure of optimal decisions, but do not enumerate all such solutions.

Previous work has explored capacity and logistics planning for recycling networks. \cite{jayaraman1999closed} and \cite{fleischmann2001impact} propose facility location models to facilitate reverse product flows for returns in closed-loop supply chains.  A variety of techniques have been introduced for designing recycling facility capacity plans, including simulations \citep{schultmann2003closed}, system dynamics models \citep{vlachos2007system,georgiadis2013flexible}, and mixed-integer linear and nonlinear programs \citep{lieckens2007reverse,zhen2018capacitated}.  \cite{kaya2014planning} define two-stage and robust capacity planning problems for recycling networks, with uncertainty in demand and return rates.  \cite{zeballos2014multi} extend to multistage capacity design models, where capacities are adjusted and facilities are opened or closed as uncertainty in supply and demand is revealed. Operationally, closed-loop recycling can reduce the environmental impacts of manufacturing \citep{bloemhof1996environmental}, increase profits \citep{li2018cost}, and mitigate high raw material prices in a competitive market \citep{raz2018recycling}.  Much of the literature on closed-loop supply chain capacity planning focuses product returns, whereas EV batteries are generally recycled at end-of-life.  In this work, we study capacity plans for recycling facilities, without a focus on the product returns and reverse network flows that are often considered in closed-loop supply chain literature.

Other studies explore the optimization of battery recycling supply chains.  \cite{hoyer2015technology}, \cite{tadaros2022location}, and \cite{rosenberg2023dynamic} study open-loop battery recycling processes.  \cite{hoyer2015technology} propose a mixed-integer linear program (MILP) for a deterministic multi-period model that optimizes capacity decisions for battery disassembly, conditioning, and recycling facilities in Germany.  Feasible capacity decisions are restricted to a discrete set of values, modeled by the construction of ``modules'' with fixed capacities.  \cite{tadaros2022location} and \cite{rosenberg2023dynamic} also propose multi-period deterministic MILPs for planning of open-loop EV battery recycling, the former for the Swedish supply chain and the latter the German.  \cite{tadaros2022location} formulate the problem as a facility location problem with fixed capacities, while \cite{rosenberg2023dynamic} optimize both the planned location of a facility as well as its capacity, selected from a discrete set of possible capacity levels.

\cite{li2018cost}, \cite{wang2020optimal}, and \cite{li2023optimization} introduce models for closed-loop battery recycling.  \cite{li2018cost} propose a multi-period model with stochasticity in the number and age of retired batteries avaiable for recycling, and \cite{wang2020optimal} investigate a single-period deterministic setting.  Both models are solved with genetic algorithms.  \cite{li2023optimization} consider optimization of the recycling network design in a single-period model with stochasticity in material recovery rate, new battery demand, and retired battery supply.  This model optimally locates facilities with fixed capacities and is solved as a two-stage MILP with second-stage decomposition.

Our work is unique in its consideration of a multi-period, stochastic, closed-loop battery recycling problem.  A multi-period stochastic model motivates new approaches for generating scenarios with parameters that vary across time.  Our model further chooses capacities from a continuous range, in contrast with the discrete feasible sets of other work on battery recycling.  This approach allows for more granularity in capacity planning and expansion without making a priori assumptions on the size of constructed facilities.  However, considering a continuous space of capacity decisions introduces concave cost functions due to economies of scale, making the model a nonconvex program instead of a MILP and increasing the difficulty of identifying globally optimal capacity decisions.

\subsection{Linearly Constrained Concave Programming}

Many important optimization problems can be formulated as a concave minimization over a polyhedron, including quadratic assignment \citep{bazaraa1982use}, binary programming \citep{raghavachari1969connections}, and applications with economies of scale \citep{feldman1966warehouse}.  Algorithms for globally solving concave programs date back to Tuy's cutting planes \citep{tuy1964concave}; subsequently, a variety of branch and bound algorithms have been proposed \citep{rosen1983global,benson1985finite}.  For a survey of methods, see \cite{pardalos1986methods} and \cite{horst2013handbook}.

We focus on problems with separable objective functions (i.e., representable as a sum of univariate functions).  \cite{shectman1998finite} propose a finitely convergent branch and bound algorithm for separable concave minimization over polyhedra, in which the feasible region is partitioned with rectangular splits and the objective is underestimated by linear functions, allowing the subproblems to be solved as linear programs.  This algorithm is similar to that of \cite{falk1969algorithm}, with improved finite convergence guarantees.  \cite{dambrosio2009global} propose a global algorithm for separable nonconvex optimization.  The algorithm splits each nonconvex function into convex and concave segments and approximates the concave segments by piecewise linear under-approximators.  The approximation is solved globally as a mixed-integer convex program, and the under-approximators of the convex pieces are updated at the incumbent iterate.  This algorithm converges in the limit to a global optimum.  \cite{magnanti2004separable} show approximate equivalence between MILP and linearly constrained separable concave minimization by generating piecewise linear approximations of the concave cost functions with a number of pieces that scales linearly in the inverse of the approximation error.

Our proposed algorithm solves over a sequence of piecewise linear under-approximations of the objective function, where each approximation is exact at previous iterates.  As opposed to \cite{falk1969algorithm} and \cite{shectman1998finite}, we partition on all variables simultaneously at a global optimum of the incumbent approximation.  This reduces the total number of iterations, improving compatibility with iterative algorithms for decomposition.  Additionally, by representing piecewise functions as MILPs, we leverage the tools, including cuts, heuristics, and branching techniques, incorporated into modern commercial solvers.  Similar approximation schemes appear in the global optimization literature.  \cite{dambrosio2009global} apply this method to general mixed-integer nonlinear optimization, a class which may not admit finite termination.  \cite{you2011stochastic} use successive piecewise linearization for concave minimization over mixed-integer linear regions, but do not discuss the finite termination of this approach.  To our knowledge, the finite convergence of this algorithm has not previously been described.

%%Model
\section{EV Battery Recycling and Manufacturing Model}
\label{sec:model}

We model the supply chain for EV batteries from the perspective of a battery manufacturer who invests in battery recycling to reduce the cost of procuring valuable raw materials. The supply chain consists of a sequence of production steps: retired battery recycling (REC), material conversion (MC), cathode  production (CP), and new battery manufacturing (NB).  Figure~\ref{fig:modelflows} depicts the material flow through these steps.  Battery recycling extracts materials from retired batteries via physical or chemical methods.  If the  recycling method recovers cathode powder, the cathode powder can be used directly for new battery manufacturing.  Otherwise, the recycled materials are used to produce new cathode powder in the cathode production step.  This step manufactures new cathode material using new or recycled precursors.  Material conversion is an intermediate step to convert recycled chemicals into the compounds required as inputs for cathode production.  Finally, in new battery manufacturing, EV battery cells are produced from cathode powder that is obtained from recycling or purchased from a market.  Under minor assumptions, our model and results also apply to investments by a recycler who sells recycled materials to battery manufacturers or other clients; this extension is discussed in Section~\ref{subsec:extensionOpenLoop}.

%%%%%%%%%%%%%%%%%% Figure %%%%%%%%%%%%%%%%%%
\begin{figure}[tp]
    \begin{center}
        \begin{tikzpicture}[scale=0.9]
        
            \tikzstyle{block} = [draw, rectangle, minimum height=2em, minimum width=6em]
        
            \node[block, align=center] (INV) at (-4,0) {Material Inventory\\\eqref{eq:constrINVMatFlowNoCP}, \eqref{eq:constrINVMatFlowCP}};
            \node[block, right=4em of INV, align=center] (MC) {Material Conversion\\\eqref{eq:constrMCProd}};
            \node[block, above=4em of MC, align=center] (NB) {New Battery Manufacturing\\\eqref{eq:constrNBProdCP}, \eqref{eq:constrNBProdNoCP}};
            \node[block, below=4em of MC, align=center] (CP) {Cathode Production\\\eqref{eq:constrCPProd}};
            \node[block, left=4em of INV, align=center] (REC) {Recycling\\\eqref{eq:constrRECProd}};
            \node[block, above=4em of REC, align=center] (RB) {Retired Battery Inventory\\\eqref{eq:constrRBMatFlow}};
        
            \draw[-stealth] (RB.south) -- (REC.north) node[midway,left]{$x^{\text{RB,RM}}$};
            \draw[-stealth] (REC.east) -- (INV.west) node[midway,above]{$x^{\text{RM,INV}}$};
        
            \draw[-stealth] (INV.east) -- (MC.west) node[midway,above]{$x^{\text{INV,MC}}$};
            \draw[-stealth] (MC.south) -- (CP.north) node[midway,left]{$x^{\text{MC,CP}}$};
            \draw[-stealth] (CP.west) -| (INV.south) node[near start,below]{$x^{\text{CP,INV}}$};
        
            \draw[-stealth] (INV.north) |- (NB.west) node[near end,above]{$x^{\text{INV,NB}}$};
        
            \draw[stealth-] (NB.east) -- ++ (6em,0) node[midway,above]{$x^{\text{NM,NB}}$};
            \draw[stealth-] (CP.east) -- ++ (6em,0) node[midway,above]{$x^{\text{NM,CP}}$};
        
            \draw[stealth-] (RB.north) -- ++ (0,5em) node[midway,left]{$s$};
            \draw[-stealth] (NB.north) -- ++ (0,5em) node[midway,left]{$d$};
            \draw[-stealth] (REC.south) -- ++ (0,-5em) node[midway,left]{$x^{\text{RM,S}\ }$};
            
        \end{tikzpicture}  
    \end{center}
    \caption{Material flows through the model in a single time period and location.  RB represents retired batteries available for recycling, INV materials in inventory, RM recycled materials, S materials sold, and NM new materials purchased.} \label{fig:modelflows}
\end{figure}
%%%%%%%%%%%%%%%%%% Figure %%%%%%%%%%%%%%%%%%

Recycling and cathode production require facilities, and facility capacities limit throughput.  We define demand for new battery cells as realized demand, which is the projected number of manufactured cells.  Under this definition, new battery demand must always be satisfied.  Although battery manufacturing is also limited by facility capacities, manufacturing capacity must always exceed the realized demand.  For this reason, we do not include manufacturing capacity decisions in our model.

Facilities can be located in a number of zones, which model differences in facility costs by location.  Zones may contain multiple facilities, and production runs in parallel across zones, with supply and demand for retired and new batteries allocated to each zone.  Recycled material and retired batteries in inventory can be transported between zones.  EverBatt \citep{dai2019everbatt} provides maximum facility capacities for its cost model.  To comply with these maxima, capacity within a zone is distributed over multiple facilities.  In EverBatt, facility costs vary with regional utility, labor, and capital costs; our zonal approach models regional differences in the cost structure, where facilities in the same zone have identical cost functions.

\subsection{Multistage Stochastic Battery Recycling Model}
We propose a multistage stochastic optimization problem for the EV battery supply chain.  The model covers two scales: a planning (strategic) scale, where facility capacity decisions are made, and an operational scale, where operational decisions, including recycling and manufacturing throughput, are made.  The operational decisions respond to uncertainty in new battery demand, retired battery supply, and material cost, which is revealed progressively through a multistage scenario tree.  A multistage approach ensures that the optimal solution at each period considers future uncertainty, allowing for a realistic representation of how operational decisions are made in practice.

The model time horizon spans $T$ periods (e.g., years), with $\mathcal{T} = \lBrack T \rBrack$.  We define $\lBrack \cdot \rBrack = \{1,\dots,\cdot\}$.  The periods are partitioned into a set of $L$ planning periods (e.g., spans of 5 years), $\mathcal{L} = \lBrack L \rBrack$, which define the temporal granularity of the investment decisions; that is, the model allows additional capacity investment to be made at the start of each planning period.  The capacity decision for a planning period constrains throughput in the time periods associated with the planning period.  The set $\mathcal{T}_l \subseteq \mathcal{T}$ gives the set of time periods associated with planning period $l \in \mathcal{L}$, and, conversely, $l_t \in \mathcal{L}$ gives the planning period associated with time period $t \in \mathcal{T}$.  

We separately partition the time periods $\mathcal{T}$ into a set of $S$ stages (e.g., spans of 10 years), $\mathcal{S} = \lBrack S \rBrack$, which describe the points at which additional information about the stochastic parameters is revealed.  The two distinct partitions of the time periods, i.e., planning periods and stages, allow for different granularity in the facility capacity solution and in the frequency that new information is revealed in the multistage scenario tree.  The scenario tree has $S$ stages, and the parameter $\sigma_t \in \mathcal{S}$ gives the stage associated with period $t \in \mathcal{T}$.  Set $\Omega_{\sigma}$ contains the nodes at stage $\sigma \in \mathcal{S}$.  Each node $\omega \in \Omega_{\sigma}$ gives a realization of the stochastic parameters (new battery demand, retired battery supply, and material cost) in the time periods associated with stage $\sigma$. A set of operational decisions in the associated periods is made at each node $\omega \in \Omega_{\sigma}$, in response to the realization of the stochastic parameters.  For example, if each stage contains 10 years, each node gives a set of parameters for the 10 years in its stage and contains operational decisions for the same 10 years.  The probability of reaching node $\omega$ in stage $\sigma$ is $p_{\omega}$, so $\sum_{\omega \in \Omega_\sigma} p_\omega = 1$.

To complete the description of the scenario tree, we define the paths which connect nodes to the root node of the tree.  For each node $\omega $, the function $a_{\omega}\,:\,\mathcal{T} \rightarrow \cup_{\sigma \in \mathcal{S}} \Omega_{\sigma}$ maps an input period $t$ to the ancestor of node $\omega$ at stage $\sigma_t$.  This function generalizes the parent of node $\omega$ by providing ancestors at all stages of the scenario tree.  Specifically, $a_{\omega}(t)$ gives the node at stage $\sigma_t$ which precedes $\omega$ on the unique path from the root node to node $\omega$.  If $t$ is a period in the same stage as $\omega$, i.e., $\omega \in \Omega_{\sigma_t}$, then $a_{\omega}(t) = \omega$.  This function is used in the inventory balance constraints \eqref{eq:constrMatFlow} to link variables which may appear in different stages.  For details on our data-driven construction of a scenario tree, see Section~\ref{sec:data}.

The set of zones within the supply chain is given by $\mathcal{Z}$.  We denote operational decisions by the variables $x$, which are indexed by time period $t \in \mathcal{T}$ and scenario tree node $\omega \in \Omega_{\sigma_t}$.  Variables $y$ are facility capacity decisions, which are indexed by planning period $l \in \mathcal{L}$.  Specific variables are introduced in the text as they appear; additionally, tables that summarize the model variables and parameters are included in Section~\ref{sec:ecModel} of the appendix.

\subsubsection*{Optimization Model}
The firm seeks to minimize total expected costs over the model horizon with an annual discount factor of $\gamma$.  
Overall, the investment model is as follows:
\begin{equation}
    \label{eq:deterministicP}
    \tag{P}
    \begin{aligned}
        \min_{x,y} \quad & \sum_{t \in \mathcal{T}} (1 - \gamma)^{t-1} \left ( C^{\text{PL}}_t(y) + \sum_{\omega \in \Omega_{\sigma_t}} p_\omega C^{\text{OP}}_{\omega,t}(x) \right ) \\
        \text{s.t.} \quad & \eqref{eq:constrProd},\ \eqref{eq:constrMatFlow},\ \eqref{eq:constrCap}\\
        & x \geq 0,\ y \geq 0.
    \end{aligned}
\end{equation}
Planning costs of the capacity decision $y$ are given by $C^{\text{PL}}_t(y)$, and operational costs of the operational decision $x$ are given in each scenario by $C^{\text{OP}}_{\omega,t}(x)$.  Capacity decisions are made over planning periods and leverage economies of scale, so the function $C^{\text{PL}}_t$ is concave and separable over elements of $y$ \citep{dai2019everbatt}.  Operational decisions are made real-time in each period and adapt to the revealed uncertainty, so we assume that these decisions do not experience economies of scale; the functions $C^{\text{OP}}_{\omega,t}$ are linear in $x$.  Linear constraints~\eqref{eq:constrProd},~\eqref{eq:constrMatFlow},~and~\eqref{eq:constrCap} enforce operational and capacity requirements and are presented in the remainder of this section.  The model~\eqref{eq:deterministicP} is a minimization of a separable concave function subject to linear constraints.

\subsubsection*{Production}
Constraints~\eqref{eq:constrProd} model the changes between materials during new battery manufacturing, material conversion, cathode production, and battery recycling.  These constraints apply for all periods $t \in \mathcal{T}$, zones $z \in \mathcal{Z}$, and scenarios $\omega \in \Omega_{\sigma_t}$.  We suppress subscripts for indices $(\omega,z,t)$, which appear on all variables $x$ and the demand $d$:
\begin{subequations}
    \label{eq:constrProd}
    \begin{align}
        & \sum_{i \in \mathcal{I}} \Delta^{\text{NB}}_{i,k} d_{i} = x^{\text{NM,NB}}_{k} + x^{\text{INV,NB}}_{k} \quad & \forall k \in \mathcal{K}^{\text{CP}}, \label{eq:constrNBProdCP}\\
        & \sum_{i \in \mathcal{I}} \Delta^{\text{NB}}_{i,k} d_{i} = x^{\text{NM,NB}}_{k} \quad & \forall k \in \mathcal{K} \setminus \mathcal{K}^{\text{CP}}, \label{eq:constrNBProdNoCP}\\
        & \sum_{k' \in \mathcal{K}^{\text{CP}}} \Delta^{\text{CP}}_{k',k} x^{\text{CP,INV}}_{k'} = x^{\text{NM,CP}}_{k} + x^{\text{MC,CP}}_{k} \quad & \forall k \in \mathcal{K} \setminus \mathcal{K}^{\text{CP}}, \label{eq:constrCPProd}\\
        & \sum_{k' \in \mathcal{K} \setminus \mathcal{K}^{\text{CP}}} \Delta^{\text{MC}}_{k',k} x^{\text{MC,CP}}_{k'} =  x^{\text{INV,MC}}_{k} \quad & \forall k \in \mathcal{K} \setminus \mathcal{K}^{\text{CP}}, \label{eq:constrMCProd}\\
        & x^{\text{RM,INV}}_{k} + x^{\text{RM,S}}_{k} = \sum_{i \in \mathcal{I},\ j \in \mathcal{J}} \Delta^{\text{REC}}_{k,i,j} x^{\text{RB,RM}}_{i,j} \quad & \forall k \in \mathcal{K}. \label{eq:constrRECProd}
    \end{align}
\end{subequations}
We remind that the process dynamics represented in constraints \eqref{eq:constrProd} are depicted in Figure~\ref{fig:modelflows}.  The set $\mathcal{K}$ gives the set of materials tracked across the production processes, such as graphite, copper, and carbon black.  The set of battery chemistries is given by $\mathcal{I}$, including chemistries of types NMC, NCA, and LFP.  Cathode powders are given by the set $\mathcal{K}^\text{CP} \subset \mathcal{K}$.  New batteries are manufactured from new or recycled cathode powder \eqref{eq:constrNBProdCP} and other non-recycled new materials (e.g., graphite, aluminum) \eqref{eq:constrNBProdNoCP} to meet battery demand $d_{\omega,z,t,i}$ for each chemistry.  Variables $x^{\text{NM,NB}}_{\omega,z,t,k}$ and $x^{\text{INV,NB}}_{\omega,z,t,k}$ give the mass of new and recycled material $k$, respectively, used in new battery production, and the parameter $\Delta^{\text{NB}}_{i,k}$ gives the amount of material $k$ needed to manufacture a unit mass of battery with chemistry $i$.  Equation~\eqref{eq:constrCPProd} models the production of cathode powder from new or recycled materials, where the variables $x^{\text{NM,CP}}_{\omega,z,t,k}$ and $x^{\text{MC,CP}}_{\omega,z,t,k}$ give the mass of new and recycled material inputs to cathode production, respectively.  Output from cathode production is given by the variable $x^{\text{CP,INV}}_{\omega,z,t,k}$.  Recycled materials used in cathode production first undergo material conversion, modeled by \eqref{eq:constrMCProd}.  Inputs to material conversion are given by variable $x^{\text{INV,MC}}_{\omega,z,t,k}$ and outputs by $x^{\text{MC,CP}}_{\omega,z,t,k}$.  The mass of material $k$ required to manufacture a unit of material $k'$ in cathode production and material conversion are given by parameters $\Delta^{\text{CP}}_{k',k}$ and $\Delta^{\text{MC}}_{k',k}$, respectively. Finally, the battery recycling processes are modeled by \eqref{eq:constrRECProd}.  The set $\mathcal{J}$ gives the available recycling processes (pyrometallurgical, hydrometallurgical, and direct), variable $x^{\text{RB,RM}}_{\omega,z,t,i,j}$ tracks the mass of batteries of chemistry $i$ recycled with process $j$, and $\Delta^{\text{REC}}_{k,i,j}$ gives the amount of material $k$ recovered per unit of recycled battery with chemistry $i$ when using process $j$.  The output materials can be placed into inventory to be used in manufacturing or sold to the market, modeled by variables $x^{\text{RM,INV}}_{\omega,z,t,k}$ and $x^{\text{RM,S}}_{\omega,z,t,k}$, respectively.  

\subsubsection*{Inventory Balance}
Constraints~\eqref{eq:constrMatFlow} model the flow of retired batteries and recycled materials between inventory and production processes.  For all $t \in \mathcal{T}$, $z \in \mathcal{Z}$, and $\omega \in \Omega_{\omega_t}$, these constraints take the following form:
\begin{subequations}
    \label{eq:constrMatFlow}
    \begin{align}
        & x^{\text{RB}}_{\omega,z,0,i} = 0 &\forall i \in \mathcal{I},\ \omega \in \Omega_{\sigma_0},\label{eq:constrRBMatFlowInit}\\
        & x^{\text{INV}}_{\omega,z,0,k} = 0 & \forall k \in \mathcal{K},\ \omega \in \Omega_{\sigma_0},\label{eq:constrINVMatFlowInit}\\
        & x^{\text{RB}}_{\omega,z,t,i} = x^{\text{RB}}_{\omega',z,t-1,i} + s_{\omega,z,t,i} 
        - \sum_{j \in \mathcal{J}} x^{\text{RB,RM}}_{\omega,z,t,i,j} + \sum_{z' \in \mathcal{Z} \setminus \{z\}} x^{\text{TR,RB}}_{\omega,z',z,t,i} - x^{\text{TR,RB}}_{\omega,z,z',t,i} & \forall i \in \mathcal{I}, \label{eq:constrRBMatFlow} \\ 
        & x^{\text{INV}}_{\omega,z,t,k} = x^{\text{INV}}_{\omega',z,t-1,k} + x^{\text{RM,INV}}_{\omega,z,t,k} - x^{\text{INV,MC}}_{\omega,z,t,k} + \sum_{z' \in \mathcal{Z} \setminus \{z\}} x^{\text{TR,RM}}_{\omega,z',z,t,k} - x^{\text{TR,RM}}_{\omega,z,z',t,k} \qquad & \forall k \in \mathcal{K} \setminus \mathcal{K}^\text{CP},\label{eq:constrINVMatFlowNoCP}\\
        & \mathrlap{x^{\text{INV}}_{\omega,z,t,k} = x^{\text{INV}}_{\omega',z,t-1,k} + x^{\text{RM,INV}}_{\omega,z,t,k} + x^{\text{CP,INV}}_{\omega,z,t,k} - x^{\text{INV,NB}}_{\omega,z,t,k} + \sum_{z' \in \mathcal{Z} \setminus \{z\}} x^{\text{TR,RM}}_{\omega,z',z,t,k} - x^{\text{TR,RM}}_{\omega,z,z',t,k}} & \forall k \in \mathcal{K}^\text{CP}, \label{eq:constrINVMatFlowCP}
    \end{align}
\end{subequations}
where $\omega' = a_\omega(t-1)$ models the change in scenario index over periods due to the multistage structure.  The mass of retired batteries with chemistry $i$ in inventory is given by variable $x^{\text{RB}}_{\omega,z,t,i}$, and inventory for material $k$ is given by $x^{\text{INV}}_{\omega,z,t,k}$.  The variables $x^{\text{TR,RM}}_{\omega,z,z',t,k}$ and $x^{\text{TR,RB}}_{\omega,z,z',t,i}$ represent the transport of material $k$ and of batteries with chemistry $i$, respectively, from zone $z$ to zone $z'$.  Constraints \eqref{eq:constrRBMatFlowInit}-\eqref{eq:constrINVMatFlowInit} enforce that initial inventory is zero.  Retired battery conservation across time periods is modeled by \eqref{eq:constrRBMatFlow}, where the mass of newly retired batteries is given by $s_{\omega,z,t,i}$, and batteries leave inventory to be recycled.  Flow balance for materials is modeled by \eqref{eq:constrINVMatFlowNoCP}-\eqref{eq:constrINVMatFlowCP}, where material enters inventory from battery recycling or cathode production and is removed for material conversion and new battery production.  Constraints \eqref{eq:constrRBMatFlow}-\eqref{eq:constrINVMatFlowCP} establish conservation of material across time periods, and therefore need to enforce inventory balance across stage boundaries in the multistage scenario tree.  Interactions across nodes and stages are handled exclusively by these constraints.

\subsubsection*{Capacity}
Constraints~\eqref{eq:constrCap} model recycling and cathode production facility capacities.  These constraints apply across all planning periods $l \in \mathcal{L}$ and zones $z \in \mathcal{Z}$:
\begin{subequations}
    \label{eq:constrCap}
    \begin{align}
        & \sum_{n \in \mathcal{N}^{\text{REC}}_l} y^{\text{REC}}_{z,l,j,n} \geq \sum_{i \in \mathcal{I}} x^{\text{RB,RM}}_{\omega,z,t,i,j} \quad & \forall j \in \mathcal{J},\ t \in \mathcal{T}_l,\ \omega \in \Omega_{\sigma_t}, \label{eq:constrCapREC}\\
        & \sum_{n \in \mathcal{N}^{\text{CP}}_{l,k}} y^{\text{CP}}_{z,l,k,n} \geq x^{\text{CP,INV}}_{\omega,z,t,k} \quad & \forall k \in \mathcal{K}^{\text{CP}},\ t \in \mathcal{T}_l,\ \omega \in \Omega_{\sigma_t}, \label{eq:constrCapCP}\\
        & \sum_{n \in \mathcal{N}^{\text{REC}}_l} y^{\text{REC}}_{z,l,j,n} \geq \sum_{n \in \mathcal{N}^{\text{REC}}_{l-1}} y^{\text{REC}}_{z,l-1,j,n} \quad & \forall j \in \mathcal{J}, \label{eq:constrCapRECIncrease}\\
        & \sum_{n \in \mathcal{N}^{\text{CP}}_{l,k}} y^{\text{CP}}_{z,l,k,n} \geq \sum_{n \in \mathcal{N}^{\text{CP}}_{l-1,k}} y^{\text{CP}}_{z,l-1,k,n} \quad & \forall k \in \mathcal{K}^{\text{CP}},\label{eq:constrCapCPIncrease}\\
        & y^{\text{REC}}_{z,l,j,n} \leq u^{\text{REC}} \quad & \forall j \in \mathcal{J},\ n \in \mathcal{N}^{\text{REC}}_l, \label{eq:constrCapRECUB}\\
        & y^{\text{CP}}_{z,l,k,n} \leq u^{\text{CP}} \quad & \forall k \in \mathcal{K}^{\text{CP}},\ n \in \mathcal{N}^{\text{CP}}_{l,k}. \label{eq:constrCapCPUB}
    \end{align}
\end{subequations}
Capacity decisions are made by planning period, so the facility capacity in planning period $l$ applies to all scenarios and time periods  $\mathcal{T}_l$ within the planning period.  Recycling capacity can be constructed at up to $N^{\text{REC}}_l$ identical facilities within each zone, where $\mathcal{N}^{\text{REC}}_l = \lBrack N^{\text{REC}}_l \rBrack$ is the facility index set.  Variable $y^{\text{REC}}_{z,l,j,n}$ gives the capacity for facility $n$ by recycling process $j$, zone $z$, and planning period $l$.  Similarly, parameter $N^{\text{CP}}_{l,k}$ gives the maximum number of production lines for cathode manufacturing.  Each cathode powder chemistry is manufactured on a different set of production lines.  The set $\mathcal{N}^{\text{CP}}_{l,k} = \lBrack N^{\text{CP}}_{l,k} \rBrack$ gives the line indices for each cathode powder type.  Variable $y^{\text{CP}}_{z,l,k,n}$ gives the capacity of line $n$ producing cathode powder $k$.  Constraint \eqref{eq:constrCapREC} limits the throughput of each recycling process by the total capacity across all facilities.  Similarly, cathode production throughput is constrained by \eqref{eq:constrCapCP}.  Constraints \eqref{eq:constrCapRECIncrease} and \eqref{eq:constrCapCPIncrease} require that total capacity is nondecreasing across planning periods.  Finally, recycling facilities and cathode production lines have maximum capacities of $u^{\text{REC}}$ and $u^{\text{CP}}$, respectively, enforced by \eqref{eq:constrCapRECUB} and \eqref{eq:constrCapCPUB}.

\subsubsection*{Cost}
Costs related to the production process are modeled by equations \eqref{eq:constrCost}.  The period index $t$ and scenario index $\omega$ are suppressed in all variables $x$ and cost coefficients $c$:
\begin{subequations}
    \label{eq:constrCost}
    \begin{align}
        & C^{\text{OP}}_{\omega,t}(x) = && \sum_{z \in \mathcal{Z}} \left ( \ \sum_{k \in \mathcal{K}} c^{\text{NB,NM}}_{k} x^{\text{NM,NB}}_{z,k} + \sum_{k \in \mathcal{K} \setminus \mathcal{K}^{\text{CP}}} \left ( c^{\text{CP,NM}}_{k} x^{\text{NM,CP}}_{z,k} + c^{\text{MC}}_{z,k} x^{\text{MC,CP}}_{z,k} \right ) \vphantom{\sum_{z' \in \mathcal{Z} \setminus \{z\}}} \right . \label{eq:constrCostOP} \\ 
        & && \qquad  + \sum_{k \in \mathcal{K}} v_{\omega,t,k} \cdot \left ( \rho x^{\text{INV}}_{z,k} - \eta x^{\text{RM,S}}_{z,k} \right ) + \sum_{k \in \mathcal{K}^{\text{CP}}} c^{\text{CP}}_{z,k} x^{\text{CP,INV}}_{z,k} + 
        \sum_{i \in \mathcal{I},\ j \in \mathcal{J}} c^{\text{REC}}_{z,i,j} x^{\text{RB,RM}}_{z,i,j} \nonumber \\
        & && \left . \qquad + \sum_{z' \in \mathcal{Z} \setminus \{z\}} \left ( \sum_{k \in \mathcal{K}} c^{\text{TR,RM}}_{z,z'} x^{\text{TR,RM}}_{z,z',k} + \sum_{i \in \mathcal{I}} c^{\text{TR,RB}}_{z,z'} x^{\text{TR,RB}}_{z,z',i} \right ) \right ) \quad \forall t \in \mathcal{T},\ \omega \in \Omega_{\sigma_t}, \nonumber\\
        & C^{\text{PL}}_t(y) = &&
         \sum_{z \in \mathcal{Z}} \left (\ \sum_{j \in \mathcal{J}} \sum_{n \in \mathcal{N}^{\text{REC}}_{l_t}} f^{\text{REC}}_{z,j}(y^{\text{REC}}_{z,l_t,j,n}) + \sum_{k \in \mathcal{K}^{\text{CP}}} \sum_{n \in \mathcal{N}^{\text{CP}}_{l_t,k}} f^{\text{CP}}_{z,k}(y^{\text{CP}}_{z,l_t,k,n}) \right ) \quad \forall t \in \mathcal{T}. &
         \label{eq:constrCostPL}
    \end{align}
\end{subequations}
Operational costs per time period, including material, utility, transportation, and inventory costs, as well as revenue incurred by selling recycled material, are functions of operational decisions $x$ and are represented by the functions $C^{\text{OP}}_{\omega,t}(x)$, defined in \eqref{eq:constrCostOP}.  These costs are linear in the operational decisions $x$, which are made in real time at each period and therefore do not leverage economies of scale.  The parameters $c^{\text{NB,NM}}_{\omega,t,k}$ and $c^{\text{CP,NM}}_{\omega,t,k}$ give the purchase costs per unit of new material $k$ used in new battery manufacturing and cathode production, respectively.  The variable costs of producing a unit of material $k$ via material conversion, manufacturing a unit of cathode powder $k$ via cathode production, and recycling a unit of batteries of chemistry $i$ via process $j$ are given by $c^{\text{MC}}_{\omega,z,t,k}$, $c^{\text{CP}}_{\omega,z,t,k}$, and $c^{\text{REC}}_{\omega,z,t,i,j}$, respectively.  Per-unit transportation costs from zone $z$ to zone $z'$ are given by $c^{\text{TR,RM}}_{\omega,z,z'}$ for materials and $c^{\text{TR,RB}}_{\omega,z,z'}$ for batteries.  The revenue generated per unit of recycled material $k$ sold is a proportion $\eta \leq 1$ of the value of the material $v_{\omega,t,k}$; selling extraneous recycled materials that are not used in downstream manufacturing helps offset recycling facility costs.  Inventory costs are similarly incurred as a proportion $\rho$ of the value of material held in inventory. We note that stochastic parameters, including all cost coefficients $c$, are indexed by node $\omega$.

Planning costs include capital and labor costs, and are represented by the functions $C^{\text{PL}}_t(y)$, defined in \eqref{eq:constrCostPL}.  Capital costs include costs for facility construction and equipment purchases.  As the capacity decisions $y$ are long-term decisions that are fixed over each planning period, the planning costs are concave due to economies of scale \citep{dai2019everbatt}. The function $C^{\text{PL}}_t(y)$ is also separable over variables $y$.  Planning costs are defined in terms of functions $f$, where $f^{\text{REC}}_{z,j}(y^{\text{REC}}_{z,l,j,n})$ gives the long-term costs of facilities by recycling process and zone, and $f^{\text{CP}}_{z,k}(y^{\text{CP}}_{z,l,k,n})$ gives those of a cathode production line producing material $k$.  The functions $f\ :\ \R_+ \rightarrow \R$ are concave lower semicontinuous piecewise with the form
\begin{equation}
    \label{eq:objFuncForm}
    f(y) = \begin{cases}
    \sum_i q_i y^{r_i} + w & y > 0\\
    0 & y = 0.
\end{cases}
\end{equation}
The functions introduce a cost of $0$ when no capacity is constructed, and otherwise are the sum of concave power terms ($r_i \in (0,1]$ and $q_i \geq 0$) and a constant term ($w \geq 0$), and thus are monotonic increasing.

\subsection{Model Properties}

In the remainder of this paper, we assume that Assumption~\ref{assump:dataProperties} holds.  The assumption enforces reasonable characteristics on the model data, including nonnegativity of costs, values, probabilities, supply, demand, material requirements, and maximum capacities.  The properties of the functions $f$ hold by their structure \eqref{eq:objFuncForm}, and discount factor $\gamma$ naturally falls on $[0,1)$.

\begin{assumption}
\label{assump:dataProperties}
The parameters $d$, $s$, $c$, $v$, $p$, and $\Delta$ are nonnegative, $u$ is positive, $\gamma \in [0,1)$, and functions $f$ are concave monotonic increasing with $f(0) = 0$.
\end{assumption}

Each scenario has a relatively complete recourse property, given in Proposition~\ref{prop:fullRecourse}.  This property ensures that, for any feasible capacity decision $y$, there is some feasible operational decision $x$.  We further show in Proposition~\ref{prop:finiteP} that the model is feasible and has an optimal solution at an extreme point of its feasible region.  These and other original proofs are provided in Section~\ref{sec:ecProofs} of the appendix.

\begin{proposition}
    \label{prop:fullRecourse}
    For any $\overline{y} \geq 0$ that satisfies \eqref{eq:constrCapRECIncrease}-\eqref{eq:constrCapCPUB}, $\{(y,x) \geq 0 \ :\  \eqref{eq:constrProd},\ \eqref{eq:constrMatFlow},\ \eqref{eq:constrCap},\ y = \overline{y}\} \neq \emptyset$.
\end{proposition}

\begin{proposition}
    \label{prop:finiteP}
    The model \eqref{eq:deterministicP} is feasible and has an optimal solution at an extreme point of its feasible region.
\end{proposition}

For convenience of notation, we introduce the operator $\mathcal{C}(\mathcal{X};v)$ which counts the number of elements of a set $\mathcal{X}$ that take the value $\nu$: $\mathcal{C}(\mathcal{X};\nu) = \sum_{x \in \mathcal{X}} \mathbb{I}(x = \nu)$.  In Theorem~\ref{thm:extPtStructure}, we note that any extreme point of the feasible region of \eqref{eq:deterministicP} satisfies a specific structure in the capacity decision $y$.  Specifically, each facility is constructed to its upper bound before adding any capacity to the next facility.  For a recycling solution $y^{\text{REC}}$ with total capacity $Y_{z,l,j} := \sum_{n \in \mathcal{N}^{\text{REC}}_l} y^{\text{REC}}_{z,l,j,n}$, the structure is defined by
\begin{equation}
    \label{eq:recExtPtProperty}
    \begin{aligned}
        \mathcal{C}(\{y^{\text{REC}}_{z,l,j,n}\}_{n \in \mathcal{N}^{\text{REC}}_l};u^{\text{REC}}) & \geq \left \lceil \frac{Y_{z,l,j}}{u^{\text{REC}}} \right \rceil - 1;\\
        \mathcal{C}(\{y^{\text{REC}}_{z,l,j,n}\}_{n \in \mathcal{N}^{\text{REC}}_l};0) & \geq |\mathcal{N}^{\text{REC}}_l| - \left \lceil \frac{Y_{z,l,j}}{u^{\text{REC}}} \right \rceil \quad \forall j \in \mathcal{J},\ l \in \mathcal{L},\ z \in \mathcal{Z},
    \end{aligned}
\end{equation}
with a corresponding structure for $y^{\text{CP}}_{z,l,k,n}$ defined in the appendix, equation \eqref{eq:CPExtPtProperty}.
Under this structure, there is at most one facility with capacity not equal to its upper or lower bound.  As \eqref{eq:deterministicP} has an optimal extreme point, we conclude that the model has an optimal solution with this structure.

\begin{theorem}
    \label{thm:extPtStructure}
    Let $(\tilde{y},\tilde{x})$ be an extreme point of $\{(y,x) \geq 0\ :\ \eqref{eq:constrProd},\ \eqref{eq:constrMatFlow},\ \eqref{eq:constrCap}\}$.  Then, $\tilde{y}^{\text{REC}}$ satisfies \eqref{eq:recExtPtProperty} and $\tilde{y}^{\text{CP}}$ satisfies the corresponding property.
\end{theorem}

\subsection{Reformulating the Model}

Under the structure of the optimal solution implied by Theorem~\ref{thm:extPtStructure}, we reformulate the capacity decision variables $y$ and associated constraints \eqref{eq:constrCap}.  In the reformulation, integer variables $y^{\text{REC}}_{z,l,j}$ and $y^{\text{CP}}_{z,l,k}$ give the number of facilities with capacities at their upper bound, and continuous variables $y^{\text{REC}}_{z,l,j,+}$ and $y^{\text{CP}}_{z,l,k,+}$ give the capacity of the single facility not at its upper or lower bound.  For all $z \in \mathcal{Z}$ and $l \in \mathcal{L}$, the reformulation enforces the following constraints:
\begin{subequations}
    \label{eq:constrCapRef}
    \begin{align}
        & u^{\text{REC}} y^{\text{REC}}_{z,l,j} + y^{\text{REC}}_{z,l,j,+} \geq \sum_{i \in \mathcal{I}} x^{\text{RB,RM}}_{\omega,z,t,i,j} \quad & \forall j \in \mathcal{J},\ t \in \mathcal{T}_l,\ \omega \in \Omega_{\sigma_t}, \label{eq:constrCapRECRef}\\ 
        & u^{\text{CP}} y^{\text{CP}}_{z,l,k} + y^{\text{CP}}_{z,l,k,+} \geq x^{\text{CP,INV}}_{\omega,z,t,k} \quad & \forall k \in \mathcal{K}^{\text{CP}},\ t \in \mathcal{T}_l,\ \omega \in \Omega_{\sigma_t}, \label{eq:constrCapCPRef}\\
        & u^{\text{REC}} y^{\text{REC}}_{z,l,j} + y^{\text{REC}}_{z,l,j,+} \geq u^{\text{REC}} y^{\text{REC}}_{z,l-1,j} + y^{\text{REC}}_{z,l-1,j,+} \quad & \forall j \in \mathcal{J}, \label{eq:constrCapRECIncreaseRef}\\
        & u^{\text{CP}} y^{\text{CP}}_{z,l,k} + y^{\text{CP}}_{z,l,k,+} \geq u^{\text{CP}} y^{\text{CP}}_{z,l-1,k} + y^{\text{CP}}_{z,l-1,k,+} \quad & \forall k \in \mathcal{K}^{\text{CP}}, \label{eq:constrCapCPIncreaseRef}\\
        & y^{\text{REC}}_{z,l,j} \in \{0,\dots,|\mathcal{N}^{\text{REC}}_l|-1\} \quad & \forall j \in \mathcal{J}, \label{eq:constrCapRECInt}\\
        & y^{\text{CP}}_{z,l,k} \in \{0,\dots,|\mathcal{N}^{\text{CP}}_{l,k}|-1\} \quad & \forall k \in \mathcal{K}^{\text{CP}}, \label{eq:constrCapCPInt}\\
        & y^{\text{REC}}_{z,l,j,+} \leq u^{\text{REC}} \quad & \forall j \in \mathcal{J}, \label{eq:constrCapRECUBRef}\\
        & y^{\text{CP}}_{z,l,k,+} \leq u^{\text{CP}} \quad & \forall k \in \mathcal{K}^{\text{CP}}. \label{eq:constrCapCPUBRef}
    \end{align}
\end{subequations}
We also introduce for all $t \in \mathcal{T}$ an adjusted version of the planning cost function:
\begin{equation*}
    \overline{C}^{\text{PL}}_t(y) = \sum_{z \in \mathcal{Z}} \left ( \sum_{j \in \mathcal{J}} \left ( f^{\text{REC}}_{z,j}(y^{\text{REC}}_{z,l_t,j,+}) + y^{\text{REC}}_{z,l_t,j} f^{\text{REC}}_{z,j}(u^{\text{REC}} ) \right ) + \sum_{k \in \mathcal{K}^{\text{CP}}} \left ( f^{\text{CP}}_{z,k}(y^{\text{CP}}_{z,l_t,k,+}) + y^{\text{CP}}_{z,l_t,k} f^{\text{CP}}_{k,z}(u^{\text{CP}}) \right ) \right ).
\end{equation*}

The full reformulated model is then given by
\begin{equation}
    \label{eq:reformulatedP}
    \tag{PMI}
    \begin{aligned}
        \min_{x,y} \quad & \sum_{t \in \mathcal{T}} (1 - \gamma)^{t-1} \left ( \overline{C}^{\text{PL}}_t(y) + \sum_{\omega \in \Omega_{\sigma_t}} p_\omega C^{\text{OP}}_{\omega,t}(x) \right ) \\
        \text{s.t.} \quad & \eqref{eq:constrProd},\ \eqref{eq:constrMatFlow},\ \eqref{eq:constrCapRef}\\
        & x \geq 0,\ y \geq 0.
    \end{aligned}
\end{equation}
The new model \eqref{eq:reformulatedP} is equivalent to the original formulation \eqref{eq:deterministicP}, shown in Theorem~\ref{thm:reformulationEquivalent}.  It also retains the linearly constrained separable concave structure of \eqref{eq:deterministicP} with the addition of integrality constraints.  Although these constraints add some complexity, the number of capacity variables is significantly reduced as they are no longer indexed by $\mathcal{N}^{\text{REC}}$ and $\mathcal{N}^{\text{CP}}$.

\begin{theorem}
    \label{thm:reformulationEquivalent}
    The models \eqref{eq:deterministicP} and \eqref{eq:reformulatedP} have the same optimal objective value.
\end{theorem}

The reformulation relies heavily on the assumption that the planning cost functions are concave and identical across recycling facilities and cathode production lines within a zone.  Although the form of the cost functions can vary across zones, this assumption may still be violated by realistic cost models.  In the following section, we describe some modeling extensions that incorporate alternative facility cost structures.

\subsection{Modeling Extensions}
\label{sec:extension}

The model~\eqref{eq:deterministicP} enforces strong assumptions on the facility cost functions which permit a compact reformulation as~\eqref{eq:reformulatedP}.  Specifically, we assume that the cost functions are identical across facilities in a zone and depend only on the capacity in a single planning period.  This structure may not apply to general, realistic cost models; for instance, facilities may have fixed costs of expansion, or experience diseconomies of scale with increasing marginal costs of capacity.  In Section~\ref{subsec:extensionPWC}, we introduce a mixed-integer concave reformulation for a setting with piecewise concave cost functions, which can model diseconomies of scale or other more general cost structures.  In Section~\ref{subsec:extensionFixedC}, we provide a reformulation when the cost functions include fixed costs of facility expansion.  Theoretical justifications of these extensions are provided in Section~\ref{sec:ecExtensions} of the appendix.  We present both reformulations only for the recycling facility costs; the same strategies apply to the cathode production line costs.

It is also of interest to understand optimal investment behaviors of battery recyclers who seek a profit by selling recycled materials to a market (i.e., open-loop recyclers).  Our model focuses only on closed-loop recyclers, who reduce the cost of manufacturing new batteries by investing in recycling facilities.  In Section~\ref{subsec:extensionOpenLoop}, we discuss general conditions under which optimal solutions derived from our closed-loop model also apply to an open-loop setting.  Finally, Section~\ref{subsec:generalExtension} describes applications of our model beyond EV battery recycling.

\subsubsection{Piecewise Concave Cost Functions}
\label{subsec:extensionPWC}
The concave cost functions $f$ model economies of scale in facility planning costs, with a capacity upper bound.  However, realistic facility costs may not conform to this model, due to discontinuities in the cost function for different capacity regimes or diseconomies of scale for large facilities instead of a static bound on capacity.  Here, we describe a reformulation in the style of~\eqref{eq:reformulatedP} when the facility cost functions are piecewise concave.  Piecewise concavity permits modeling of capacity regimes by concave segments, or of diseconomies of scale by convex piecewise linear regions.

Suppose that the recycling planning cost functions $f^{\mathrm{REC}}_{z,j}$ are lower semicontinuous piecewise concave with domain $[0,u^{\mathrm{REC}}_Q]$: 
\begin{equation}
    f^{\mathrm{REC}}_{z,j}(y) = \sum_{q = 1}^Q \mathbb{I}[y \in (u^{\mathrm{REC}}_{q-1}, u^{\mathrm{REC}}_{q}]] f^{\mathrm{REC}}_{z,j,q}(y).
\end{equation}
These functions must satisfy the following properties for all $q \in \lBrack Q \rBrack$: $f^{\mathrm{REC}}_{z,j,q}(y)$ are concave on $[u^{\mathrm{REC}}_{q-1},u^{\mathrm{REC}}_q]$; $u^{\mathrm{REC}}_q < u^{\mathrm{REC}}_{q+1}$; and $f^{\mathrm{REC}}_{z,j,q}(u^{\mathrm{REC}}_q) = f^{\mathrm{REC}}_{z,j,q+1}(u^{\mathrm{REC}}_q)$.  We define $u^{\mathrm{REC}}_0 = 0$.  To reformulate the problem in this setting, we introduce integer variables $y^{\mathrm{REC}}_{z,l,j,q}$ for each segment $q$ that count the number of facilities at the interval upper bound $u^{\mathrm{REC}}_q$.  We also introduce continuous variables $y^{\mathrm{REC}}_{z,l,j,q,+}$ which determine the capacity of the single facility with capacity not at an interval endpoint.  We refer to the interval containing this continuous capacity as the \emph{active} interval.  The variable $y^{\mathrm{REC}}_{z,l,j,q,+}$ gives the continuous facility capacity if the interval $[u^{\mathrm{REC}}_{q-1},u^{\mathrm{REC}}_q]$ is active, and otherwise is fixed to the interval endpoint $u^{\mathrm{REC}}_{q-1}$.  Binary variables $y^{\mathrm{REC,I}}_{z,l,j,q}$ indicate whether the $q$-th interval is active.  Under this reformulation, the total recycling capacity of process $j$ in zone $z$ and period $l$ is 
\begin{equation}
    \label{eq:PWCTotalCap}
    \sum_{q=1}^Q u^{\mathrm{REC}}_q y^{\mathrm{REC}}_{z,l,j,q} + y^{\mathrm{REC}}_{z,l,j,q,+} - u^{\mathrm{REC}}_{q-1} (1-y^{\mathrm{REC,I}}_{z,l,j,q}),
\end{equation}
where the final term removes contributions from variables $y^{\mathrm{REC}}_{z,l,j,q,+}$ on intervals that are inactive.  The throughput constraint~\eqref{eq:constrCapREC} and monotonicity constraint~\eqref{eq:constrCapRECIncrease} can be written by replacing the total capacity $\sum_{n \in \mathcal{N}^{\mathrm{REC}}_l} y^{\mathrm{REC}}_{z,l,j,n}$ with~\eqref{eq:PWCTotalCap}.
For all $l \in \mathcal{L}$ and $z \in \mathcal{Z}$, the variables $y^{\mathrm{REC}}$ are additionally subject to the following logic:
\begin{subequations}
    \label{eq:PWCreformulation}
    \begin{align}
        & \sum_{q = 1}^Q y^{\mathrm{REC}}_{z,l,j,q} \leq |\mathcal{N}^{\mathrm{REC}}_l| - 1  \quad & \forall j \in \mathcal{J},\label{eq:PWCreformulationCount}\\
        & \sum_{q = 1}^Q y^{\mathrm{REC,I}}_{z,l,j,q} =  1  \quad & \forall j \in \mathcal{J},\label{eq:PWCreformulationIndicCount}\\
        & y^{\mathrm{REC}}_{z,l,j,q} \in \{0,\dots,|\mathcal{N}^{\mathrm{REC}}_l| - 1\},\ y^{\mathrm{REC,I}}_{z,l,j,q} \in \{0,1\} \quad & \forall q \in \lBrack Q \rBrack,\ j \in \mathcal{J},\label{eq:PWCreformulationDomain}\\
        & u^{\mathrm{REC}}_{q-1} \leq y^{\mathrm{REC}}_{z,l,j,q,+} \leq u^{\mathrm{REC}}_{q} y^{\mathrm{REC,I}}_{z,l,j,q} + u^{\mathrm{REC}}_{q-1} (1-y^{\mathrm{REC,I}}_{z,l,j,q}) \quad & \forall q \in \lBrack Q \rBrack,\ j \in \mathcal{J}.\label{eq:PWCreformulationIndicDef}
    \end{align}
\end{subequations}
Constraint~\eqref{eq:PWCreformulationCount} bounds the number of facilities with capacities set at an interval endpoint, and constraint~\eqref{eq:PWCreformulationIndicCount} enforces that one indicator for the continuous capacity is active.  Constraint~\eqref{eq:PWCreformulationDomain} gives integrality requirements, and constraint~\eqref{eq:PWCreformulationIndicDef} establishes  interval constraints for the active interval and binds $y^{\mathrm{REC}}_{z,l,j,q,+}$ to the interval endpoint for inactive intervals.
The recycling planning cost contribution in period $t$ of this formulation is
\begin{equation}
    \sum_{z \in \mathcal{Z}} \sum_{j \in \mathcal{J}} \sum_{q=1}^Q\left (y^{\mathrm{REC}}_{z,l_t,j,q} f^{\mathrm{REC}}_{z,j,q} (u^{\mathrm{REC}}_q) + f^{\mathrm{REC}}_{z,j,q}(y^{\mathrm{REC}}_{z,l_t,j,q,+}) - (1-y^{\mathrm{REC,I}}_{z,l,j,q})f^{\mathrm{REC}}_{z,j,q}(u^{\mathrm{REC}}_{q-1})\right ).
\end{equation}
The final term again removes cost contributions from variables $y^{\mathrm{REC}}_{z,l,j,q,+}$ on inactive intervals.  The cathode production capacity logic can be reformulated similarly.  This reformulation introduces three variables for every piecewise segment $q$: two integer variables and one variable with a concave objective function.  However, this reformulation maintains the benefit of formulation~\eqref{eq:reformulatedP}, as the capacity variables are not indexed by $\mathcal{N}^{\mathrm{REC}}$ and $\mathcal{N}^{\mathrm{CP}}$.

\subsubsection{Fixed Cost of Expansion}
\label{subsec:extensionFixedC}
The facility costs $f$ amortize the cost of capacity over the lifetime of the facility.  However, increasing the capacity of an existing facility may incur an additional fixed cost, beyond the amortized cost of constructing a new facility of the same capacity.  Here, we describe a reformulation that incorporates fixed costs of facility expansion.

Suppose that the planning cost function additionally contains the following terms, which incur a fixed cost $f^{\mathrm{EXP}}_{z,l,j} > 0$ for expanding a recycling facility:
\begin{equation}
    \label{eq:expansionCostDef}
    \sum_{n \in \mathcal{N}^{\mathrm{REC}}_{l}} f^{\mathrm{EXP}}_{z,l,j} \mathbb{I}[y^{\mathrm{REC}}_{z,l,j,n} > y^{\mathrm{REC}}_{z,l-1,j,n}].
\end{equation}
We define $y^{\mathrm{REC}}_{z,0,j,n} = 0$, and take $\tilde{y}^{\mathrm{REC}}_{z,l-1,j,n} = 0$ if $n \in \mathcal{N}^{\mathrm{REC}}_l \setminus \mathcal{N}^{\mathrm{REC}}_{l-1}$.  Note that the expansion cost is incurred whenever the capacity of a facility is increased, regardless of the capacity in the previous period.  

We reformulate this problem with a similar strategy as~\eqref{eq:reformulatedP}: integer variables $y^{\mathrm{REC}}_{z,l,j}$ count the number of facilities with capacities at their upper bound, and continuous variables $y^{\mathrm{REC}}_{z,l,l',j,+}$ give the capacity of the facilities not at their upper or lower bounds.  There may be up to $L$ of these facilities (one from each planning period); variable $y^{\mathrm{REC}}_{z,l,l',j,+}$ models the capacity in period $l$ of the $l'$-th facility with continuous capacity.  For convenience, we assume that $|\mathcal{N}^{\mathrm{REC}}_l| \geq L$. We also introduce binary variables $y^{\mathrm{REC,EXP}}_{z,l,l',j}$ that indicate whether the $l'$-th facility is expanded in period $l$.  In the reformulation, the total recycling capacity of process $j$ in zone $z$ and period $l$ is
\begin{equation}
    \label{eq:FCTotalCap}
    u^{\mathrm{REC}} y^{\mathrm{REC}}_{z,l,j} + \sum_{l' \in \mathcal{L}} y^{\mathrm{REC}}_{z,l,l',j,+};
\end{equation}
the throughput constraint~\eqref{eq:constrCapREC} and monotonicity constraint~\eqref{eq:constrCapRECIncrease} can be written by replacing the total capacity $\sum_{n \in \mathcal{N}^{\mathrm{REC}}_l} y^{\mathrm{REC}}_{z,l,j,n}$ with~\eqref{eq:FCTotalCap}.
These variables are additionally constrained for all $l \in \mathcal{L}$ and $z \in \mathcal{Z}$ as follows:
\begin{subequations}
    \label{eq:FCreformulation}
    \begin{align}
        & y^{\text{REC}}_{z,l,j} \in \{0,\dots,|\mathcal{N}^{\text{REC}}_l|-L\} \quad & \forall j \in \mathcal{J},\label{eq:FCreformulationIntegerDom}\\
        & y^{\mathrm{REC,EXP}}_{z,l,l',j} \in \{0,1\} \quad & \forall j \in \mathcal{J},\ l' \in \mathcal{L},\label{eq:FCreformulationIndDom}\\
        & y^{\text{REC}}_{z,l,l',j,+} \leq y^{\text{REC}}_{z,l-1,l',j,+} + u^{\mathrm{REC}}y^{\mathrm{REC,EXP}}_{z,l,l',j} \quad & \forall j \in \mathcal{J},\ l' \in \mathcal{L},\label{eq:FCreformulationIndUB}\\
        & 0 \leq y^{\text{REC}}_{z,l,l',j,+} \leq  u^{\mathrm{REC}} \quad & \forall j \in \mathcal{J},\ l' \in \mathcal{L}.\label{eq:FCreformulationContinuousDom}
    \end{align}
\end{subequations}
We take $y^{\mathrm{REC}}_{z,0,l',j,+} = 0$.  Constraints~\eqref{eq:FCreformulationIntegerDom}~and~\eqref{eq:FCreformulationIndDom} enforce integrality, constraint~\eqref{eq:FCreformulationIndUB} defines the capacity expansion indicator, which must be $1$ if the facility capacity increases between planning periods, and constraint~\eqref{eq:FCreformulationContinuousDom} enforces bounds on the continuous capacities.  The reformulated facility cost contribution in period $t$ is
\begin{equation}
    \sum_{z \in \mathcal{Z}} \sum_{j \in \mathcal{J}} \left (  y^{\mathrm{REC}}_{z,l_t,j} f^{\mathrm{REC}}_{z,j}(u^{\mathrm{REC}}) + \sum_{l' \in \mathcal{L}} f^{\mathrm{REC}}_{z,j}(y^{\mathrm{REC}}_{z,l_t,l',j,+}) \right ),
\end{equation}
and the facility expansion costs~\eqref{eq:expansionCostDef} are
\begin{equation}
    \sum_{z \in \mathcal{Z}} \sum_{j \in \mathcal{J}} \sum_{l \in \mathcal{L}}f^{\mathrm{EXP}}_{z,l,j} \left [ y^{\mathrm{REC}}_{z,l,j} - y^{\mathrm{REC}}_{z,l-1,j}\right ]^+ + \sum_{l' \in \mathcal{L}} f^{\mathrm{EXP}}_{z,l,j} y^{\mathrm{REC,EXP}}_{z,l,l',j},
\end{equation}
where $y^{\mathrm{REC}}_{z,0,j} = 0$.  The maximum $[\cdot]^+ := \max \{\cdot,0\}$ is representable with linear constraints under an epigraph reformulation.  The cathode production capacity variables admit a similar representation.  This reformulation introduces $L^2$ binary variables and $L^2$ continuous variables with concave cost functions for each zone and recycling process, in addition to the integer variables already introduced in~\eqref{eq:reformulatedP}.  Again, the reformulated capacity variables are no longer indexed by $\mathcal{N}^{\mathrm{REC}}$ and $\mathcal{N}^{\mathrm{CP}}$.

\subsubsection{Open-Loop Models}\label{subsec:extensionOpenLoop} 
Model~\eqref{eq:deterministicP} is of a closed-loop supply chain, where a battery manufacturer invests in recycling to recover materials that are used to manufacture new batteries.  Alternatively, the supply chain may be open-loop, where a recycling company invests in capacity and sells the recovered materials to battery manufacturers and other parties through a market.  Under a mild assumption, the optimal recycling investment from our closed-loop model remains optimal in the open-loop setting.  

If the nonnegativity constraint on $x^{\text{NM,NB}}$ (the amount of new material used in new battery production) is non-binding, $x^{\text{NM,NB}}$ can be replaced with its definition from constraint \eqref{eq:constrNBProdCP}.  This assumption holds if the material requirements for manufacturing new batteries cannot be met through recycling alone.  Then, the objective contribution from $x^{\text{NM,NB}}$ becomes
\begin{equation}
    \sum_{z \in \mathcal{Z},\,t \in \mathcal{T},\,k \in \mathcal{K}^{\text{CP}},\,\omega \in \Omega_{\sigma_t}} c^{\text{NM,NB}}_{\omega,t,k} x^{\text{NM,NB}}_{z,\omega,t,k} = \sum_{z \in \mathcal{Z},\,t \in \mathcal{T},\,k \in \mathcal{K}^{\text{CP}},\,\omega \in \Omega_{\sigma_t}} c^{\text{NM,NB}}_{\omega,t,k} \left ( \sum_{i \in \mathcal{I}} \Delta^{\text{NB}}_{i,k} d_{\omega,z,t,i} - x^{\text{INV,NB}}_{z,\omega,t,k} \right ).
\end{equation}
This change effectively removes new battery manufacturing from the model and adds a revenue associated with ``selling'' recycled cathode powder at the market price, which represents the dynamics of an open-loop recycler.  In the remainder of this work, we focus on the closed-loop recycling model, although our approach, optimal solutions, and discussion generalize to an open-loop setting.

\subsubsection{General Recycling Applications}
\label{subsec:generalExtension}
The style of model~\eqref{eq:deterministicP} is relevant for recycling and remanufacturing supply chains across industries.  Multi-period decision-making, economies of scale in capacity investments, diverse recycling technologies, and uncertainty in raw material prices, supply of recycled materials, and demand for new products are common features in many applications.  Although we introduce and study this model for EV battery recycling supply chains, our discussion, reformulation, and algorithmic approaches answer pertinent questions for investment planning across recycling industries, including recycling of metals, plastics, electronics, and renewable energy assets (i.e., solar photovoltaics and windmills).

%%Algorithm
\section{Algorithmic Approaches}
\label{sec:algorithm}
Although reformulation decreases the number of first-stage variables, problems~\eqref{eq:deterministicP}~and~\eqref{eq:reformulatedP} may be large due to the inclusion of many zones $\mathcal{Z}$ or scenarios $\Omega$.  Efficient solution algorithms for these problems are vital for two main reasons.  First, efficient algorithms support solution of highly detailed problems with large variable spaces.  Modeling additional zones and recycling technologies improves realism and allows for high-fidelity, granular decision making, and adding additional scenarios permits greater coverage of the set of possible cost and demand outcomes, improving the trustworthiness of the model.  Second, users may need to solve variants of the problem quickly in succession.  Solving a diverse set of models is vital for recyclers who seek to test what-if scenarios and alternate assumptions, or for policymakers who look to understand the industry response to a variety of policies, as explored in Section~\ref{subsec:policy}. 

To these ends, we require efficient algorithms for the global minimization of separable concave functions over mixed-integer linear feasible sets.  Leading approaches for this class are spatial branch and bound algorithms (sB\&B, e.g., \citealt{shectman1998finite}), which work by partitioning the feasible region and employing linear underapproximation of the objective functions.  BARON is a commercial global solver for mixed-integer nonconvex problems that implements sB\&B, aided by many modern techniques for improving performance \citep{zhang2025solving}.  Spatial branch and bound algorithms are commonly built to solve problems with concave objective functions \citep[e.g.,][]{malik2022two}, and BARON is often used to find global solutions to nonconvex problems in modern applications \citep{allman2017framework,ejeh2023minlp}.  However, sB\&B and BARON are not able to solve our model efficiently and reliably at scale (see Section~\ref{sec:computation}).

In Section~\ref{subsec:apwl}, we describe our algorithm for this class of problems.  The algorithm employs piecewise linear approximations of the objective functions and solves a series of MILPs, where the piecewise approximation is updated at the incumbent optimal solution; this style of iterative algorithm has appeared in the literature on mixed-integer nonlinear optimization \citep{dambrosio2009global, you2011stochastic}.  Compared to sB\&B, this algorithm maintains a global perspective of the problem and only introduces breakpoints to the objective function at incumbent global solutions.  These global solutions are optima of MILPs, which can be solved efficiently by leveraging commercial linear solvers that are generally more mature than mixed-integer nonlinear solvers.  We prove a novel result on the finite convergence of this algorithm to a global optimum.  
In Section~\ref{sec:ssDecomposition}, we describe a Benders' decomposition and grouping strategy that reduces the number of cutting planes that are required.  This cut grouping improves computational performance and facilitates the inclusion of many scenarios at the operational scale.  We demonstrate the efficiency of these algorithmic approaches and compare against sB\&B, BARON, and typical cut aggregations in Section~\ref{sec:computation}.

\subsection{Adaptive Piecewise Linear Approximation Algorithm}
\label{subsec:apwl}
We propose a finitely convergent algorithm to find globally optimal solutions for separable concave minimization problems with linear constraints and mixed-integer variables.  The algorithm under-approximates each univariate concave objective function with a piecewise linear function, optimizes over the approximation, then updates the functions by introducing breakpoints at the incumbent solution, making the approximation exact at the current iterate.  The subproblem objectives provide lower bounds on the global optimum, and revisiting an iterate yields an optimality certificate. 

In this section, we define new notation which is distinct from the previous section.  For analysis of our algorithm, we introduce a stylized version of the models \eqref{eq:deterministicP} and \eqref{eq:reformulatedP}:
\begin{equation}
    \label{eq:SCP}
    \tag{SCP}
    \begin{aligned}
        \min_{(y,x) \in \mathcal{X}} \quad & \sum_{i = 1}^{n_y} f_i(y_i) + \sum_{\omega \in \Omega_{1}} p_\omega c_\omega^T x_\omega,
    \end{aligned}
\end{equation}
where the feasible region is 
\begin{equation}
    \mathcal{X} = \{(y,\{x_\omega\}_{\omega \in \Omega_1}) \geq 0 \ :\ y \in \mathbb{Z}^{n_{\text{I}}} \times \R^{n_y - n_{\text{I}}},\ A y = b,\ B_\omega x_\omega + D y = d_\omega\ \forall \omega \in \Omega_{1}\}.
\end{equation}  
The functions $f_i\ :\ \R \rightarrow \R$ are concave lower semicontinuous, $p_\omega \in \R$, $c_\omega \in \R^{n_x}$, $A \in \R^{m_1 \times n_y}$, $b \in \R^{m_1}$, $B_\omega \in \R^{m_2 \times n_x}$, $D \in \R^{m_2 \times n_y}$ and $d_\omega \in \R^{m_2}$.  The number of integer capacity variables is given by $n_{\text{I}}$.
This notation emphasizes the scenario-based structure in the operational scale with blocks for each first-stage node, but generalizes the structure of descendant nodes by modeling them in the extensive form.  That is, for $\omega \in \Omega_1$, the parameters $(c_\omega,B_\omega,D,d_\omega)$ contain objective and constraint data for all descendant scenarios of node $\omega$.
We assume $\mathcal{X}$ is nonempty and bounded, so any feasible $y$ satisfies some consistent bounds $\underline{y} \leq y \leq \overline{y}$.  This assumption is reasonable as the models have finite optimal solutions by Proposition~\ref{prop:finiteP}.  We further assume a relatively complete recourse property, which holds for our models by Proposition~\ref{prop:fullRecourse}:
\begin{equation}
    \label{eq:fullRecourseProp}
    \{x \geq 0\ :\ B_\omega x = d_\omega - D y\} \neq \emptyset \quad \forall \omega \in \Omega_1,\ y \in \mathbb{Z}^{n_{\text{I}}} \times \R^{n_y - n_{\text{I}}}\ :\ y \geq 0,\ Ay = b.
\end{equation}

We now detail our algorithm.  Given a set of $k_i$ breakpoints $\{y_i^{(j)}\}_{j = 1}^{k_i}$ for variable $y_i$ ordered so that $y_i^{(j)} < y_i^{(j+1)}$, we construct piecewise linear functions $\overline{f}_i\ :\ [y_i^{(1)},y_i^{(k_i)}] \rightarrow \R$ by 
\begin{equation}
    \label{eq:underapproximatorConstruction}
    \overline{f}_i(y_i) = {\Bigg \{ } \frac{f_i(y_i^{(j+1)}) - f_i(y_i^{(j)})}{y_i^{(j+1)} - y_i^{(j)}} (y_i - y_i^{(j)}) + f_i(y_i^{(j)}) \quad y_i \in [y_i^{(j)}, y_i^{(j+1)} ) \quad \forall j \in \lBrack k_i-1 \rBrack.
\end{equation}
We build a subproblem \eqref{eq:SCPSP} that replaces the objective of \eqref{eq:SCP} with the piecewise functions $\overline{f}_i$:
\begin{equation}
    \label{eq:SCPSP}
    \tag{SP}
    \begin{aligned}
        \min_{(y,x) \in \mathcal{X}} \quad & \sum_{i = 1}^n \overline{f}_i(y_i) + \sum_{\omega \in \Omega_{1}} p_\omega c_\omega^T x_\omega.\\
    \end{aligned}
\end{equation}
As $\overline{f}_i$ are underestimators of $f_i$, the subproblems are relaxations of \eqref{eq:SCP}.  The piecewise functions can be modeled with SOS-2 constraints or mixed-integer variables.

Denote the (polyhedral) set formed by fixing the discrete variables of $\mathcal{X}$ to some solution $\overline{y}$ by $\mathcal{X}(\overline{y}) = \{(y,x) \in \mathcal{X}\ :\ y_i = \overline{y}_i\ \forall i \in \lBrack n_{\text{I}} \rBrack \}$.  Lemma~\ref{lemma:extPointOptDiscreteFixed} shows that there is an optimal solution for \eqref{eq:SCPSP} that falls at an extreme point of $\mathcal{X}(y^*)$ for some $y^*$.  

\begin{lemma}
    \label{lemma:extPointOptDiscreteFixed}
    Let $(y^*,x^*) \in \mathcal{X}$ be an optimal solution to \eqref{eq:SCPSP}.  Then, there is an optimal solution that is an extreme point of $\mathcal{X}(y^*)$.
\end{lemma}

\begin{assumption}
    \label{assump:extPoint}
    Solving \eqref{eq:SCPSP} yields an optimal solution $(y^*,x^*)$ that is an extreme point of $\mathcal{X}(y^*)$.
\end{assumption}

We propose the adaptive piecewise linear approximation algorithm (aPWL), Algorithm~\ref{alg:PWL}, that solves a sequence of subproblems \eqref{eq:SCPSP} and updates the breakpoints defining the piecewise functions at every iteration.  As the subproblems minimize piecewise linear functions subject to linear and integrality constraints, they can be solved to global optimality via an LP-based branch and bound solver.  Assumption~\ref{assump:extPoint} establishes that solving \eqref{eq:SCPSP} in step 3 yields one of the optimal extreme points described in Lemma~\ref{lemma:extPointOptDiscreteFixed}.  A branch and bound solver can be tailored to generate optimal extreme points of the restricted feasible region, satisfying this assumption.  Under Assumption~\ref{assump:extPoint}, the algorithm converges finitely to a global optimum of \eqref{eq:SCP}, shown in Theorem~\ref{thm:finiteConvergence}.

\begin{algorithm}[t]
\caption{Adaptive Piecewise Linear Approximation}\label{alg:PWL}
Initialize $(y^{(1)}_i,y^{(2)}_i) \gets (\underline{y}_i, \overline{y}_i) \text{ and } k_i \gets 2 \quad \forall i \in \lBrack n_y \rBrack$.  Set $s \gets 1$.\\
For all $i \in \lBrack n_y \rBrack$, construct $\overline{f}_i^s$ by \eqref{eq:underapproximatorConstruction} with breakpoints $\{y_i^{(j)}\}_{j = 1}^{k_i}$.\\
Solve ${\displaystyle \min_{(y,x) \in \mathcal{X}} \quad \sum_{i = 1}^n \overline{f}_i^s(y_i) + \sum_{\omega \in \Omega_1} p_\omega c_\omega^T x_\omega}$ to optimal solution $(\tilde{y}^{s},\{\tilde{x}^{s}_\omega\}_{\omega \in \Omega_1})$.\\
If $\overline{f}_i^s (\tilde{y}_i^{s}) \geq f_i(\tilde{y}_i^{s})$ for all $i$, STOP, $(\tilde{y}^{s},\{\tilde{x}^{s}_\omega\}_{\omega \in \Omega_1})$ is optimal for \eqref{eq:SCP}.\\
Otherwise, for each $i$ where $\overline{f}_i^s (\tilde{y}_i^{s}) < f_i(\tilde{y}_i^{s})$, add breakpoint $y_i^{(k_i + 1)} \gets \tilde{y}_i^{s}$.  Set $k_i \gets k_i + 1$.  Reindex $\{y_i^{(j)}\}_{j = 1}^{k_i}$ so that $y_i^{(j)} < y_i^{(j+1)}$ for all $i \in \lBrack n_y \rBrack$ and $j \in \lBrack k_i - 1 \rBrack$.\\
Set $s \gets s + 1$.  Go to 2.\\
\end{algorithm}

\begin{theorem}
    \label{thm:finiteConvergence}
    Algorithm~\ref{alg:PWL} terminates finitely with a global optimum of \eqref{eq:SCP}.
\end{theorem}

\subsection{Scale Decomposition}
\label{sec:ssDecomposition}

We further leverage the multistage stochastic structure of the model by employing a Benders' decomposition scheme \citep{benders1962partitioning} in step 3 of Algorithm~\ref{alg:PWL} that separates decisions at the planning scale from those at the operational scale.  In the context of stochastic optimization, this decomposition is also referred to as the L-shaped method \citep{van1969shaped}.  The subproblems \eqref{eq:SCPSP} are solved iteratively by generating cuts from the operational scale.  The operational problem is separable over the first-stage nodes $\Omega_1$ and thus can be solved in parallel as linear programs.  We do not further decompose over nodes at later stages in the multistage scenario tree; experiments on this style of decomposition (i.e., the nested L-shaped method, see \citealt[Ch.~6.1]{birge2011introduction}) do not reduce computation times for the size of problem considered in this work.  Importantly, every scenario has a relatively complete recourse property \eqref{eq:fullRecourseProp}, so the decomposition scheme does not need to add cuts that enforce feasibility.

For each operational problem indexed by $\omega \in \Omega_1$, we denote the extreme points of the dual feasible region $\{\pi\ :\ p_\omega c_\omega \geq B^T_\omega \pi\}$ by $\pi^{(\omega)}_j$ for $j \in \mathcal{J}_\omega$, where $\mathcal{J}_\omega$ is an index set for the dual extreme points of problem $\omega$.  Under property \eqref{eq:fullRecourseProp}, every operational problem has a finite optimal solution for any capacity decision $y$.  Thus, each dual problem has a finite optimum at some extreme point $\pi^{(\omega)}_j$.  The dual problems can then be equivalently represented by the convex piecewise linear value functions
\begin{equation}
    \label{eq:dualPWL}
    v_\omega (y) = \max_{j \in \mathcal{J}_\omega}\ (d_\omega - D y)^T \pi^{(\omega)}_j.
\end{equation}
The Benders' decomposition approximates $v_\omega$ by adding valid cuts at incumbent capacity solutions $y$.  Importantly, added cuts remain valid between iterations of Algorithm~\ref{alg:PWL}, allowing warmstarting of the decomposition between iterations and decreasing the number of new cuts required to converge.

For the sake of computational efficiency, we only partially decompose the operational problems.  We construct groups of first-stage nodes and apply the Benders' cut generation approach to each group.  This approach is similar to that employed by \cite{adulyasak2015benders}.  Let the groups $\{\mathcal{G}_e\}_{e = 1}^{n_g}$ be a partition of the nodes $\Omega_1$.  Then, we generate cuts for the grouped value functions
\begin{equation}
    \label{eq:dualPWLGroup}
    \max_{\{j_\omega\}_{\omega \in \mathcal{G}_e} \in \prod_{\omega \in \mathcal{G}_e} \mathcal{J}_\omega}\ \sum_{\omega \in \mathcal{G}_e} (d_\omega - D y)^T \pi^{(\omega)}_{j_\omega},
\end{equation}
instead of the individual functions \eqref{eq:dualPWL}. Grouping in such a way reduces the number of convex functions to model but increases the number of linear pieces in each function.  We describe how the groups $\{\mathcal{G}_e\}_{e = 1}^{n_g}$ are chosen in Section~\ref{sec:scenarioTree}.

%%Results
\section{Data Sources and Scenario Generation}
\label{sec:data}

In this section, we detail how the parameters of the optimization model are populated.  Section~\ref{subsec:deterministicData} describes the deterministic data, including index sets and capacity cost functions, which are primarily derived from EverBatt \citep{dai2019everbatt}.  Sections~\ref{subsec:stockScenarios}~and~\ref{subsec:costScenarios} describe a multistage scenario generation process for battery supply and demand and material cost, respectively.  The scenarios model the correlation between EV adoption and critical mineral prices by leveraging projections from the International Energy Agency.  This approach can easily be extended to generate correlated demand and cost scenarios for other energy transition technologies, including solar photovoltaics and wind turbines.

\subsection{Deterministic Data}
\label{subsec:deterministicData}

We model the recycling and manufacture of six battery chemistries: NMC(111), NMC(532), NMC(622), NMC(811), NCA, and LFP.  NMC are lithium nickel manganese cobalt oxides, NCA is a lithium nickel cobalt aluminum oxide, and LFP is lithium iron phosphate.  Batteries can be recycled through pyrometallurgical, hydrometallurgical, or direct  processes.  The time horizon covers 2021 to 2050 with annual granularity, and planning periods are the six 5-year intervals: 2021-2025, 2026-2030, 2031-2035, 2036-2040, 2041-2045, and 2046-2050.  The stages of the multistage scenario tree are the three 10-year intervals: 2021-2030, 2031-2040, and 2041-2050.

Deterministic parameters are computed using the extensive low-level formulas and high-fidelity data from the EverBatt model \citep{dai2019everbatt}.  The parameters $\Delta$, $u$, $f$, $c^{\text{TR,RM}}$, and $c^{\text{TR,RB}}$ are derived from EverBatt.  We take a discount rate $\gamma = 3\%$ to match the Federal Energy Management Program discount rate for 2021 \citep{kneifel2020FEMP}.  Inventory costs are set as $\rho = 25\%$ of material value \citep{azzi2014inventory}, and we assume that $\eta = 70\%$ of material value is recovered when selling recycled materials to the market.\textsuperscript{\endnote{\cite{pourmohammadi2008sustainable} estimate $\eta$ as $60\%$ for recycled aluminum.  Recycled carbon black is similarly discounted by $\eta \in [70\%,85\%]$ relative to its virgin counterpart \citep{contec2023carbonblack}.}}  

The numbers of recycling facilities and cathode production lines are determined by the number of facilities and lines needed to recycle the largest annual retired battery supply and to manufacture cathode powder to satisfy the largest annual new battery demand during each planning period:
\begin{equation*}
    |\mathcal{N}^{\text{REC}}_l| = \left \lceil \frac{1}{u^{\text{REC}}} \max_{t \in \mathcal{T}_l}\ \max_{\omega \in \Omega_{\sigma_t}}\ \sum_{i \in \mathcal{I},\ z \in \mathcal{Z}} s_{\omega,z,t,i}\right \rceil \quad \text{and} \quad |\mathcal{N}^{\text{CP}}_{l,k}| = \left \lceil \frac{1}{u^{\text{CP}}} \max_{t \in \mathcal{T}_l}\ \max_{\omega \in \Omega_{\sigma_t}}\ \sum_{i \in \mathcal{I},\ z \in \mathcal{Z}}  \Delta^{\text{NB}}_{i,k} d_{\omega,z,t,i} \right \rceil,
\end{equation*}
where the maximum capacities from EverBatt are $u^{\text{REC}} = 100,000$ tonnes/year and $u^{\text{CP}} = 2,000$ tonnes/year \citep{dai2019everbatt}.

\subsection{Battery Supply and Demand Scenarios}
\label{subsec:stockScenarios}
We follow the methodology and data from \cite{xu2020future} for constructing deterministic global EV stock projections under the International Energy Agency's Sustainable Development Scenario (SDS) and Stated Policies Scenario (STEPS) \citep{iea2020scenarios}.  The projections give the total number of active EVs (EV \textit{stock}), including battery and plug-in hybrid EVs, in each year over the 30-year period from 2021 to 2050.  We label these two projections $EV_t^{\text{SDS}}$ and $EV_t^{\text{STEPS}}$.  Time periods are standardized, so $t = 1$ corresponds to 2021 and $T = 30$ to 2050.  The parameter $EV_0 = EV_0^{\text{SDS}} = EV_0^{\text{STEPS}}$ gives the measured number of active EVs in 2020.

We denote the year-over-year rate of change in EV stock under the two projections by $\delta^{\text{EV,SDS}}_t$ and $\delta^{\text{EV,STEPS}}_t$, e.g., $\delta^{\text{EV,SDS}}_t = EV^{\text{SDS}}_t / EV^{\text{SDS}}_{t-1}$.  Let $\delta^{\text{EV}}_t$ be the random variable that represents the rate of change in EV stock for year $t$.  We define the distribution of $\delta^{\text{EV}}_t$ by
\begin{equation}
    \delta^{\text{EV}}_t = \frac{1}{2} \left | \delta_t^{\text{EV,SDS}} - \delta_t^{\text{EV,STEPS}} \right | Z^{\text{EV}}_{\sigma_t} + \frac{1}{2} \left ( \delta_t^{\text{EV,SDS}} + \delta_t^{\text{EV,STEPS}} \right ),
\end{equation}
where the randomness is described by a random variable $Z^{\text{EV}}_\sigma$ for each stage $\sigma \in \mathcal{S}$.  We assume that these random variables are distributed as truncated standard normals on the interval $[-2,2]$ and are independent across stages.  

We generate a discrete distribution by approximating $Z^{\text{EV}}_\sigma$ with $n_d$ values via Gaussian quadrature \citep{miller1983discrete}.  For each index $\xi_{\sigma} \in \lBrack n_d \rBrack$, Gaussian quadrature produces an observation of $Z^{\text{EV}}_{\sigma}$, labeled $\overline{Z}^{\text{EV}}_{\xi_\sigma}$, and a corresponding probability $p^{\text{EV}}_{\xi_\sigma}$.  Gaussian quadrature guarantees that the first $2n_d - 1$ moments of the discrete distributions match those of the underlying truncated normal distributions.  We convert the observations $\overline{Z}^{\text{EV}}_{\xi_\sigma}$ into a set of multistage demand scenarios by considering combinations of observations across stages, i.e., taking the cartesian product over $S$ copies of the observations.  The multistage index is some $\xi \in \lBrack n_d \rBrack ^S$, and the corresponding time series for the observations of the rate of change in EV stock, denoted $\{\delta^{\text{EV}}_{\xi,t}\}_{t \in \mathcal{T}}$, are computed by
\begin{equation}
    \delta^{\text{EV}}_{\xi,t} = \frac{1}{2} \left | \delta_t^{\text{EV,SDS}} - \delta_t^{\text{EV,STEPS}} \right | \overline{Z}^{\text{EV}}_{\xi_{\sigma_t}} + \frac{1}{2} \left ( \delta_t^{\text{EV,SDS}} + \delta_t^{\text{EV,STEPS}} \right ).
\end{equation}
Then, the EV stock observations for the demand scenario, labeled $EV_{\xi,t}$ are given by 
\begin{equation}
    EV_{\xi,t} = EV_0 \prod_{t' = 1}^t \delta^{\text{EV}}_{\xi,t'}.
\end{equation}
This approach relies on the independence of $Z^{\text{EV}}_\sigma$ across stages, which implies the independence of the rate of change $\delta^{\text{EV}}_t$ in different stages.  However, the observations of EV stock are not independent across stages, as they depend on the history of the random variable $\delta^{\text{EV}}_t$.  Each observation $EV_{\xi,t}$ only depends on the realization of $\overline{Z}^{\text{EV}}_{\sigma}$ at the current or prior stages, which is consistent with how information is revealed over stages of a scenario tree.

We next separate the number of active EV batteries by age.  Using data from \cite{xu2020future}, we designate $LS_a$ as the percentage of batteries of age $a$ that retire at their current age, with a maximum lifespan of $A$.  Denote by $AB_{\xi,t,a}$ the number of active batteries in year $t$, demand scenario $\xi$ that are $a$ years old.  The parameter $AB_{\xi,t,a}$ is computed recursively for all $\xi \in \lBrack n_d \rBrack^S$: 
\begin{equation}
    AB_{\xi,t,a} = \begin{cases}
        (1 - LS_a) AB_{\xi,t-1,a-1} & a > 0\\
        \left [ EV_{\xi,t} - EV_{\xi,t-1} + {\displaystyle\sum_{a' = 1}^{A}} LS_{a'} AB_{\xi,t-1,a'-1} \right ] ^+ & a = 0
    \end{cases} \quad \forall t \in \mathcal{T},\ a \in \{0,\dots,A\},
\end{equation}
where $[\cdot]^+ = \max\{\cdot,0\}$.  By this definition, the number of new ($a = 0$) batteries is the sum of the year-over-year increase in active batteries and the number of retired batteries, and older ($a>0$) batteries are reduced by the retirement proportion $LS_a$ as they age.  The initial state $AB_{\xi,0,a}$ is taken from historical data \citep{iea2020scenarios,iea2023explorer}.

We convert from battery counts to mass and split by cathode chemistry, distinguishing between two projections for market shares of battery chemistries.  \cite{xu2020future} also provide these projections: in the first (the NCX projection), NMC chemistries are more prominent, and in the second (the LFP projection), LFP chemistries are more prominent.  We label these market share projections $MS_{\text{NCX},t,i}$ and $MS_{\text{LFP},t,i}$, representing the proportion of new battery sales that are of chemistry $i$ by year and projection.  These parameters include historical data through year $-A$. 
We now index by market share projection and EV stock observation: $(\psi,\xi) \in \{\text{NCX},\text{LFP}\} \times \lBrack n_d \rBrack^S$.  
To account for changes in average battery size due to differences in battery EV (BEV) and plug-in hybrid (PHEV) adoption, we compute an additional parameter $m_{\xi,t,i}$, representing the average mass of a new battery by chemistry.  The computation of $m_{\xi,t,i}$ is described in Section~\ref{sec:ecCCParameter} of the appendix.  

Supply and demand for new batteries by chemistry is then given, respectively, by
$$s_{(\psi,\xi),z,t,i} = \beta_z m_{\xi,t,i} \sum_{a = 1}^{A} MS_{\psi,t-a,i} LS_a AB_{\xi,t-1,a-1} \quad \text{and} \quad d_{(\psi,\xi),z,t,i} = \beta_z m_{\xi,t,i} MS_{\psi,t,i} AB_{\xi,t,0},$$
where $\beta_z \in (0,1]$ is a scale parameter that gives the proportion of the global EV market allocated to each zone $z$.  Figure~\ref{fig:sdScenarios} shows a set of supply and demand scenarios generated with this method.

%%%%%%%%%%%%%%%%%% Figure %%%%%%%%%%%%%%%%%%
\begin{figure}[tp]
\begin{center}
    {\begin{subfigure}[t]{.48\textwidth} \centering
    \includegraphics[width=\linewidth]{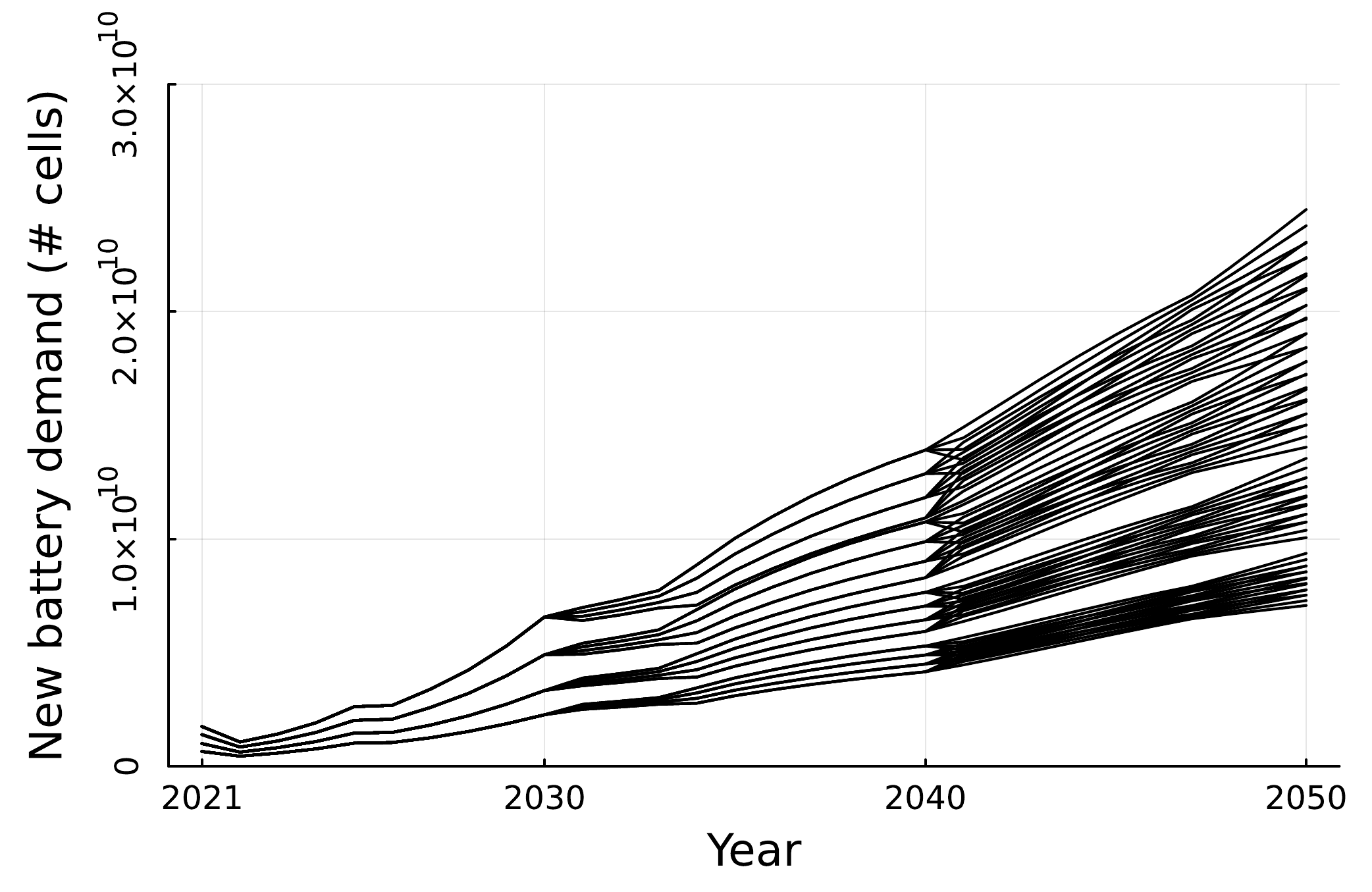}
    \captionsetup{font=footnotesize}
    \caption{}
    \end{subfigure}
    \hspace{0.2cm}
    \begin{subfigure}[t]{.48\textwidth} \centering
    \includegraphics[width=\linewidth]{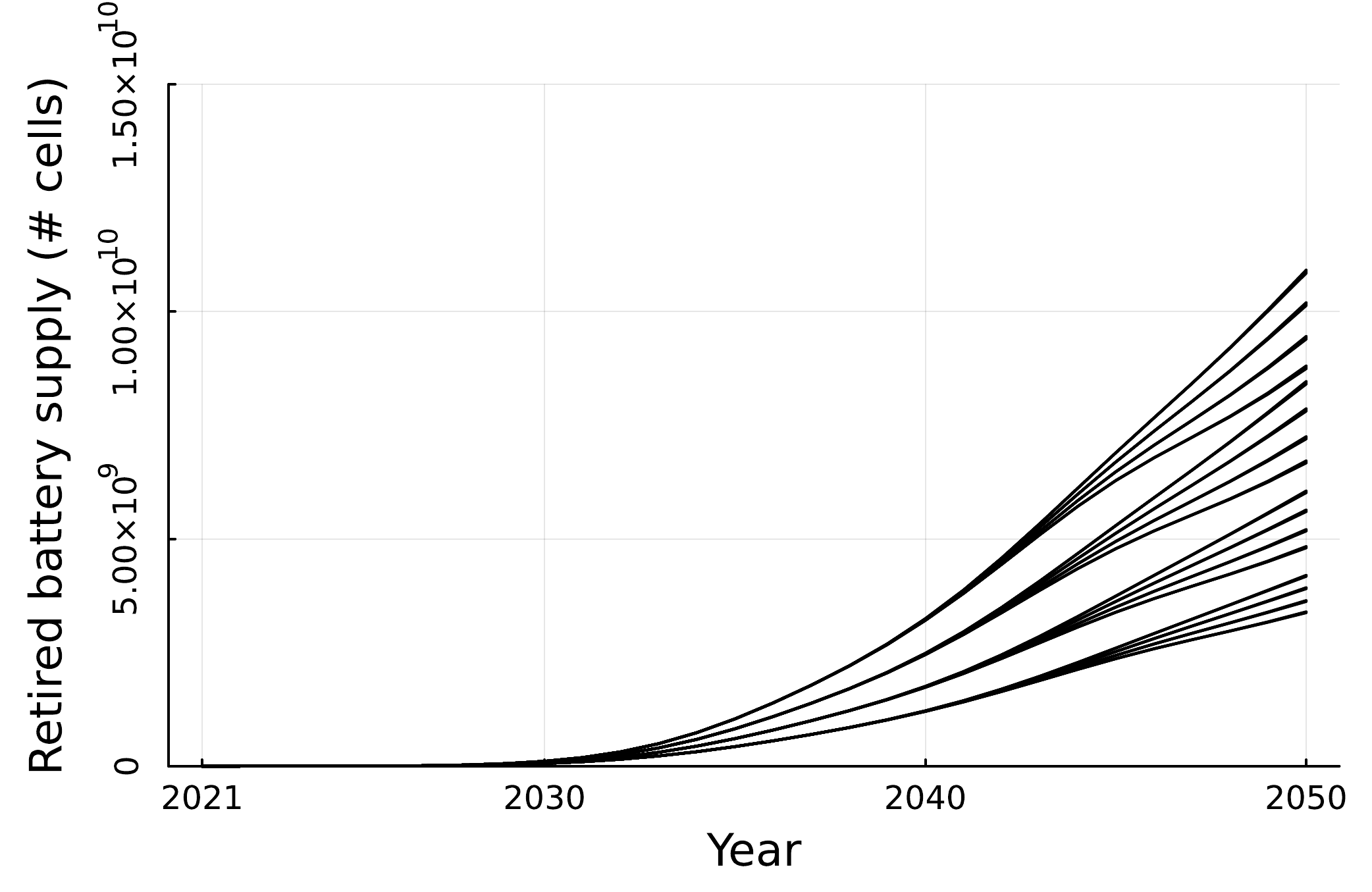}
    \captionsetup{font=footnotesize}
    \caption{}
    \end{subfigure}}
\caption{A set of three-stage ($S = 3$) new battery demand (a) and recycled battery supply (b) scenarios for chemistry NMC(622) with $n_d = 4$ and $\beta = 1$.} \label{fig:sdScenarios}
\end{center}
\end{figure}
%%%%%%%%%%%%%%%%%% Figure %%%%%%%%%%%%%%%%%%

\subsection{Cost Scenarios}
\label{subsec:costScenarios}
We assume that changes in cost coefficients are due to fluctuations in the value of \emph{critical minerals}: copper (Cu), cobalt (Co), lithium (Li), manganese (Mn), and nickel (Ni).  Let $\mathcal{M} = \{\text{Cu}, \text{Co}, \text{Li}, \text{Mn}, \text{Ni}\}$ be this set of metals.  From this assumption, only material costs and values vary by scenario and time; all other costs are static at the values from EverBatt.

To account for correlations between EV demand and metal prices, we construct $n_c$ cost scenarios for each demand scenario $(\psi,\xi) \in \{\text{NCX},\text{LFP}\} \times \lBrack n_d \rBrack^S$.  We leverage metal price projections from \cite{pescatori2021energy}, who construct forecasts in line with the assumptions of the SDS and STEPS projections.  For each metal, this work fits a vector autoregressive model to predict metal production and metal price.  The regression depends on three randomly distributed exogenous shocks to the metal market, namely, an aggregate demand shock, a metal-specific demand shock, and a metal supply shock.  Following the approach of \cite{antolin2021structural}, a conditional distribution of the future metal price time series is estimated by a sampling algorithm, where the distribution is conditioned so that predicted metal production matches the projected values from SDS or STEPS and the aggregate demand and metal supply shocks follow their unconditional distributions.  Effectively, this conditional distribution is of predicted metal prices that result from a series of metal demand shocks that generate the projected metal production levels from SDS or STEPS.  From these results, we take for each metal an annual median price estimate and a 40\% highest posterior density (HPD) region under each of the SDS and STEPS projections from 2020 to 2040 \citep[Fig. 4]{pescatori2021energy}.\textsuperscript{\endnote{\cite{pescatori2021energy} do not provide projections for manganese prices.  However, we find that manganese and copper prices are highly correlated.  Price data from \cite{imf2023pcps} between 2013 and 2021 show that manganese and copper prices have a correlation coefficient of $0.91$.  As validation, a correlation of prices between 1950 and 2010 from \cite{usgs2013metalprice} yields a coefficient of $0.85$.  Given this high correlation, we predict the median value and HPD bounds for manganese from the corresponding values for copper using a linear regression fitted on data from \cite{imf2023pcps}.}}  The data is extended to 2050 by setting the values for years 2041-2050 to the value from 2040.  Let $v^{\text{MED,SDS}}_{t,m}$ and $v^{\text{MED,STEPS}}_{t,m}$ be the median estimates for each projection, and $[v^{\text{LB,SDS}}_{t,m},v^{\text{UB,SDS}}_{t,m}]$ and $[v^{\text{LB,STEPS}}_{t,m},v^{\text{UB,STEPS}}_{t,m}]$ be the HPD intervals for each metal $m$.

As demand and cost data are associated with the same underlying International Energy Agency projections (i.e., SDS and STEPS), we use the relationship between demand scenarios and the projections to produce a set of scenario-adjusted cost distribution parameters.  First, we characterize each demand scenario by a multiplier
\begin{equation}
    \lambda_{\xi,t} = \underset{[0,1]}{\proj} \left (\frac{ EV_{\xi,t} - EV^{\text{STEPS}}_t}{EV^{\text{SDS}}_t - EV^{\text{STEPS}}_t} \right ).
\end{equation}
This multiplier characterizes in each year how similar the EV stock realization is to the SDS and STEPS projections, where values near $1$ imply similarity to SDS and values near $0$ similarity to STEPS.  We compute the scenario-adjusted distribution parameters
$v^{\text{MED}}_{\xi,t,m}$, $v^{\text{LB}}_{\xi,t,m}$, and $v^{\text{UB}}_{\xi,t,m}$ from these multipliers, e.g.,
$$v^{\text{MED}}_{\xi,t,m} = v^{\text{MED,SDS}}_{\xi,t,m} \lambda_{\xi,t} + v^{\text{MED,STEPS}}_{\xi,t,m} (1 - \lambda_{\xi,t}).$$
Under this formula, for demand scenarios that are closer to the SDS projection than STEPS, the cost distribution parameters are also closer to their SDS counterpart.  This introduces reasonable relationships between EV demand and metal prices.

Let $v_{\xi,t,m}$ be the random variable that represents the price of metal $m$, with the distribution
\begin{equation}
    v_{\xi,t,m} = \left ( \frac{v^{\text{UB}}_{\xi,t,m} - v^{\text{MED}}_{\xi,t,m}}{\Phi^{-1}(0.7)} \mathbb{I}(Z^{\text{COST}} \geq 0) + \frac{v^{\text{LB}}_{\xi,t,m} - v^{\text{MED}}_{\xi,t,m}}{\Phi^{-1}(0.3)} \mathbb{I}(Z^{\text{COST}} < 0) \right ) Z^{\text{COST}} + v^{\text{MED}}_{\xi,t,m},
\end{equation}
where $\Phi$ is the cumulative distribution function for the standard normal distribution.  Here, the randomness is described by a single random variable $Z^{\text{COST}}$ that is distributed as a truncated standard normal on $[\Phi^{-1}(0.3),\Phi^{-1}(0.7)]$.  The marginal distributions for each $v_{\xi,t,m}$ are described by two truncated normal distributions, one for samples above the median and another for those below.  The distributions are constructed so that $v^{\text{MED}}_{\xi,t,m}$ matches the median of $v_{\xi,t,m}$ and the HPD interval $[v^{\text{LB}}_{\xi,t,m},v^{\text{UB}}_{\xi,t,m}]$ defines an equal-tailed 40\% credible interval for $v_{\xi,t,m}$.  
We again generate a discrete distribution via two Gaussian quadratures for $Z^{\text{COST}}$, one on the interval $[\Phi^{-1}(0.3),0]$ and one on $[0,\Phi^{-1}(0.7)]$, with $\left \lfloor \frac{n_c}{2} \right \rfloor$ samples allocated to each.  The quadrature yields values $v_{(\xi,\zeta),t,m}$ and probabilities $p^{\text{COST}}_{\zeta}$ for each scenario $\zeta \in \lBrack n_c \rBrack$.\textsuperscript{\endnote{The quadrature probabilities are reweighted by $\frac{1}{2}$ to account for $\mathbb{P}(Z^{\text{COST}} \geq 0)$ and $\mathbb{P}(Z^{\text{COST}} < 0)$.}}  This approach only guarantees that the first $n_c - 2$ moments of the discrete distribution match the underlying distribution, but ensures that scenario observations across time lie at the same cumulative probability level of the underlying distributions.  Although the observations of $Z^{\text{COST}}$ from the Gaussian quadrature are identical in each stage, the distribution parameters $(v^{\text{MED}}_{\xi,t,m},v^{\text{LB}}_{\xi,t,m},v^{\text{UB}}_{\xi,t,m})$ are derived from the demand observation and therefore follow the same multistage structure as in the demand scenarios.

From the metal price observations, we construct scenario-dependent cost parameters.  Consider some material cost parameter $c$ (e.g., cost of lithium carbonate, cost of LFP cathode powder), where $\overline{c}$ gives the deterministic value of this parameter from EverBatt.  Let $PC_m$ be the proportion by mass of the material composed of the metal $m$.  We compute the base cost $c^{\text{BASE}}_{\xi} = \overline{c} - \sum_{m \in \mathcal{M}} PC_m v^{\text{MED}}_{\xi,0,m}$, which is the cost of the material excluding the value of the metals it contains.  For cost index $\zeta$, contributions from metal prices are added back to the base cost, $c_{(\xi,\zeta),t} = c^{\text{BASE}}_\xi + \sum_{m \in \mathcal{M}} PC_m v_{(\xi,\zeta),t,m}.$  Then, the cost coefficients in \eqref{eq:constrCostOP} are computed as the sum of the relevant scenario-dependent material cost time series $c_{(\xi,\zeta),t}$ and other utility cost terms.  Figure~\ref{fig:costScenarios} shows a set of cost scenarios generated with this method.

Beyond batteries, other technologies associated with the energy transition are manufactured with critical minerals.  Wind turbines require with nickel, manganese, and chromium, and solar photovoltaics contain silicon \citep{iea2021minerals}, all of which are classified as critical minerals \citep{usgs2022critical}.  Manufacturing and recycling strategies for these products may be sensitive to material prices, which are likely correlated with product demand; our approach to model this correlation can be applied broadly to study other energy transition technologies.

\subsection{Multistage Scenario Tree}
\label{sec:scenarioTree}
We now connect this method of scenario generation to the notation introduced in Section~\ref{sec:model}.  The set of scenario tree nodes at stage $\sigma$ is $\Omega_{\sigma} = \{\text{NCX},\text{LFP}\} \times \lBrack n_d \rBrack^\sigma \times \lBrack n_c \rBrack$. The total number of scenarios (i.e., leaf nodes in the multistage tree) is $|\Omega_S| = 2 n_d^S n_c$.  For some stage $\sigma \in \mathcal{S}$ and node $\omega = (\psi,\xi,\zeta) \in \Omega_\sigma$, associated probabilities are given by $p_{\omega} = \frac{1}{2} p^{\text{COST}}_{\zeta} \prod_{\sigma'=1}^\sigma p^{\text{EV}}_{\xi_{\sigma'}}$, demand and supply by $d_{\omega,z,t,i} = d_{(\psi,\xi),z,t,i}$ and $s_{\omega,z,t,i} = s_{(\psi,\xi),z,t,i}$, and material costs by $c_{\omega,t} = c_{(\xi,\zeta),t}$.  The ancestor function for node $\omega \in \Omega_{\sigma}$ is defined as 
\begin{equation}
    a_{\omega}(t) = \begin{cases}
        (\psi,\xi_{1:\sigma_t},\zeta) \quad & \sigma_t \leq \sigma\\
        (\psi,\xi,\zeta) \quad & \text{otherwise},
    \end{cases}
\end{equation} where $\xi_{1:\sigma_t}$ contains the first $\sigma_t$ entries of $\xi$. If period $t$ is associated with an earlier stage than node $\omega$, the ancestor node is obtained by removing realizations of the demand randomness at stages after $\sigma_t$, while retaining the cost and market share realizations $\zeta$ and $\psi$.  The multistage scenario tree and relative timing of decisions from models \eqref{eq:deterministicP} and \eqref{eq:reformulatedP} are depicted in Figure~\ref{fig:scenarioTree}.

Importantly, this scenario tree structure yields natural groupings for the decomposition discussed in Section~\ref{sec:ssDecomposition}.  As $\zeta$ varies in the set $\lBrack n_c \rBrack$, the scenario supply $s_{\omega,z,t,i}$ and demand $d_{\omega,z,t,i}$ do not change.  Thus, a natural grouping at the first stage is $\mathcal{G}_{\psi,\xi} = \{(\psi,\xi,\zeta)\}_{\zeta = 1}^{n_c}$ for $(\psi,\xi) \in \{\text{NCX},\text{LFP}\} \times \lBrack n_d \rBrack$, with $n_g = 2n_d$.

%%%%%%%%%%%%%%%%%% Figure %%%%%%%%%%%%%%%%%%
\begin{figure}[tp]
    \centering
    \begin{minipage}[t]{0.48\textwidth}
        \centering
        \includegraphics[width=\linewidth]{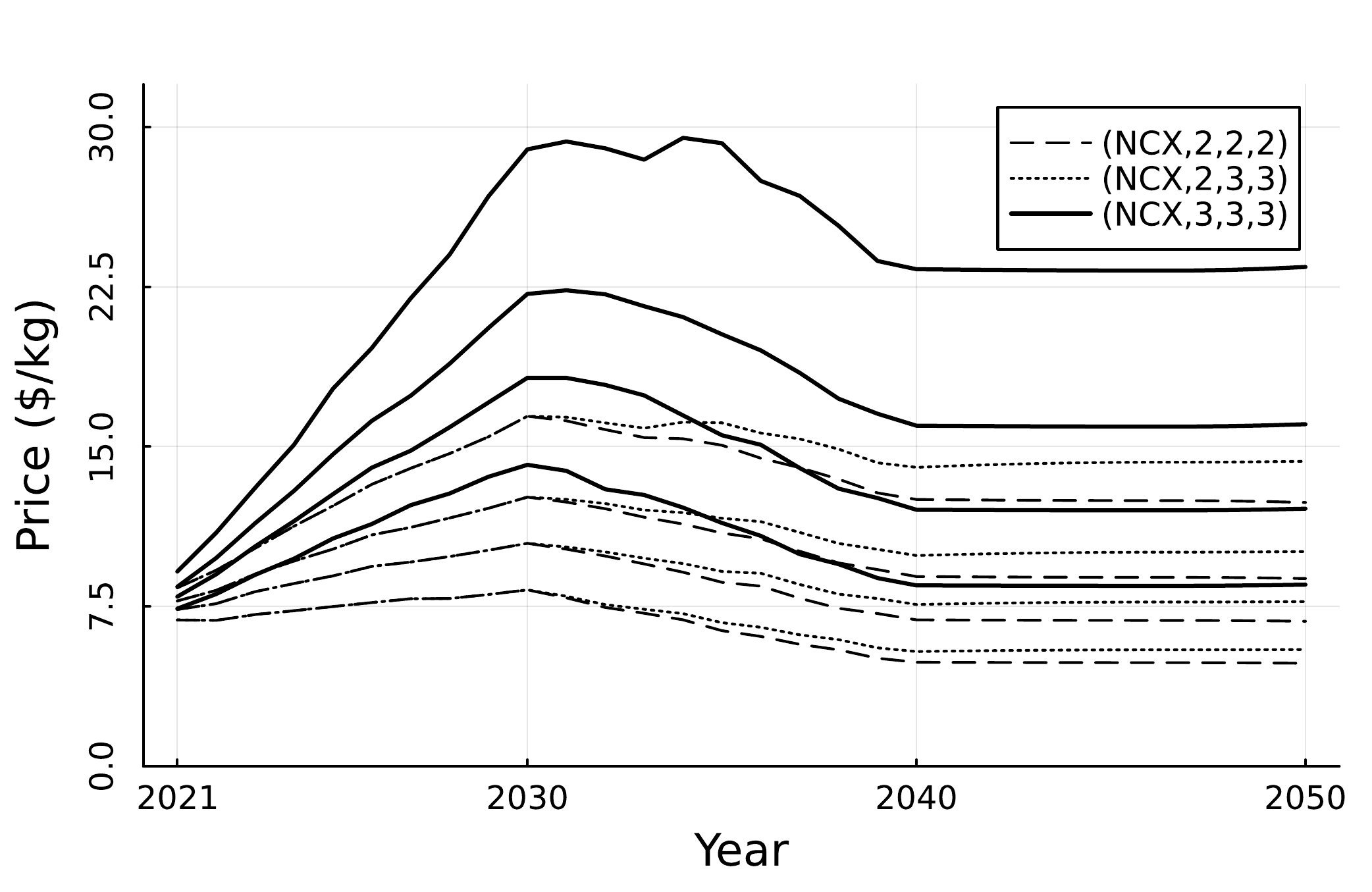}
        \caption{Four cost scenarios ($S=3$, $n_c = n_d = 4$) for the value $v$ of lithium carbonate under demand scenario $(\psi,\xi) = (\text{NCX},2,2,2)$ (low EV demand), demand scenario $(\text{NCX},3,3,3)$ (high demand), and demand scenario $(\text{NCX},2,3,3)$ (mixed demand).} \label{fig:costScenarios}
    \end{minipage}
    \hspace{0.2cm}
    \begin{minipage}[t]{0.48\textwidth}
        \centering
        \begin{tikzpicture}[scale=.75, transform shape]
        
            \tikzstyle{circ} = [draw, circle, minimum height=3em, minimum width=3em, font=\small]
        
            \node[draw,rectangle,dashed] (top) at (0,0) {Capacity Decision $y$};
            \node[circ, below left=3em and 5em of top.south] (NCX) {NCX};
            \node[circ, below right=3em and 5em of top.south] (LFP) {LFP};
        
            \draw[-stealth] (top.south) -- (NCX.north);
            \draw[-stealth] (top.south) -- (LFP.north);
        
            \node[circ, below left=3em and 1 em of NCX] (ncx1) {$1$};
            \node[circ, below right=3em and 1 em of NCX] (ncxc) {$n_c$};
            \node[] at ($(ncx1)!1/2!(ncxc)$) {$\dots$};
        
            \node[circ, below left=3em and 1 em of LFP] (lfp1) {$1$};
            \node[circ, below right=3em and 1 em of LFP] (lfpc) {$n_c$};
            \node[] at ($(lfp1)!1/2!(lfpc)$) {$\dots$};
        
            \draw[-stealth] (NCX.south) -- (ncx1.north);
            \draw[-stealth] (NCX.south) -- (ncxc.north);
            \draw[-stealth] (NCX.south) -- ($(ncx1.north)!1/3!(ncxc.north)$);
            \draw[-stealth] (NCX.south) -- ($(ncx1.north)!2/3!(ncxc.north)$);
        
            \draw[-stealth] (LFP.south) -- (lfp1.north);
            \draw[-stealth] (LFP.south) -- (lfpc.north);
            \draw[-stealth] (LFP.south) -- ($(lfp1.north)!1/3!(lfpc.north)$);
            \draw[-stealth] (LFP.south) -- ($(lfp1.north)!2/3!(lfpc.north)$);
        
            \node[circ, below=3em of ncx1] (ncx111) {$1$};
            \node[circ, below=3em of ncxc] (ncx1d1) {$n_d$};
            \node[] at ($(ncx111)!1/2!(ncx1d1)$) (ncx1dots) {$\dots$};
        
            \draw[-stealth] (ncx1.south) -- (ncx111.north);
            \draw[-stealth] (ncx1.south) -- (ncx1d1.north);
            \draw[-stealth] (ncx1.south) -- ($(ncx111.north)!1/3!(ncx1d1.north)$);
            \draw[-stealth] (ncx1.south) -- ($(ncx111.north)!2/3!(ncx1d1.north)$);
        
            \node[circ, below=3em of lfp1] (lfpc11) {$1$};
            \node[circ, below=3em of lfpc] (lfpcd1) {$n_d$};
            \node[] at ($(lfpc11)!1/2!(lfpcd1)$) (lfp1dots) {$\dots$};
        
            \draw[-stealth] (lfpc.south) -- (lfpc11.north);
            \draw[-stealth] (lfpc.south) -- (lfpcd1.north);
            \draw[-stealth] (lfpc.south) -- ($(lfpc11.north)!1/3!(lfpcd1.north)$);
            \draw[-stealth] (lfpc.south) -- ($(lfpc11.north)!2/3!(lfpcd1.north)$);
        
            \node[] at ($(ncx1d1)!1/2!(lfpc11)$) (middots) {$\dots$};
        
            \node[below=3em of ncx1dots.center, rotate=90, anchor=center] {$\cdots$};
            \node[below=3em of ncx111.center, rotate=90, anchor=center] {$\cdots$};
            \node[below=3em of ncx1d1.center, rotate=90, anchor=center] {$\cdots$};
            \node[circ, below=6em of ncx111] (ncx11s) {$1$};
            \node[circ, below=6em of ncx1d1] (ncx1ds) {$n_d$};
            \node[] at ($(ncx11s)!1/2!(ncx1ds)$) {$\dots$};
        
            \draw[-stealth] (ncx111.south)+(0em,-3em) -- (ncx11s.north);
            \draw[-stealth] (ncx111.south)+(0em,-3em) -- (ncx1ds.north);
            \draw[-stealth] (ncx111.south)+(0em,-3em) -- ($(ncx11s.north)!1/3!(ncx1ds.north)$);
            \draw[-stealth] (ncx111.south)+(0em,-3em) -- ($(ncx11s.north)!2/3!(ncx1ds.north)$);
        
            \node[below=3em of lfp1dots.center, rotate=90, anchor=center] {$\cdots$};
            \node[below=3em of lfpc11.center, rotate=90, anchor=center] {$\cdots$};
            \node[below=3em of lfpcd1.center, rotate=90, anchor=center] {$\cdots$};
            \node[circ, below=6em of lfpc11] (lfpd1s) {$1$};
            \node[circ, below=6em of lfpcd1] (lfpdds) {$n_d$};
            \node[] at ($(lfpd1s)!1/2!(lfpdds)$) {$\dots$};
        
            \draw[-stealth] (lfpcd1.south)+(0em,-3em) -- (lfpd1s.north);
            \draw[-stealth] (lfpcd1.south)+(0em,-3em) -- (lfpdds.north);
            \draw[-stealth] (lfpcd1.south)+(0em,-3em) -- ($(lfpd1s.north)!1/3!(lfpdds.north)$);
            \draw[-stealth] (lfpcd1.south)+(0em,-3em) -- ($(lfpd1s.north)!2/3!(lfpdds.north)$);
        
            \node[] at ($(ncx1ds)!1/2!(lfpd1s)$) (middotsS) {$\dots$};
        
            \node[right=2.5em of lfpcd1, align=center, rotate=270, anchor=center] (labelstage1) {Stage 1 \\ Operational \\ Decision $x$};
        
            \node[right=2.5em of lfpdds, align=center, rotate=270, anchor=center] (labelstageS) {Stage $S$ \\ Operational \\ Decision $x$};
        
            \draw[dashed] (ncx111.north -| ncx111.west) rectangle (lfpcd1.south -| lfpcd1.east);
            \draw[dashed] (ncx11s.north -| ncx11s.west) rectangle (lfpdds.south -| lfpdds.east);
            
        \end{tikzpicture}
        \caption{Multistage scenario tree and relative placement of capacity and operational decisions.  Within each market share node (NCX,LFP), there are $n_c$ cost nodes, and there are $n_d$ new demand nodes for each node in the previous stage.  Due to correlations between material cost and EV demand, cost parameters vary between demand nodes.
        } \label{fig:scenarioTree}
    \end{minipage}
\end{figure}
%%%%%%%%%%%%%%%%%% Figure %%%%%%%%%%%%%%%%%%

\section{Computational Performance}
\label{sec:computation}

In this section, we computationally evaluate the effectiveness of the aPWL algorithm and decomposition schemes.  Numerical experiments are conducted in Julia 1.9 on dual 48-core AMD EPYC 9474F@3.60 GHz processors with 386 GB of RAM, allowing for parallelization over 96 cores.  Linear and mixed-integer problems are solved with Gurobi 10.0.  We use BARON 23.6.23 with CPLEX 22.1 and the embedded version of Ipopt as the mixed-integer and nonlinear subsolvers.  Problems are solved to a relative optimality tolerance of $10^{-4}$ and a primal feasibility tolerance of $10^{-8}$.

\subsection{Concave Minimization Algorithms}
We first evaluate the performance of aPWL (Algorithm~\ref{alg:PWL}) and the impact of the reformulated problem in a setting where subproblems are solved in the extensive form without stage decomposition.  These models have a single operational scenario, where supply and demand are taken from the SDS and NCX projections, and material costs are kept constant at the EverBatt values.  To encourage greater branching on capacity variables, the direct recycling capacity is fixed to 0, encouraging investment in cathode production facilities in the optimal solution.  The model contains a single zone, the United States (U.S.).

For the continuous problem \eqref{eq:deterministicP}, we compare aPWL to a finitely convergent spatial branch and bound algorithm (sB\&B, \citealt{shectman1998finite}) and to the global optimization package BARON \citep{sahinidis1996baron}.  sB\&B is a standard algorithm for global optimization of this problem class, and BARON is a leading commercial solver for global (mixed-integer) nonlinear programs.  For the reformulation \eqref{eq:reformulatedP}, we compare aPWL to BARON.  In this setting, we do not evaluate sB\&B due to the presence of integrality constraints.  To solve over discontinuous objective functions with BARON, we reformulate \eqref{eq:objFuncForm} with binary variables:
\begin{equation*}
        y \leq u z,\quad f(y) = \sum_i q_i y^{r_i} + w z,\quad z \in \{0,1\}.
\end{equation*}

%%%%%%%%%%%%%%%%%% Figure %%%%%%%%%%%%%%%%%%
\begin{table}[tp]
    \centering
    \begin{tabular}{c|c|c|c|c|c|c|c|c|c|c}
        & \multicolumn{6}{c|}{Problem \eqref{eq:deterministicP}} & \multicolumn{4}{c}{Problem \eqref{eq:reformulatedP}}\\
        \hline
        & & \multicolumn{2}{c|}{aPWL} & \multicolumn{2}{c|}{sB\&B} & BARON & & \multicolumn{2}{c|}{aPWL} & BARON \\
        $\beta_{\text{U.S.}}$ & $n_y$ & Time (s) & \# Iter & Time (s) & \# Iter & Time (s) & $n_y$ & Time (s) & \# Iter & Time (s) \\
        \hline
         0.1 & 3,765 & 167.4 & 2 & 2,421.6 & 2,132 & 724.3 & 96 & 2.7 & 2 & 7.0\\
         0.15 & 5,635 & 616.0 & 3 & 1,130.1 & 369 & 1,803.1 & 96 & 5.3 & 2 & 6.7\\
         0.2 & 7,508 & 561.9 & 2 & 2,519.4 & 445 & 3,779.5 & 96 & 9.3 & 2 & 10.0\\
         0.25 & 9,377 & 1,295.4 & 3 & 4,866.2 & 471 & 11,209.6 & 96 & 14.0 & 1 & 16.7\\
         0.3 & 11,250 & 1,008.6 & 2 & 3,438.0 & 123 & 14,298.7 & 96 & 20.1 & 1 & 22.8\\
         0.35 & 13,119 & 2,162.4 & 2 & --- & --- & 24,437.8 & 96 & 30.9 & 1 & 34.3\\
         0.4 & 14,994 & 2,622.0 & 2 & 17,185.0 & 263 & --- & 96 & 40.2 & 1 & 43.1\\
         0.45 & 16,861 & 4,722.1 & 2 & --- & --- & --- & 96 & 50.4 & 1 & 51.9\\
         0.5 & 18,735 & 5,481.8 & 2 & 13,483.6 & 99 & --- & 96 & 62.6 & 1 & 65.8\\
    \end{tabular}
    \caption{Algorithm and formulation comparison for the extensive form of \eqref{eq:deterministicP} and \eqref{eq:reformulatedP} with a single zone in the U.S. ($|\mathcal{Z}| = 1$).  A ---\ indicates that the model was not solved within 10 hours.  $n_y$ gives the number of capacity variables, which have concave cost functions.}
    \label{tbl:deterministicResults}
\end{table}
%%%%%%%%%%%%%%%%%% Figure %%%%%%%%%%%%%%%%%%

Table~\ref{tbl:deterministicResults} compares performance in the deterministic setting by algorithm and model across a variety of scales, generated by varying the proportion $\beta_{\text{U.S.}}$ of the global EV market that is modeled.  For the original formulation \eqref{eq:deterministicP}, aPWL significantly outperforms all other methods, solving all problems within the time limit, and can terminate up to 14x faster than both sB\&B and BARON.  Moreover, aPWL exhibits more consistent performance than sB\&B as the problem scale increases.  The reformulation \eqref{eq:reformulatedP} allows for extreme reductions in solve times, allowing global optima to be found in minutes with either aPWL or BARON.  For the reformulation, aPWL slightly outperforms BARON on all problems.

Table~\ref{tbl:zoneResults} presents results that evaluate the scalability of the reformulation~\eqref{eq:reformulatedP} as the number of zones $|\mathcal{Z}|$ increases.  In these experiments, we allocate a total demand of $0.5$ evenly across a variable number of zones, so $\beta_z = 0.5/|\mathcal{Z}|$.  Half of the zones are located in the U.S., and the other half are in China.  For all experiments with more than one zone, the original formulation~\eqref{eq:deterministicP} can be solved by neither aPWL nor BARON within the 10~hour time limit, and cases with more than $10$ zones also encounter memory limitations due to the number of first-stage variables.  Reformulated problems~\eqref{eq:reformulatedP} with up to 50~zones can be solved by aPWL within 10~hours.  As the number of first-stage variables scales with the number of zones, solve times for the reformulation also increase, approaching 7~hours with 50 zones.  aPWL outperforms BARON on all reformulated problems, with more drastic performance differences in cases with many zones.  With 40-zones, aPWL solves nearly 3x faster than BARON, and with 50-zones, aPWL solves within 7~hours, whereas BARON does not terminate within the time limit.

%%%%%%%%%%%%%%%%%% Figure %%%%%%%%%%%%%%%%%%
\begin{table}[tp]
    \centering
    \begin{tabular}{c|c|c|c|c|c|c|c|c}
        & \multicolumn{4}{c|}{Problem \eqref{eq:deterministicP}} & \multicolumn{4}{c}{Problem \eqref{eq:reformulatedP}}\\
        \hline
        & & \multicolumn{2}{c|}{aPWL} & BARON & & \multicolumn{2}{c|}{aPWL} & BARON \\
        $|\mathcal{Z}|$ & $n_y$ & Time (s) & \# Iter & Time (s) & $n_y$ & Time (s) & \# Iter & Time (s) \\
        \hline
         1 & 18,745 & 5,481.8 & 2 & --- & 96 & 62.6 & 1 & 65.8\\
         2 & 37,470 & --- & --- & --- & 192 & 120.8 & 1 & 143.1\\
         5 & 93,675 & --- & --- & --- & 480 & 296.3 & 3 & 987.1\\
         10 & 187,350 & --- & --- & --- & 960 & 630.9 & 4 & 656.8\\
         20 & 374,700 & --- & --- & --- & 1,920 & 1,674.1 & 4 & 1,777.7\\
         40 & 749,400 & --- & --- & --- & 3,840 & 11,034.5 & 5 & 31,355.9\\
         50 & 936,750 & --- & --- & --- & 4,800 & 24,819.3 & 6 & ---\\
    \end{tabular}
    \caption{Computational performance of each formulation by number of zones.  A ---\ indicates that the model was not solved within 10~hours.  $n_y$ gives the number of capacity variables, which have concave cost functions.}
    \label{tbl:zoneResults}
\end{table}
%%%%%%%%%%%%%%%%%% Figure %%%%%%%%%%%%%%%%%%

\subsection{Decomposition Approaches}
We next evaluate the performance of the Benders' decomposition when solving subproblems within aPWL.  These experiments use instances of the reformulated problem \eqref{eq:reformulatedP} with multiple operational scenarios.  We compare different nodal groupings against solving the subproblem in its extensive form, without decomposition.  The nodal groupings considered are the standard single-cut ($\mathcal{G}_1 = \Omega_1$) and multi-cut ($\mathcal{G}_\omega = \{\omega\}\ \forall \omega \in \Omega_1)$ aggregation strategies, as well as the intermediate grouping discussed in Section~\ref{sec:scenarioTree} ($\mathcal{G}_{\psi,\xi} = \{(\psi,\xi,\zeta)\}_{\zeta = 1}^{n_c}\ \forall (\psi,\xi) \in \{\text{NCX},\text{LFP}\} \times \lBrack n_d \rBrack$).  The problems have two zones: $\mathcal{Z} = \{\text{U.S.},\ \text{China}\}$.  A set of test cases are generated by varying the model scale $(\beta_\text{U.S.},\beta_{\text{China}})$ and the number of scenarios $|\Omega_S| = 2 n_d^S n_c$, with $n_c = n_d$ and $S = 3$ for all problems.  As optimal solutions in this setting have nonzero cathode production capacity, we no longer fix direct recycling capacity to 0.

%%%%%%%%%%%%%%%%%% Figure %%%%%%%%%%%%%%%%%%
\begin{table}[tp]
    \centering
    \begin{tabular}{c|c|c|c|c|c|c|c|c}
         & & & \multicolumn{2}{c|}{Benders'} & \multicolumn{2}{c|}{Benders'} & \multicolumn{2}{c}{Benders'}\\
         & & \multicolumn{1}{c|}{Extensive Form} & \multicolumn{2}{c|}{Single-Cut} & \multicolumn{2}{c|}{Grouped-Cut} & \multicolumn{2}{c}{Multi-Cut}\\
         $(\beta_{\text{U.S.}},\beta_{\text{China}})$ & $|\Omega_S|$ & Time (s) & Time (s) & \# Iter & Time (s) & \# Iter & Time (s) & \# Iter\\
         \hline
         \multirow{3}{*}{(0.05,\,0.125)}
         & 162 & 456.9 & 59.8 & 250 (7) & 24.5 & 63 (8) & 30.9 & 52 (7)\\
         & 1,250 & 17,334.7 & 210.4 & 307 (8) & 63.0 & 53 (6) & 100.5 & 46 (4)\\
         & 20,000 & --- & 1,532.5 & 249 (7) & 810.6 & 52 (5) & 930.8 & 37 (3)\\
         \hline
         \multirow{3}{*}{(0.1,\,0.25)} 
         & 162 & 280.0 & 59.2 & 261 (4) & 19.5 & 63 (7) & 19.3 & 45 (3)\\
         & 1,250 & 2,694.4 & 155.5 & 229 (2) & 56.6 & 53 (3) & 68.8 & 36 (4)\\
         & 20,000 & --- & 1,576.2 & 240 (5) & 727.9 & 52 (4) & 1,168.3 & 40 (5)\\
         \hline
         \multirow{3}{*}{(0.15,\,0.375)} 
         & 162 & 356.7 & 44.3 & 203 (1) & 14.2 & 55 (1) & 21.2 & 49 (3)\\
         & 1,250 & 2,019.9 & 175.8 & 270 (7) & 59.6 & 52 (5) & 93.5 & 45 (6)\\
         & 20,000 & --- & 1,500.9 & 240 (4) & 717.7 & 47 (1) & 851.3 & 35 (2)\\
         \hline
         \multirow{3}{*}{(0.2,\,0.5)}
         & 162 & 747.6 & 52.0 & 244 (4) & 16.1 & 64 (2) & 17.5 & 43 (3)\\
         & 1,250 & 2,592.7 & 170.1 & 260 (7) & 57.1 & 48 (2) & 61.7 & 36 (2)\\
         & 20,000 & --- & 1,564.0 & 239 (3) & 759.2 & 50 (4) & 985.3 & 36 (4)
    \end{tabular}
    \caption{Subproblem decomposition algorithm performance.  \# Iter counts the total Benders' iterations across all subproblems, with the number of iterations of Algorithm~\ref{alg:PWL} given in parentheses. A --- indicates that the model was not solved within 10 hours.}
    \label{tbl:stochasticResults}
\end{table}
%%%%%%%%%%%%%%%%%% Figure %%%%%%%%%%%%%%%%%%

Table~\ref{tbl:stochasticResults} shows the computational results of the aPWL algorithm with the various decomposition approaches.  Benders' decompositions outperform the extensive form in all test cases.  Further, the grouped-cut aggregation strategy outperforms both single-cut and multi-cut implementations, solving up to 38\% faster than the next fastest approach.  The grouped-cut strategy effectively balances the tradeoff between increased subproblem size from adding many cuts in each iteration and information loss from aggregating cuts.  With grouped-cut aggregation, the decomposition improves solve times by up to 275x over the extensive form.  As the scale $\beta$ increases, solve times decrease.  This occurs due to the use of relative optimality tolerances, with objective values increasing at larger scales.  However, the size of the reformulated model does not grow accordingly.

\section{Managerial Insight and Case Studies}
\label{sec:insight}

This section discusses the insights obtained from the optimal solutions to our models.  Section~\ref{subsec:optSol} presents optimal investment plans and quantifies the reductions in cost and environmental impact from recycling.  Section~\ref{subsec:hydrovalue} measures the operational value of hydrometallurgical recycling when the chemistry of new and recycled batteries do not match.  Section~\ref{subsec:policy} considers the ability of tariffs, capacity grants, and production credits to encourage U.S. domestic investment in battery recycling capacity.  Section~\ref{subsec:uncertainty} analyzes how firms respond to different levels of uncertainty in the EV stock and material cost scenario distributions.

\subsection{Optimal Investment Strategies}
\label{subsec:optSol}

We analyze the optimal solutions of two cases.  The first (Case~1) has a single zone for the U.S., with $\beta_{\text{U.S.}} = 0.2$.  The second (Case~2) has two zones, one for the U.S. and one for China, with $(\beta_{\text{U.S.}},\beta_{\text{China}}) = (0.2,0.5)$.  Both cases have 20,000 operational scenarios ($S = 3$, $n_c = n_d = 10$).  The $20\%$ and $50\%$ of global EV market shares correspond approximately to projections for market share in the U.S. and China, respectively, through 2030 \citep{iea2023explorer}.  As direct recycling is newer and less commercialized than other recycling processes, we analyze two variations of each case: one where direct recycling facilities can be constructed (labeled DR1) and one where the direct recycling capacity is fixed to $0$ (labeled DR0). 

In Section~\ref{sec:ecZoneCost} of the appendix, we discuss how cost parameters vary between the U.S. and China and provide a sensitivity analysis on these values.  This analysis shows that the optimal objective value and recycling solution are not sensitive to these location-dependent cost parameters, but the cathode production solution exhibits sensitivity to the cost of equipment in China.

%%%%%%%%%%%%%%%%%% Figure %%%%%%%%%%%%%%%%%%
\begin{table}[tp]
    \centering
    \begin{tabular}{c|r|c|c|c|c|c|c}
         & & \multicolumn{3}{c|}{Case 1 (U.S.)} & \multicolumn{3}{c}{Case 2 (U.S. + China)}\\
         & & R0 & DR0 & DR1 & R0 & DR0 & DR1\\
        \hline
        \multirow{4}{*}{Manufacturing} & Cost (billion \$) & 1,732 & 1,365 & \bf{1,352} & 6,062 & 4,706 & \bf{4,688}\\
        & Energy Consumption (PJ) & 25,938 & 25,609 & \bf{24,317} & 91,981 & 90,738 & \bf{88,803}\\
        & GHG Emissions (Mt) & 1,648 & 1,616 & \bf{1,536} & 5,987 & 5,910 & \bf{5,793}\\
        & Recycled Cells (million) & 0 & 42,758 & \bf{42,765} & 0 & 150,391 & \bf{150,460} \\
        \hline
        \multirow{3}{*}{Disposal} & Energy Consumption (PJ) & 80 & \bf{1} & \bf{1} & 281 & 2 & \bf{1} \\
        & GHG Emissions (Mt) & 437 & \bf{4} & \bf{4} & 1,531 & 8 & \bf{7} \\
        & Disposed Cells (million) & 43,205 & 447 & \bf{440} & 151,218 & 826 & \bf{757}
    \end{tabular}
    \caption{Comparison of optimal solutions in expected cost and environmental impact of manufacturing (including recycling) and disposal of unrecycled batteries.  R0 indicates the solution with no recycling capacity, DR0 the solution where no direct recycling is allowed, and DR1 the solution where all recycling methods are allowed.
    }
    \label{tbl:solutionValue}
\end{table}
%%%%%%%%%%%%%%%%%% Figure %%%%%%%%%%%%%%%%%%

\subsubsection*{Comparison of Impact}  To demonstrate the value of recycling, we compare to a solution where no recycling capacity is constructed (labeled R0) in objective cost, total energy consumption, and greenhouse gas (GHG) emissions.  Energy consumption and GHG emissions are evaluated in expectation over the operational scenarios and aggregated over the model horizon. GHG emissions are reported in equivalent units of carbon dioxide (CO2).  These metrics are calculated by the EverBatt environmental impact model \citep{dai2019everbatt} and include the impact of recycling and cathode production as well as the upstream contribution of new materials that are purchased.  We also report the estimated energy consumption and emissions from disposal of battery cells which are not recycled.  Disposal requires $2.47$ MJ of energy and produces $13.43$ kg of equivalent CO2 emissions per kilogram of cells \citep{zhu2020research}.  Variation R0 provides baseline cost and environmental impact values for a solution in which new batteries are manufactured only using new materials.  The improvement in these metrics relative to the R0 variation is due to the replacement of new materials with their recycled counterpart during new battery manufacturing and quantifies the impact of optimal investment in EV battery recycling capacity on the supply chain.

Table~\ref{tbl:solutionValue} compares the metrics in the variations R0, DR0, and DR1.  Incorporating recycling into the supply chain reduces the total objective cost by about $22\%$ in Case~2, variation DR0, with $0.3\%$ additional reduction when direct recycling is allowed (variation DR1).  In Case~1, the cost savings are more significant under direct recycling, increasing from $21\%$ (DR0) to $22\%$ (DR1).  In both cases, incorporating non-direct recycling (DR0) reduces the energy consumption of the manufacturing process marginally ($\sim 1\%$), with more significant reductions of $6\%$ in Case~1 and $3\%$ in Case~2 under direct recycling (DR1).  Similarly, GHG emissions from manufacturing are reduced by $1-2\%$ when incorporating non-direct recycling (DR0), and are reduced by up to $7\%$ with direct recycling (DR1).  Although there are large, comparable cost savings under direct and non-direct recycling, improvements in environmental impact are more significant under direct recycling, especially when only considering production in the U.S.  This result can be partially attributed to the additional energy consumption and emissions from the cathode production step, which is needed when hydrometallurgical recycling is used but avoided with direct recycling.

Recycling also reduces the number of retired batteries which require disposal.  In expectation, over 99\% of retired battery cells are recycled in all cases, meaning that enough recycling capacity is constructed to recycle all retired battery cells in the most likely scenarios.  There is an equivalent reduction in the expected number of disposed cells and the environmental impacts (energy consumption and GHG emissions) of disposal.  The energy requirements of disposal are insignificant relative to those of recycling and new battery manufacturing.  However, the adoption of recycling yields a significant reduction in GHG emissions from disposal: across manufacturing, recycling, and disposal, total GHG emissions are reduced by 22\% and 26\% in Case~1, variations DR0 and DR1, and by 21\% and 23\% in Case~2, respectively.

Beyond impacts on cost, energy, and emissions, battery recycling has other externalities.  Battery disposal in landfill causes environmental hazards, including soil and groundwater leaching of hazardous material \citep{mrozik2021environmental}; these risks are reduced by recycling, which avoids placing battery material in landfills.  Recycling provides an alternate path to acquire critical minerals, reducing dependence on international supply chains which are vulnerable to physical or political disruption \citep{gemechu2017geopolitical,nygaard2023geopolitical}.  On the other hand, end-of-life EV batteries maintain sufficient capacity for use as stationary energy storage on the power grid \citep{faessler2021stationary}.  Whether retired batteries should be recycled or processed for stationary storage is a strategic decision that depends on the value of each approach and the demand for storage.  Batteries repurposed for storage are valued between $\$90 / \text{kWh}$ and $\$105 / \text{kWh}$, less refurbishment costs.  Estimates of refurbishment costs vary widely between $\$15 / \text{kWh}$ and $\$62 / \text{kWh}$ \citep{neigum2024technology}.  In Case~2, variation DR1, if all batteries that are recycled in the optimal solution are repurposed for storage, a discounted net value between $\$315$ billion and $\$1,013$ billion is realized, compared to a value of $\$1,375$ billion from recycling.\textsuperscript{\endnote{We assume an energy density of $200$ Wh/kg \citep[][Fig.~2]{myung2017nickel}.}}  Even under the most conservative cost structure, the value of recycling exceeds that of repurposing batteries for stationary storage.

%%%%%%%%%%%%%%%%%% Figure %%%%%%%%%%%%%%%%%%
\begin{figure}[tp]
\begin{center}
    {\begin{subfigure}[t]{.48\textwidth} \centering
    \includegraphics[width=\linewidth]{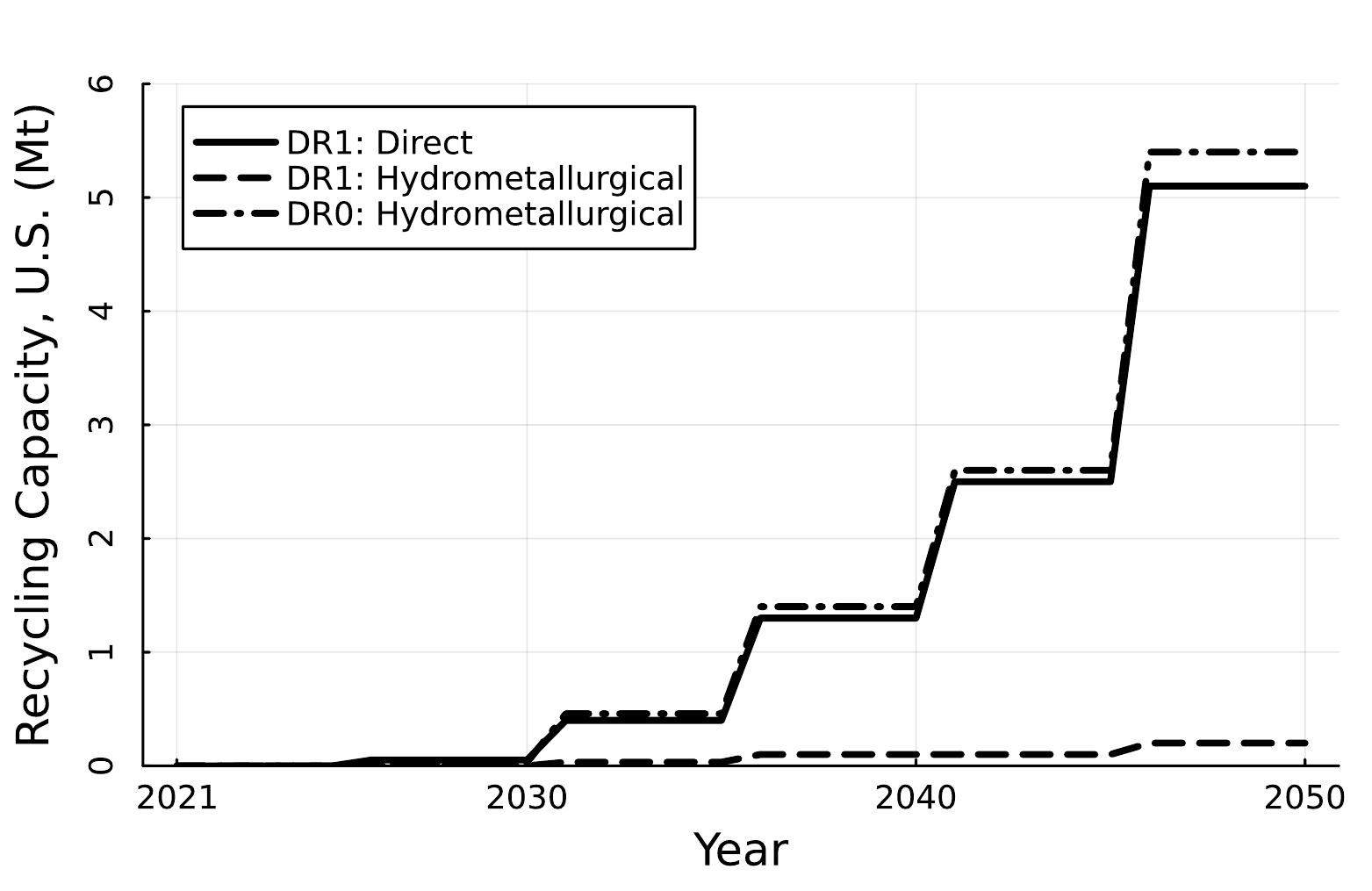}
    \captionsetup{font=footnotesize}
    \caption{Case 1, U.S. Recycling Capacity}
    \end{subfigure}
    \hspace{0.2cm}
    \begin{subfigure}[t]{.48\textwidth} \centering
    \includegraphics[width=\linewidth]{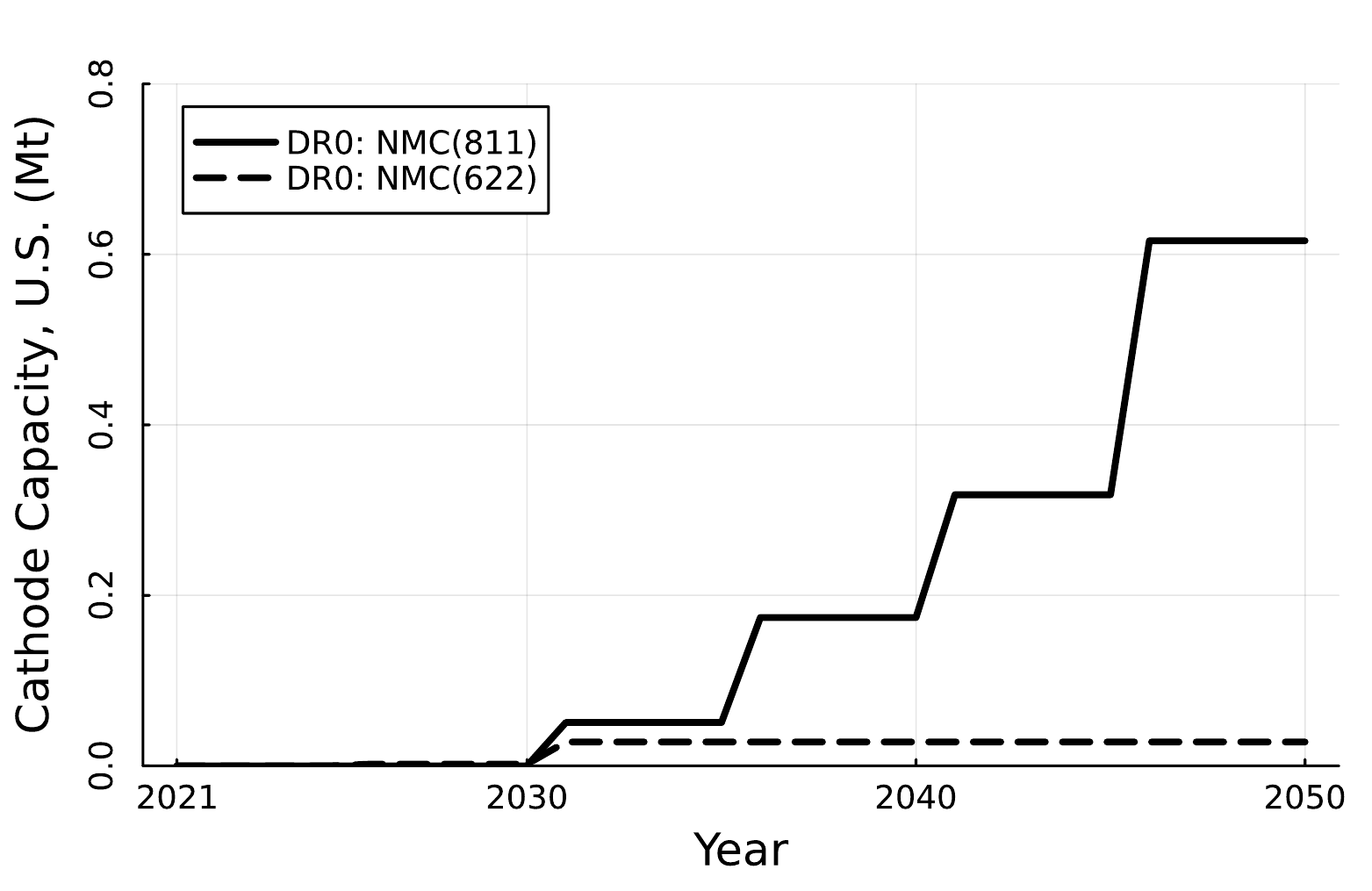}
    \captionsetup{font=footnotesize}
    \caption{Case 1, U.S. Cathode Production Capacity}
    \end{subfigure}
    \newline
    \begin{subfigure}[t]{.48\textwidth} \centering
    \includegraphics[width=\linewidth]{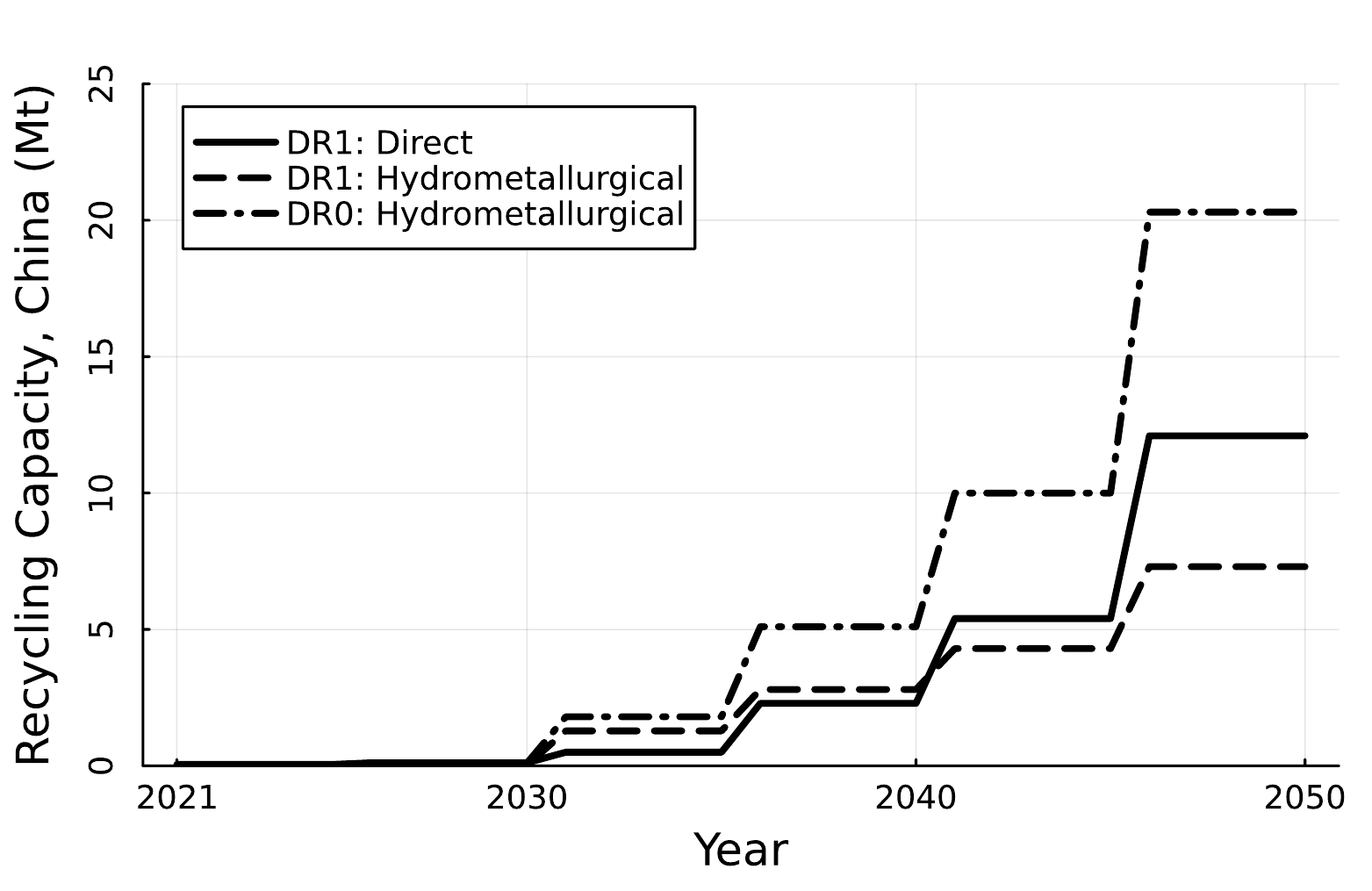}
    \captionsetup{font=footnotesize}
    \caption{Case 2, China Recycling Capacity}
    \end{subfigure}
    \hspace{0.2cm}
        \begin{subfigure}[t]{.48\textwidth} \centering
    \includegraphics[width=\linewidth]{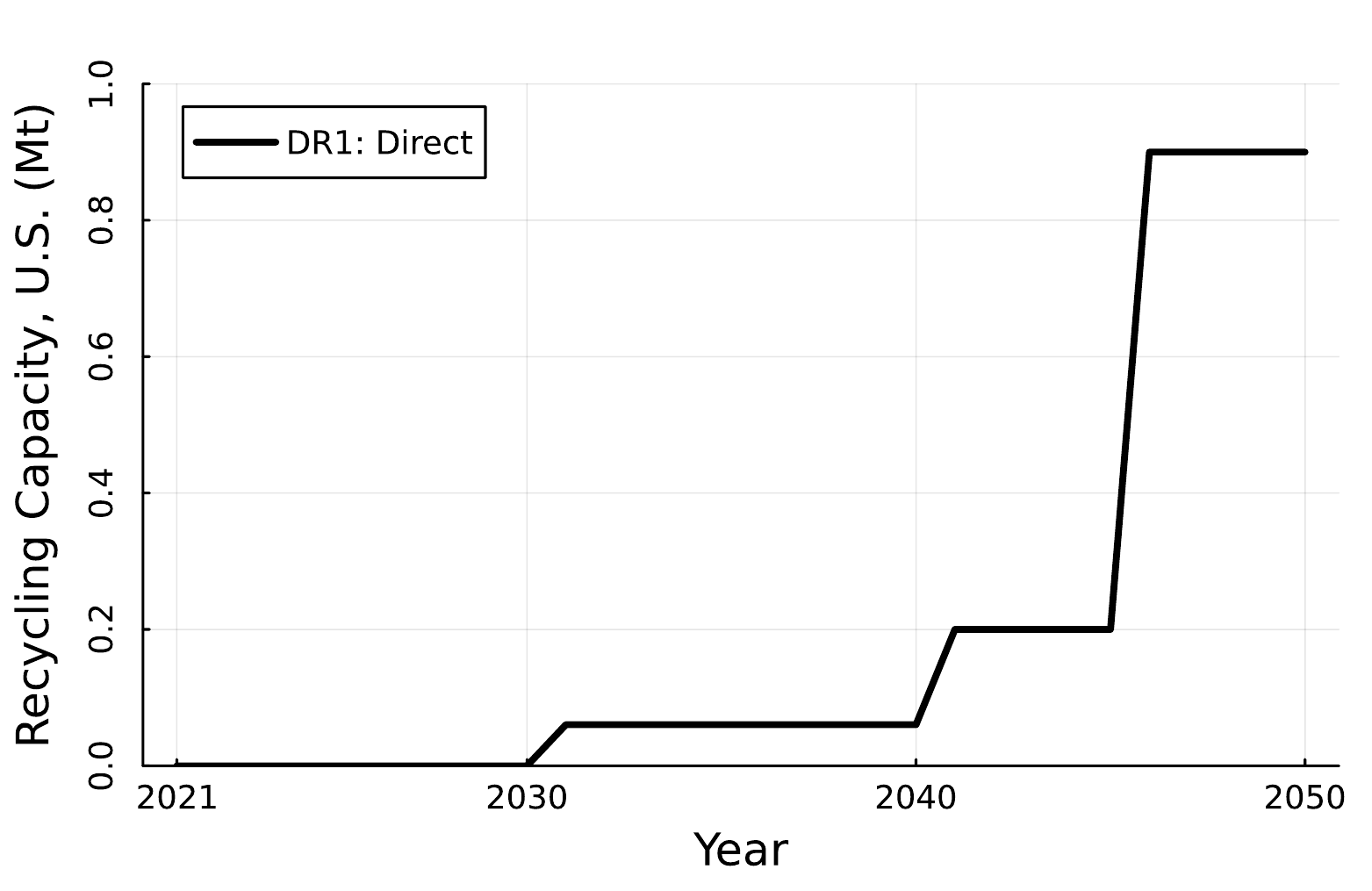}
    \captionsetup{font=footnotesize}
    \caption{Case 2, U.S. Recycling Capacity}
    \end{subfigure}
    \newline
    \begin{subfigure}[t]{.48\textwidth} \centering
    \includegraphics[width=\linewidth]{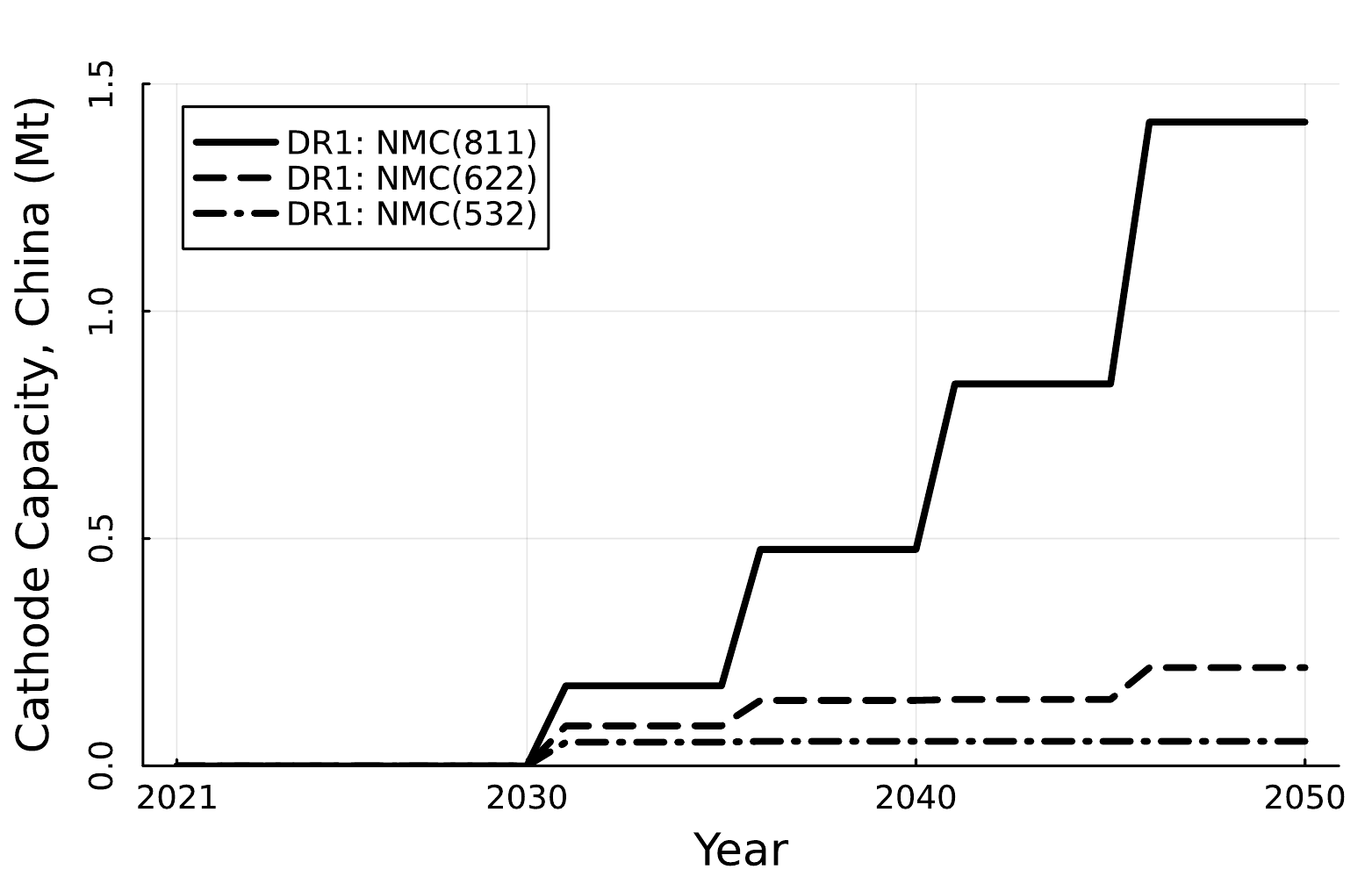}
    \captionsetup{font=footnotesize}
    \caption{Case 2 DR1, China Cathode Production Capacity}
    \end{subfigure}
    \hspace{0.2cm}
    \begin{subfigure}[t]{.48\textwidth} \centering
    \includegraphics[width=\linewidth]{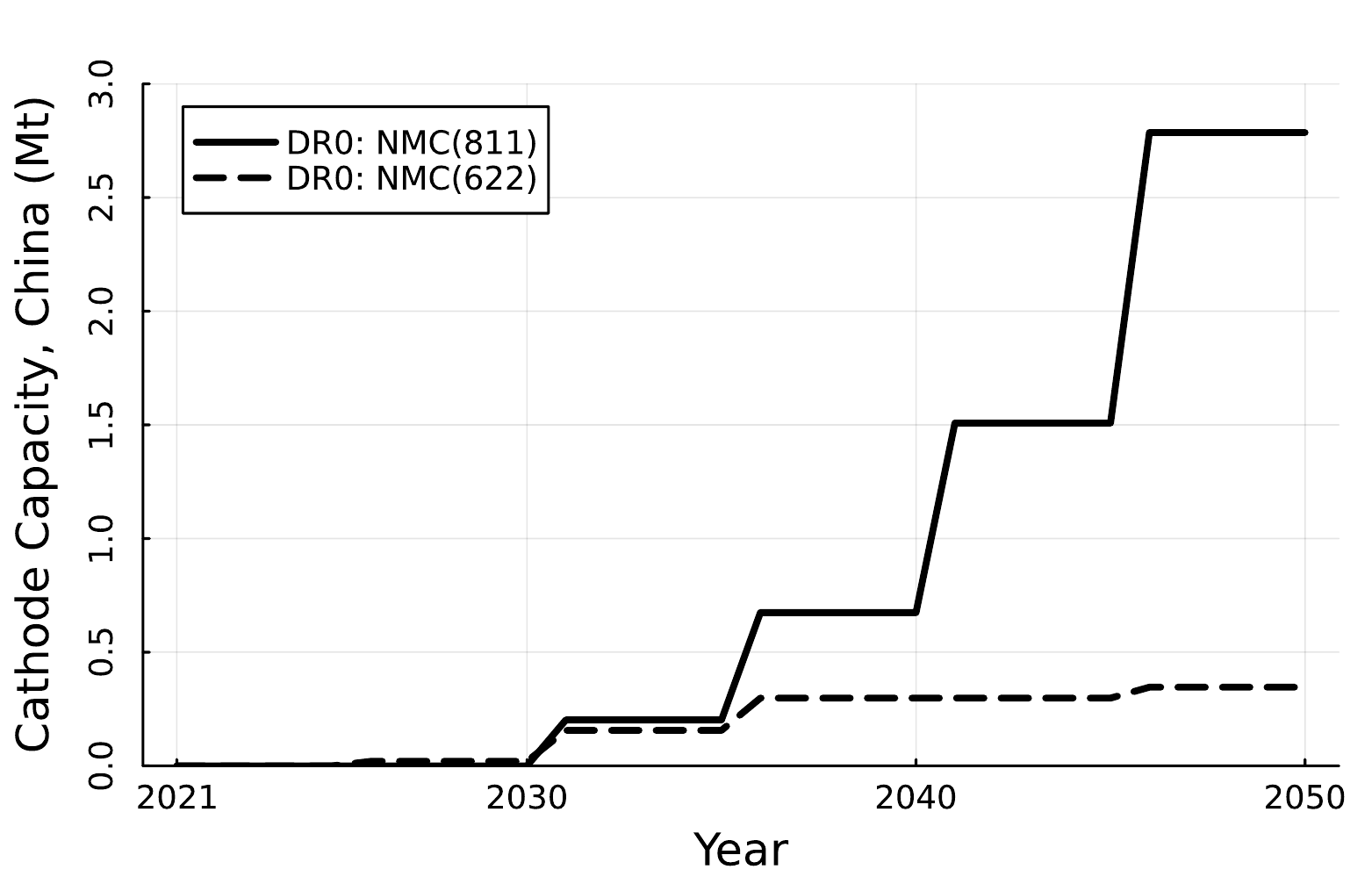}
    \captionsetup{font=footnotesize}
    \caption{Case 2 DR0, China Cathode Production Capacity}
    \end{subfigure}
    }
\caption{Optimal capacity investment decisions for recycling and cathode production facilities by zone. No cathode production capacity is constructed for the DR1 variation of Case 1.  Recycling capacity solutions for DR0 and DR1 models are shown on the same plots.} \label{fig:optCapacity}
\end{center}
\end{figure}
%%%%%%%%%%%%%%%%%% Figure %%%%%%%%%%%%%%%%%%

\subsubsection*{Optimal Investment Plan} Figure~\ref{fig:optCapacity} depicts the optimal capacity decisions for the variations DR1 and DR0 of both cases.    In Case~1, when direct recycling is allowed (variation DR1), the U.S. primarily constructs direct recycling capacity starting in 2026, supplemented by a smaller amount of hydrometallurgical capacity. The optimal solution constructs neither pyrometallurgical capacity nor cathode production capacity. When direct recycling cannot be constructed (variation DR0), the U.S. constructs a comparable amount of total recycling capacity, with direct recycling facilities replaced with hydrometallurgical.  To produce cathode powder from recycled material, cathode production capacity for NMC(622) and NMC(811) is constructed in this variation.  As of 2023, $650$ kt of U.S. battery recycling capacity is announced to be completed by 2030 \citep{icct2023capacity}, which is greater than the optimal capacity of $430$ kt recommended by our model in the DR1 variation and the $460$ kt recommended in the DR0 variation.  However, the announced capacity exceeds expectations for the available supply of retired batteries, suggesting potential overinvestment in recycling facilities \citep{abi2024capacity}.  This difference may be due to factors beyond cost, including environmental and geopolitical considerations.

In the DR1 variation of Case~2, a mixture of hydrometallurgical and direct recycling capacity is constructed in China, starting in 2021.  The U.S. constructs only direct recycling capacity in a lesser amount than in Case 1, starting in 2031.  Cathode production capacity for NMC(532), NMC(622), and NMC(811) is developed in China, and no cathode production capacity is installed in the U.S.  At the operational scale, retired batteries are often transported from the U.S. to China to be recycled, then recycled cathode powder is transported back to the U.S.  In the DR0 variation, only hydrometallurgical recycling capacity is constructed in China, and no recycling capacity is installed in the U.S.  The total installed capacity is comparable to the total amount in the DR1 variation.  Relative to the DR1 variation, NMC(622) and NMC(811) cathode production capacity is increased and capacity for NMC(532) is eliminated.

%%%%%%%%%%%%%%%%%% Figure %%%%%%%%%%%%%%%%%%
\begin{figure}[tp]
    \begin{subfigure}[t]{.48\textwidth} \centering
    \includegraphics[width=\linewidth]{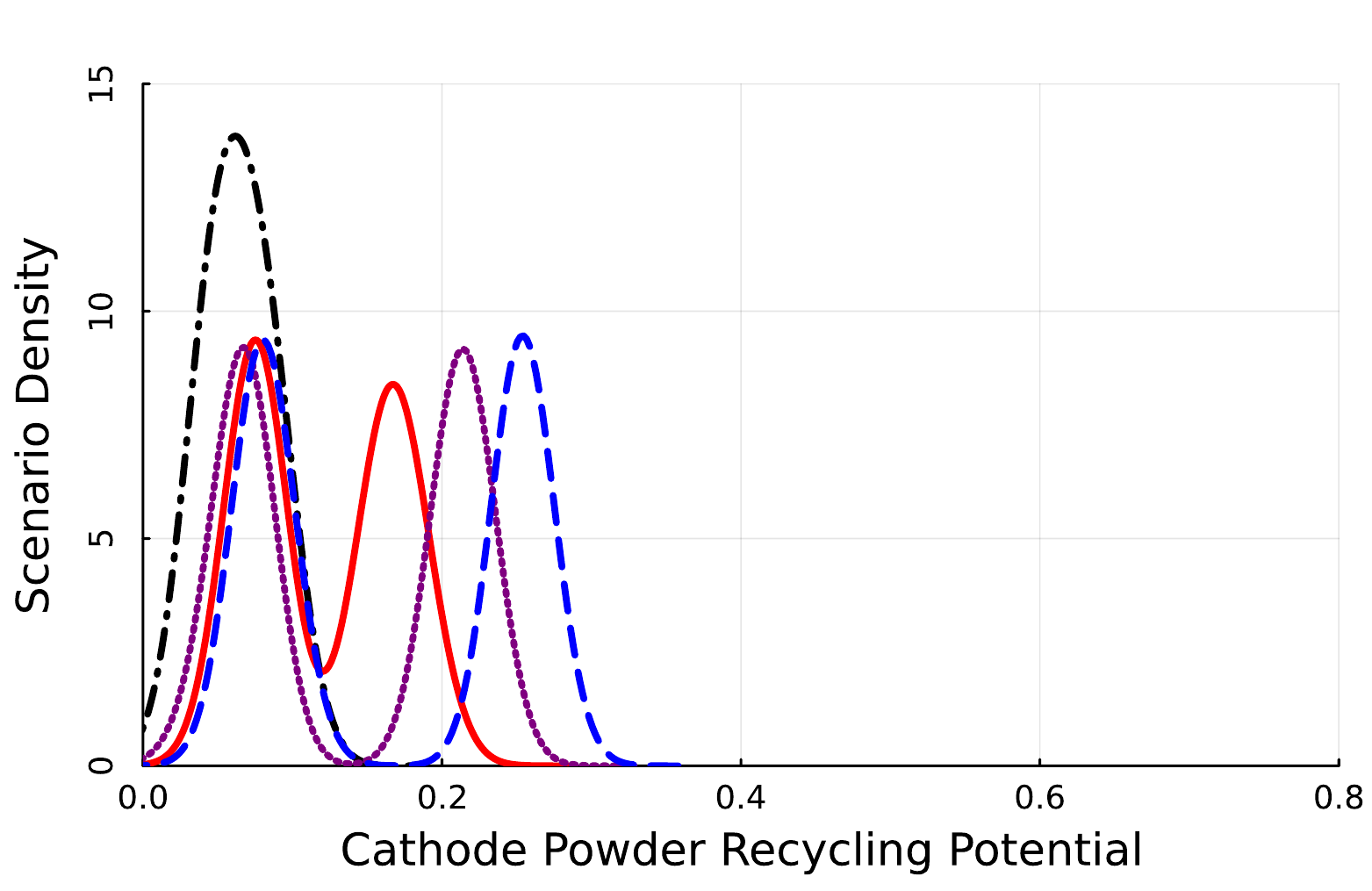}
    \captionsetup{font=footnotesize}
    \caption{Case 1 DR1}
    \end{subfigure}
    \hspace{0.2cm}
    \begin{subfigure}[t]{.48\textwidth} \centering
    \includegraphics[width=\linewidth]{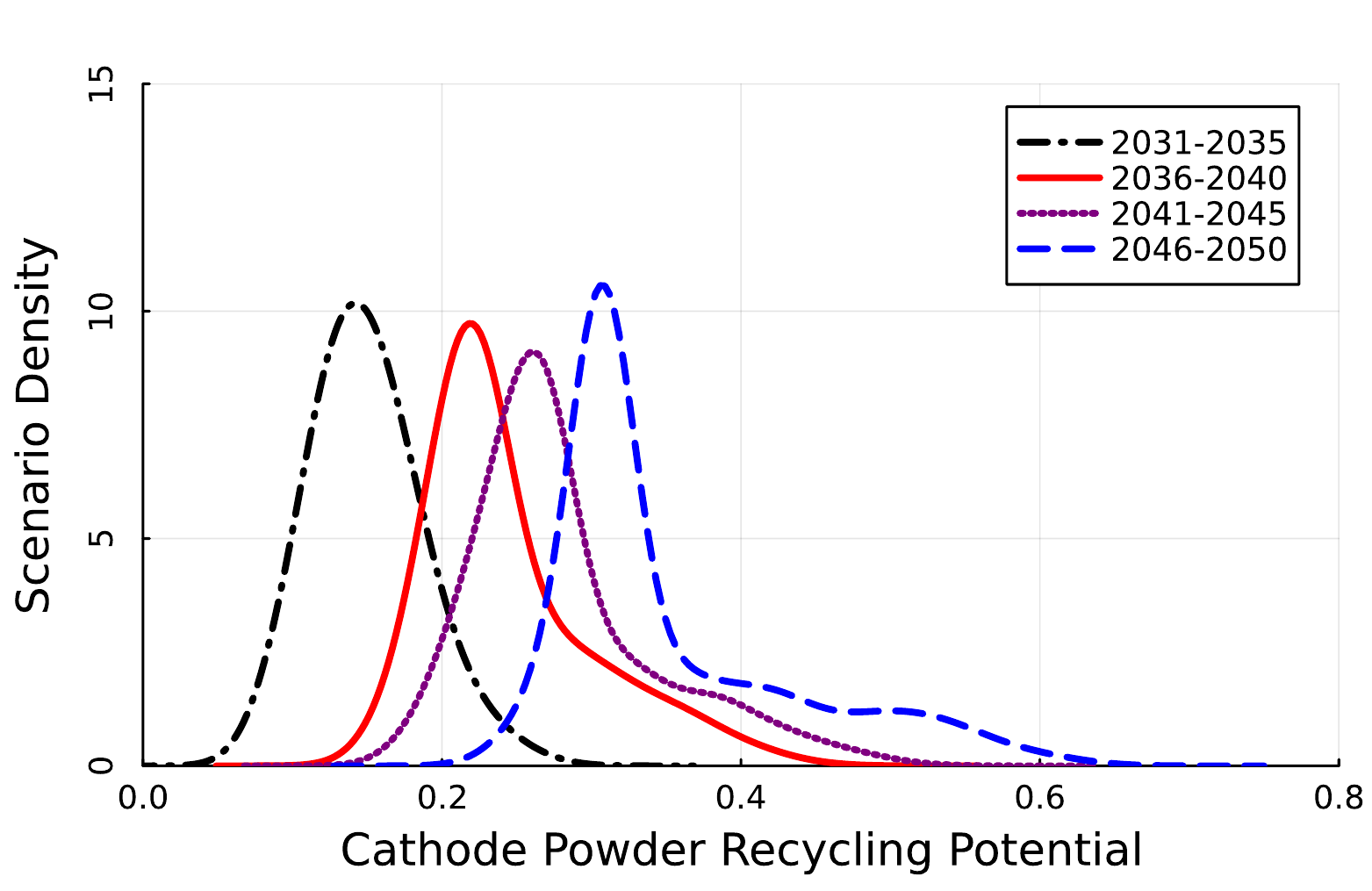}
    \captionsetup{font=footnotesize}
    \caption{Case 1 DR0}
    \end{subfigure}
    \newline
    \begin{subfigure}[t]{.48\textwidth} \centering
    \includegraphics[width=\linewidth]{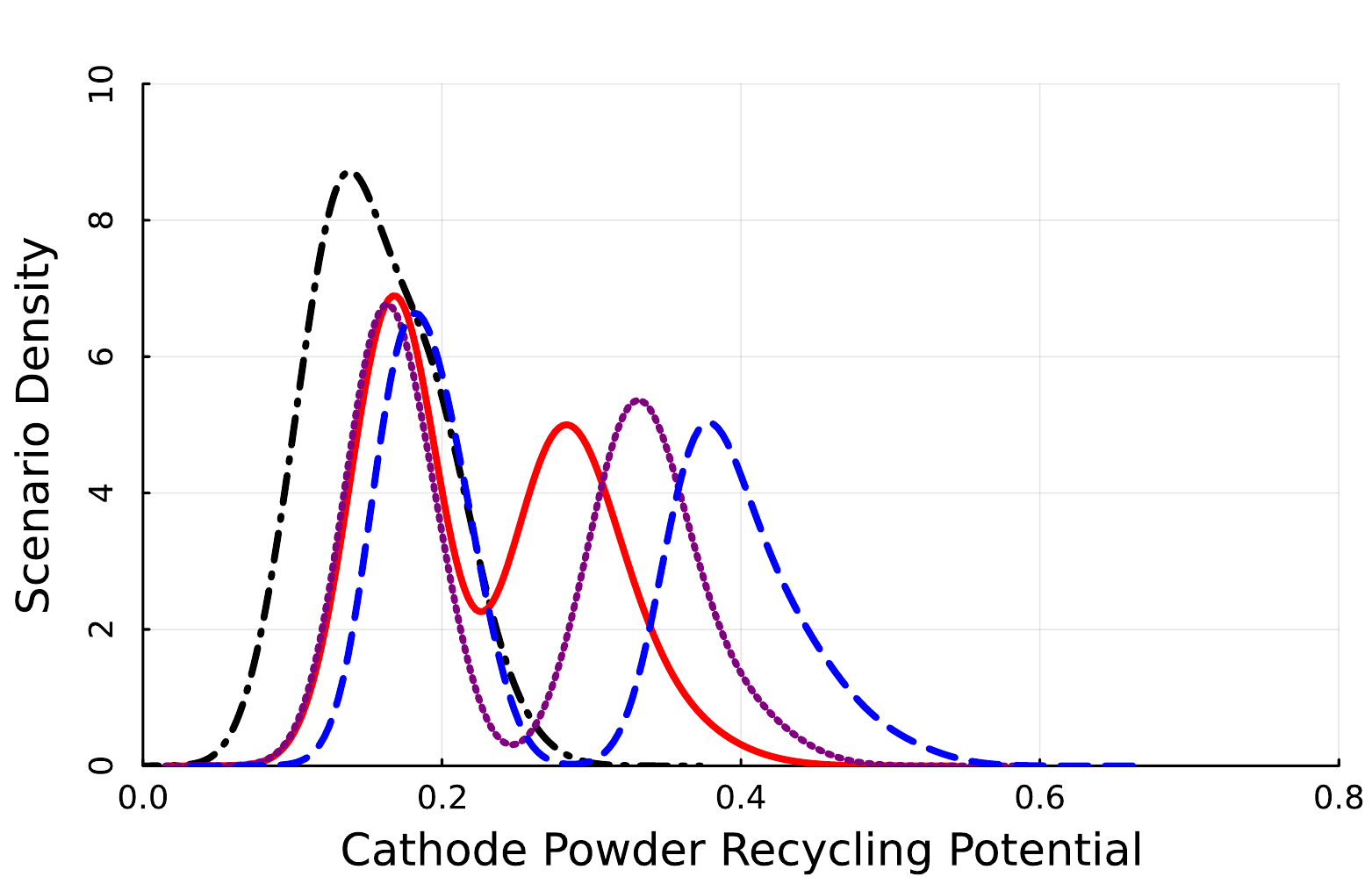}
    \captionsetup{font=footnotesize}
    \caption{Case 2 DR1}
    \end{subfigure}
    \hspace{0.2cm}
    \begin{subfigure}[t]{.48\textwidth} \centering
    \includegraphics[width=\linewidth]{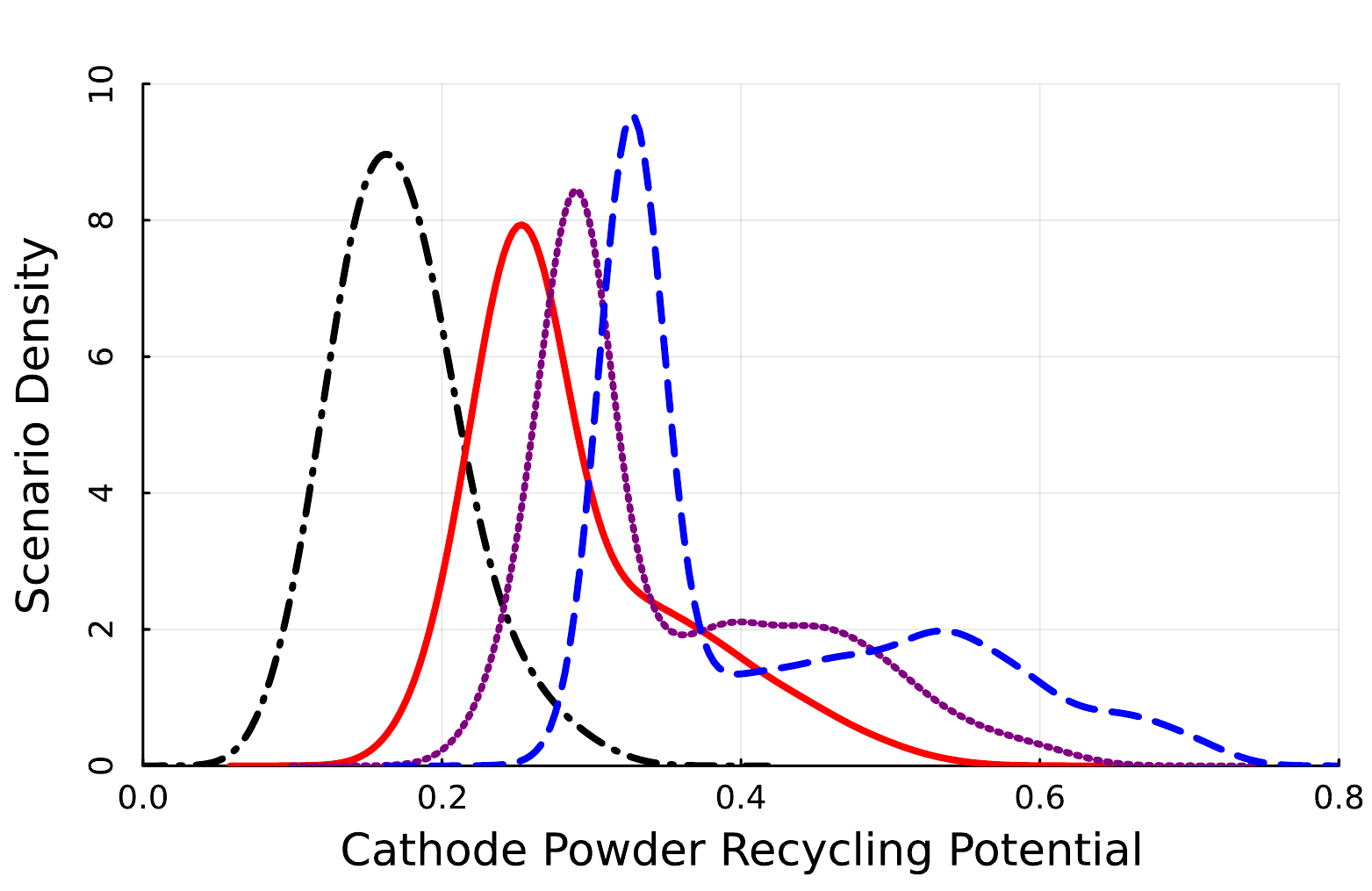}
    \captionsetup{font=footnotesize}
    \caption{Case 2 DR0}
    \end{subfigure}
    \caption{Kernel density estimation of the proportion of cathode powder used in new battery manufacturing that is produced via recycling (i.e., recycling potential of cathode powder) across operational scenarios, indexed by planning period.} \label{fig:recPropDensity}
\end{figure}
%%%%%%%%%%%%%%%%%% Figure %%%%%%%%%%%%%%%%%%

%%%%%%%%%%%%%%%%%% Figure %%%%%%%%%%%%%%%%%%
\begin{figure}[tp]
    \begin{subfigure}[t]{.48\textwidth} \centering
    \includegraphics[width=\linewidth]{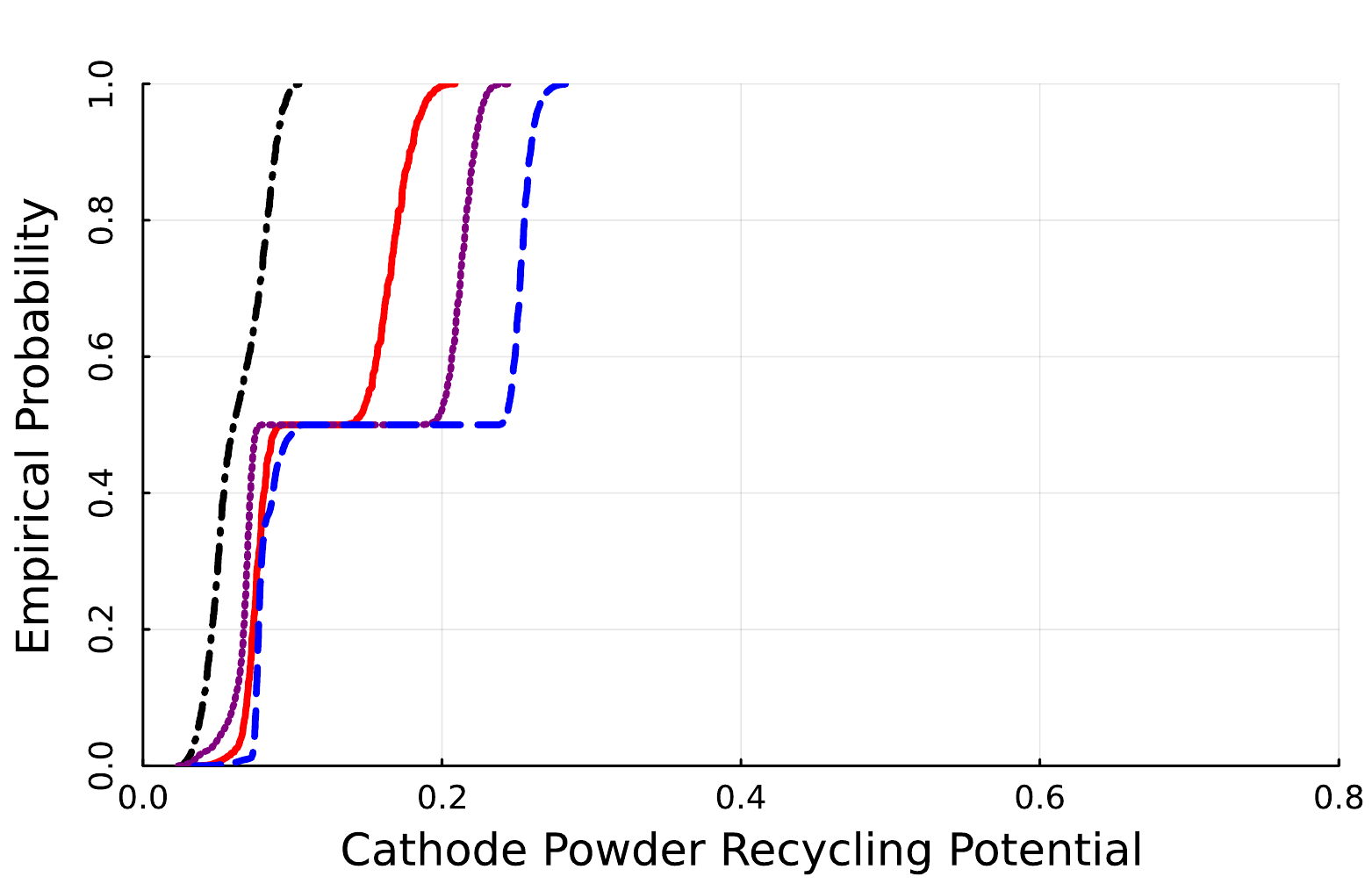}
    \captionsetup{font=footnotesize}
    \caption{Case 1 DR1}
    \end{subfigure}
    \hspace{0.2cm}
    \begin{subfigure}[t]{.48\textwidth} \centering
    \includegraphics[width=\linewidth]{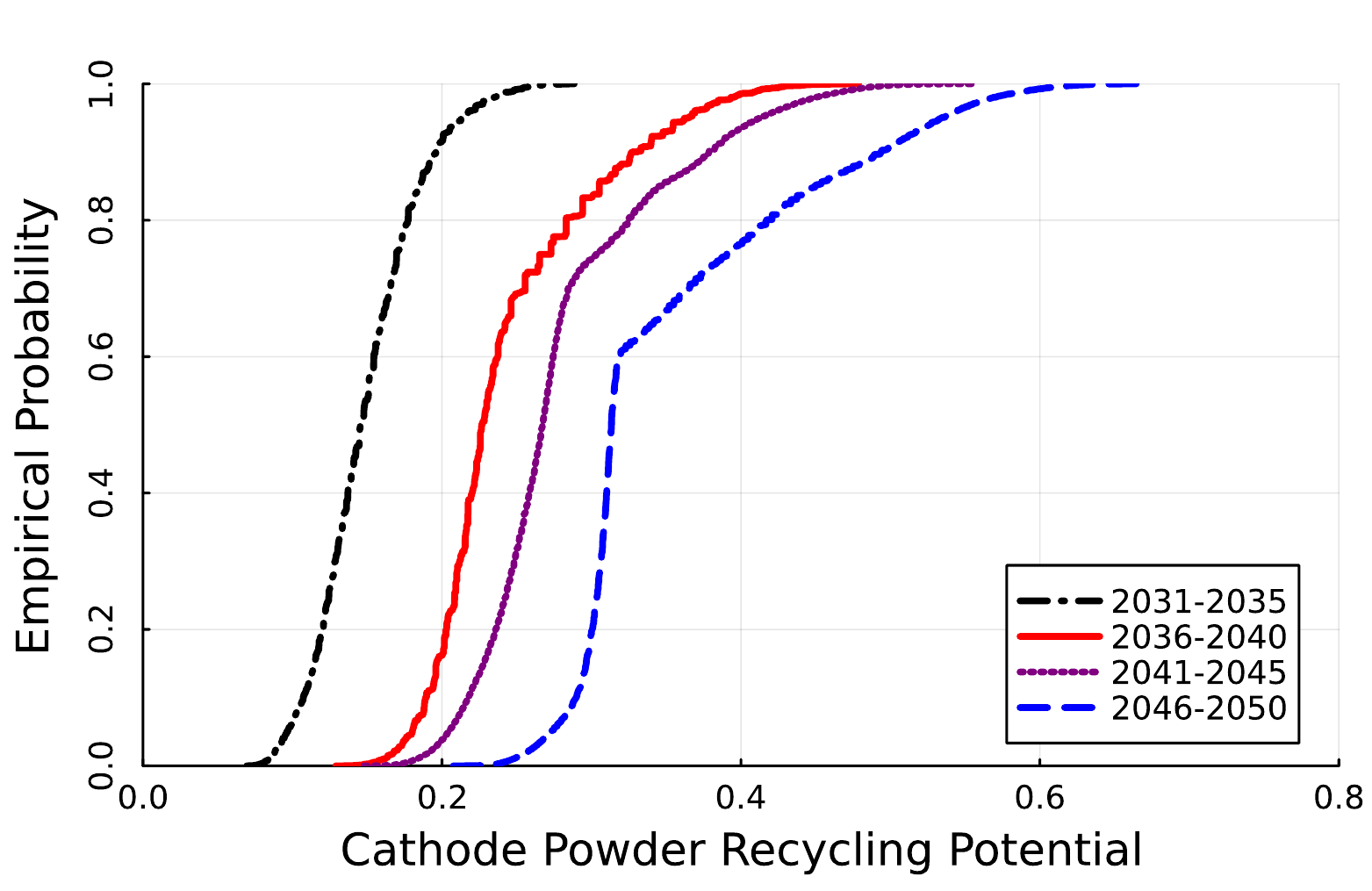}
    \captionsetup{font=footnotesize}
    \caption{Case 1 DR0}
    \end{subfigure}
    \newline
    \begin{subfigure}[t]{.48\textwidth} \centering
    \includegraphics[width=\linewidth]{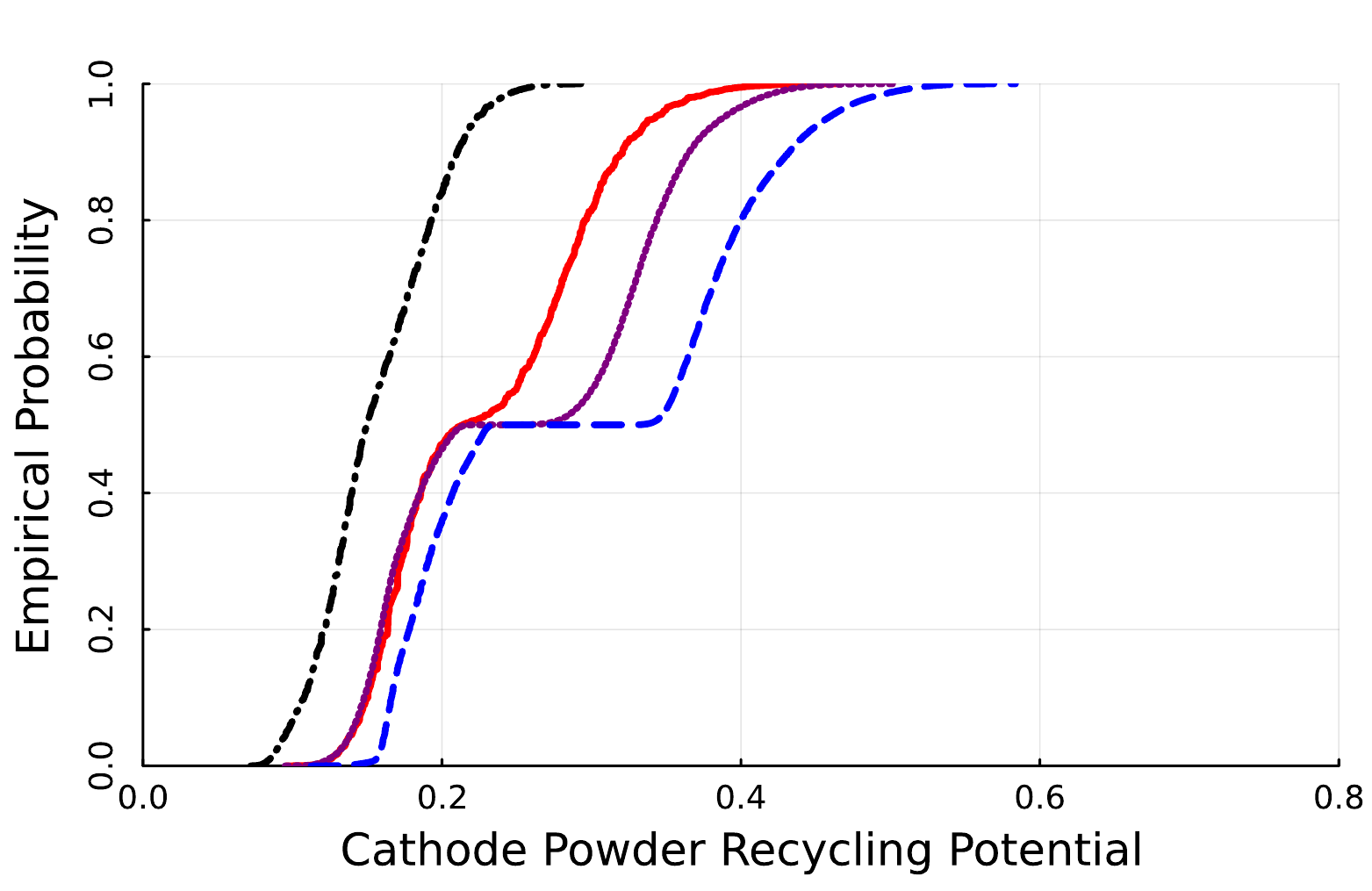}
    \captionsetup{font=footnotesize}
    \caption{Case 2 DR1}
    \end{subfigure}
    \hspace{0.2cm}
    \begin{subfigure}[t]{.48\textwidth} \centering
    \includegraphics[width=\linewidth]{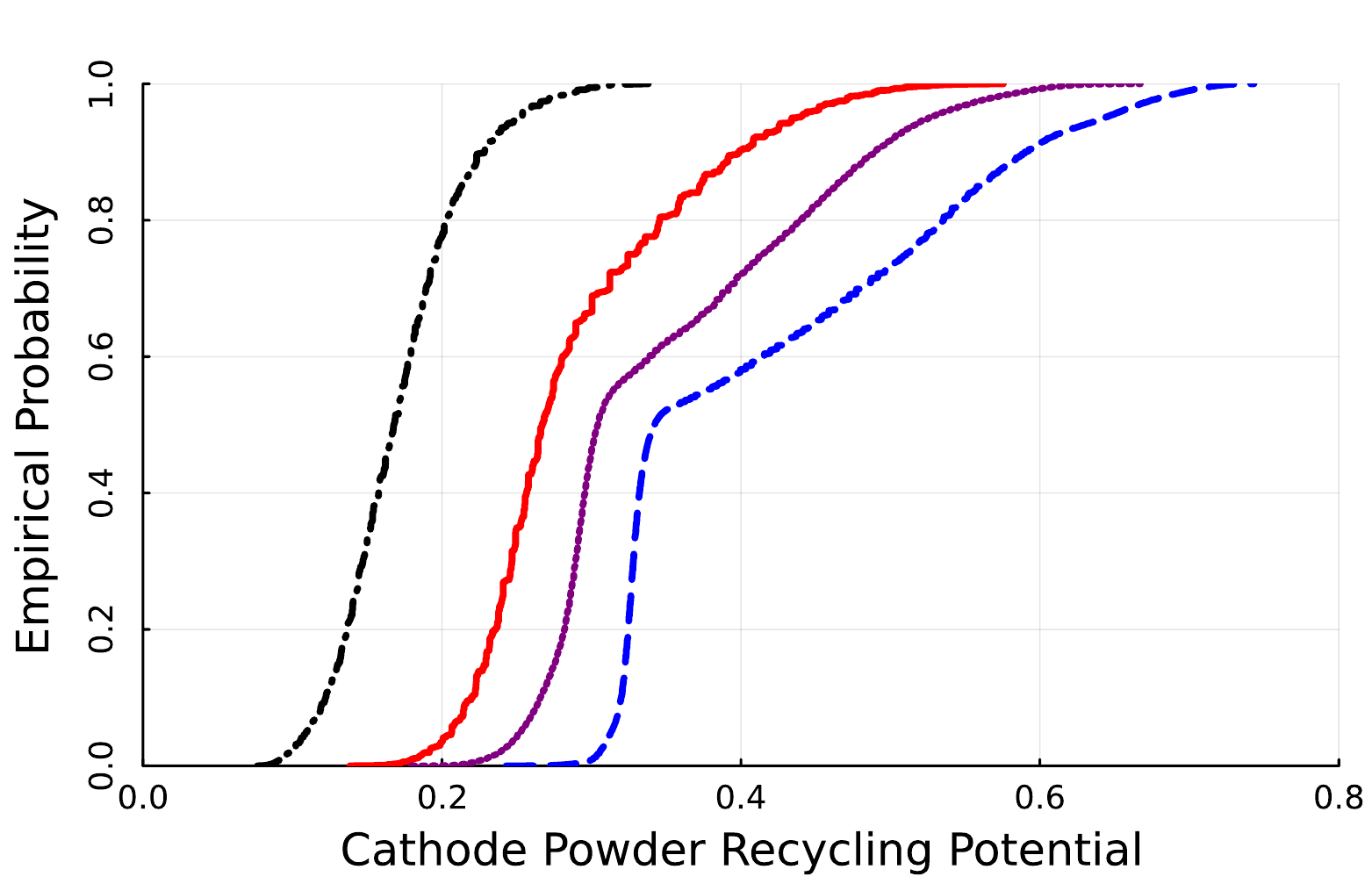}
    \captionsetup{font=footnotesize}
    \caption{Case 2 DR0}
    \end{subfigure}
    \caption{Empirical cumulative density function of the proportion of cathode powder used in new battery manufacturing that is produced via recycling (i.e., recycling potential of cathode powder) across operational scenarios, indexed by planning period.} \label{fig:recPropECDF}
\end{figure}
%%%%%%%%%%%%%%%%%% Figure %%%%%%%%%%%%%%%%%%

\subsubsection*{Closed-Loop Recycled Material Usage} We consider the proportion of cathode material used to manufacture new batteries that is sourced from recycled batteries, which is called the closed-loop \textit{recycling potential} for cathode material, defined by \cite{xu2020future}.  We compute the recycling potential for each scenario $\omega \in \Omega_S$ and planning period $l$ as 
\begin{equation}
    \label{eq:recyclingPotential}
    \frac{\sum_{z \in \mathcal{Z},\ k \in \mathcal{K}^{\text{CP}},\ t \in \mathcal{T}_l}\ x^{\text{INV,NB}}_{z,a_\omega(t),t,k}}{\sum_{z \in \mathcal{Z},\ k \in \mathcal{K}^{\text{CP}},\ t \in \mathcal{T}_l}\ \left ( x^{\text{INV,NB}}_{z,a_\omega(t),t,k} + x^{\text{NM,NB}}_{z,a_\omega(t),t,k} \right )}.
\end{equation} A recycling potential of 50\% means that, within the planning period, half of the cathode material used in new battery manufacturing comes from recycled sources.  

Figures~\ref{fig:recPropDensity}~and~\ref{fig:recPropECDF} visualize the recycling potential as a distribution over the operational scenarios.  Figure~\ref{fig:recPropDensity} shows the density of the recycling potential (i.e., a smoothed probability density function), and Figure~\ref{fig:recPropECDF} gives an empirical cumulative distribution function.  We observe that scenarios with higher prominence of LFP cathode powders have lower recycling potential than NCX scenarios.  Because LFP only contains small amounts of lithium and does not contain cobalt or manganese, it is inexpensive to manufacture from new materials.  This property reduces the cost effectiveness of recycling and gives the distributions a bimodal structure, where the first mode corresponds to LFP scenarios and the second to NCX scenarios.  

Furthermore, solutions that construct hydrometallurgical recycling capacity (i.e., Case 1 DR0 and both Case 2 variations) yield right-tailed distributions.  Recall that hydrometallurgical recycling recovers cathode precursors, while direct recycling recovers cathode powder.  Recycled precursors can be manufactured into any cathode powder chemistry, allowing these materials to be used in manufacturing even as battery demand shifts between chemistries.  Alternatively, cathode powder recovered from direct recycling can only be used to manufacture batteries of the same chemistry.  Especially in DR0 solutions, the prevalence of hydrometallurgical recycling allows flexibility in how recycled materials are used in new battery manufacturing and allows for high rates of recycled material usage in some scenarios.  This result emphasizes a hidden benefit of hydrometallurgical recycling over direct methods when considering variable cathode chemistry demand, even though minimum-cost solutions prefer direct recycling.  We quantify the value of this flexibility in Section~\ref{subsec:hydrovalue}.

Due to the low manufacturing cost of LFP powder, in the DR1 variations, the mode corresponding to LFP scenarios does not shift towards higher recycling potential as additional capacity becomes available in later planning periods.  In DR0 variations, the flexibility of hydrometallurgical recycling allows lithium recovered from LFP batteries to be used to manufacture cathode powders of other, more expensive chemistries.  This increases the recycling potential of the first mode in these variations, even though the expensive chemistries are less prevalent in LFP scenarios.

In the DR1 variations, the maximum level of recycling potential across scenarios approaches $30\%$ in Case~1 and $60\%$ in Case~2, and the expected recycling potential is $13\%$ in Case~1 and $23\%$ in Case~2.  Due to the increased investment in hydrometallurgical capacity in DR0 variations, the maximum recycling potential increases to $70\%$ in Case~1 and $80\%$ in Case~2, and the expected potential increases to $26\%$ in Case~1 and and $31\%$ in Case~2.

%%%%%%%%%%%%%%%%%% Figure %%%%%%%%%%%%%%%%%%
\begin{figure}[tp]
    \centering
    \includegraphics[width=.75\linewidth]{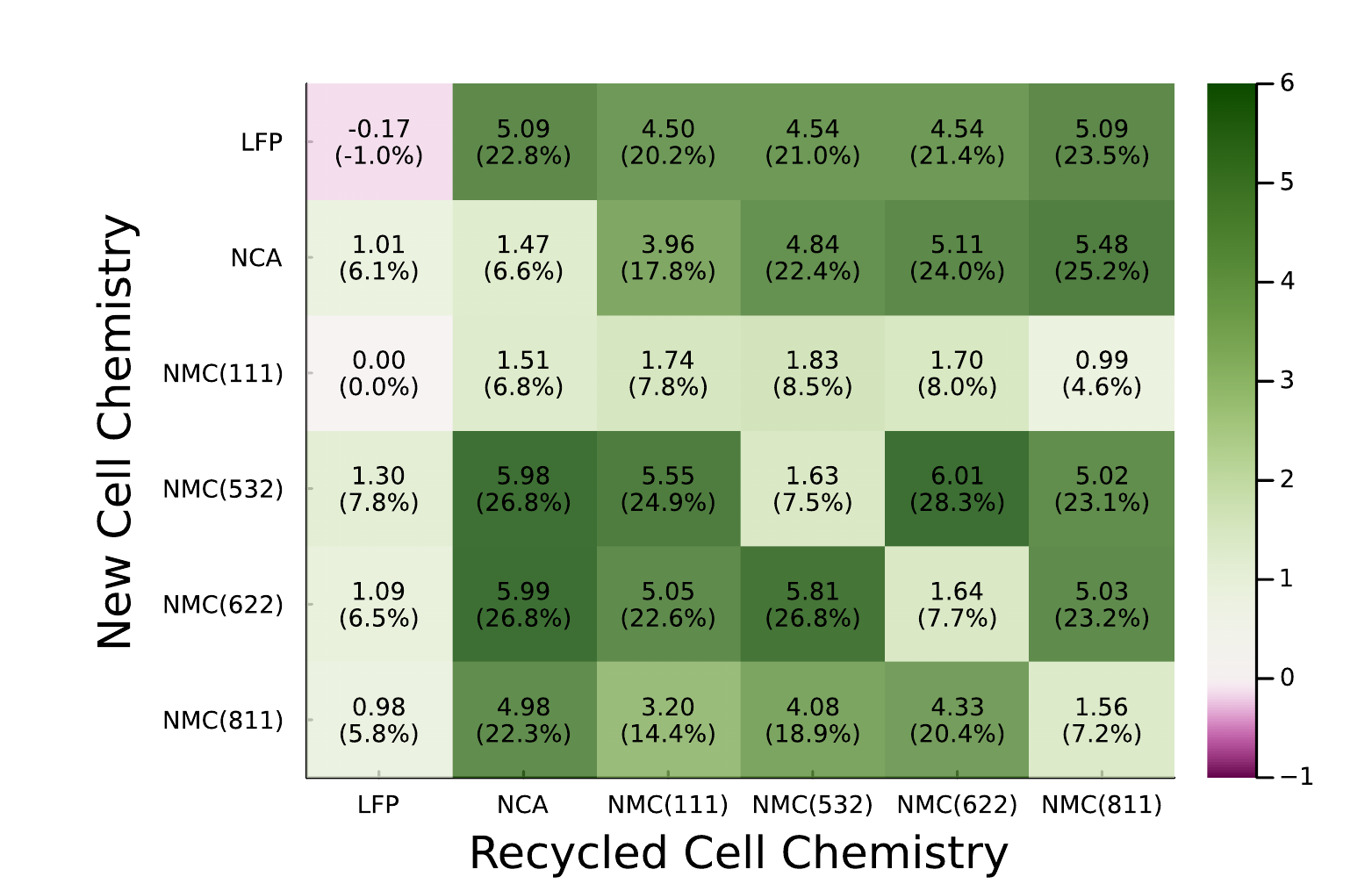}
    \caption{Incremental value (\$) of hydrometallurgical over direct recycling by cathode chemistry of new and retired batteries, when converting 1~kg of retired battery cells into 1~kg of new battery cells.  The magnitude of the incremental value, relative to the cost of manufacturing new battery cells with virgin materials, is given in parentheses.  Diagonal elements give the incremental value from the material recovery rate, and off-diagonal elements give the combined value from the recovery rate and manufacturing flexibility.}
    \label{fig:hydroValue}
\end{figure}
%%%%%%%%%%%%%%%%%% Figure %%%%%%%%%%%%%%%%%%

\subsection{The Value of Hydrometallurgical Recycling}
\label{subsec:hydrovalue}

Hydrometallurgical recycling recovers cathode precursors instead of cathode powder and therefore requires investment in cathode production to manufacture new batteries, which incurs additional cost relative to direct recycling.  For this reason, minimum-cost solutions invest primarily in direct recycling.  However, hydrometallurgical recycling has two benefits over direct recycling:  first, it recovers a higher proportion of the cathode material (98\%, compared to 90\%), and second, it yields recycled precursors that can be manufactured into any cathode powder chemistry.  The latter benefit is especially valuable when the distribution of cathode chemistry shifts over time, causing a mismatch in the chemistries of new and retired batteries.  In this section, we explore the value of these benefits.

To quantify this value, we compute the cost of manufacturing 1~kg of new battery cells of chemistry $i^{\mathrm{MAN}}$, given that 1~kg of retired battery cells of chemistry $i^{\mathrm{REC}}$ are available for recycling.  We use deterministic material costs from EverBatt and assume that recycled materials are only used for manufacturing new batteries and cannot be sold to an external market.  We calculate the manufacturing costs when direct recycling capacity is available and when hydrometallurgical recycling capacity is available; the incremental value of hydrometallurgical capacity is the difference between these costs.\textsuperscript{\endnote{Cathode production capacity and hydrometallurgical recycling capacity are made available together; the incremental value represents the additional value of both hydrometallurgical recycling and cathode production capacity.}}  This incremental value represents the per-unit operational savings of manufacturing batteries of chemistry $i^{\mathrm{MAN}}$ using recycled materials, where the savings are obtained by recycling batteries of chemistry $i^{\mathrm{REC}}$ with hydrometallurgical recycling instead of direct recycling.   We only consider differences in operational costs and ignore differences in planning costs (i.e., capacity investment).

Figure~\ref{fig:hydroValue} shows the incremental value of hydrometallurgical capacity over direct recycling capacity for each pair of chemistries $(i^{\mathrm{MAN}},i^{\mathrm{REC}}) \in \mathcal{I}^2$ and reports the relative magnitude of the value as a proportion of the base manufacturing cost without recycling (i.e., manufacturing from virgin materials).  The diagonal elements, where the same chemistry is recycled and manufactured, provide the value of higher recovery rates for hydrometallurgical recycling.  For most chemistries, this value is between $\$1.50$ and $\$1.75$ per kg of retired and new batteries, or 6\% to 8\% of the base manufacturing cost.  However, the cost of the required cathode production capacity exceeds this incremental value, so direct recycling is more cost effective for matching chemistries.  The off-diagonal elements of Figure~\ref{fig:hydroValue} demonstrate the value of recycling flexibility, as they include the additional savings of using recycled cathode precursors to manufacture batteries of a different chemistry than the recycled batteries.  This type of manufacturing is not possible with direct recycling, which directly recovers cathode powder of the recycled battery chemistry.  For most chemistry pairs, the total incremental value is between $\$3.00$ and $\$6.00$ per kg of retired and new batteries, or 15\% to 30\% of the base manufacturing cost.  Then, the value of chemistry flexibility alone (without the value of the higher material recovery rate) is between $\$1.25$ and $\$4.50$ per kg.  

Recycling LFP batteries has lower incremental value than other chemistries.  LFP contains less lithium by mass than other cathode types and does not contain the nickel, cobalt, manganese, or aluminum found in NCA and NMC.  For this reason, the cost savings of manufacturing NCA or NMC batteries after recycling LFP is relatively low.  However, lithium recovered from NCA or NMC recycling can be used to manufacture LFP batteries, so hydrometallurgical recycling provides comparable incremental value ($\$4.50$ to $\$5.00$ per kg) when manufacturing LFP after recycling other chemistries.  Recycling LFP to manufacture LFP batteries has a negative incremental value, as the lower proportion of lithium in LFP reduces the incremental value of the higher recovery rate from hydrometallurgical recycling.  This reduced incremental value is dominated by the additional operational costs of cathode production, meaning that direct recycling incurs lower total operational costs than hydrometallurgical recycling for the LFP-LFP chemistry pair.

\subsection{Policy Analysis: Capacity Grants or Production Credits?}
\label{subsec:policy}

%%%%%%%%%%%%%%%%%% Figure %%%%%%%%%%%%%%%%%%
\begin{figure}[tp]
    \begin{subfigure}[t]{.48\textwidth} \centering
    \includegraphics[width=\linewidth]{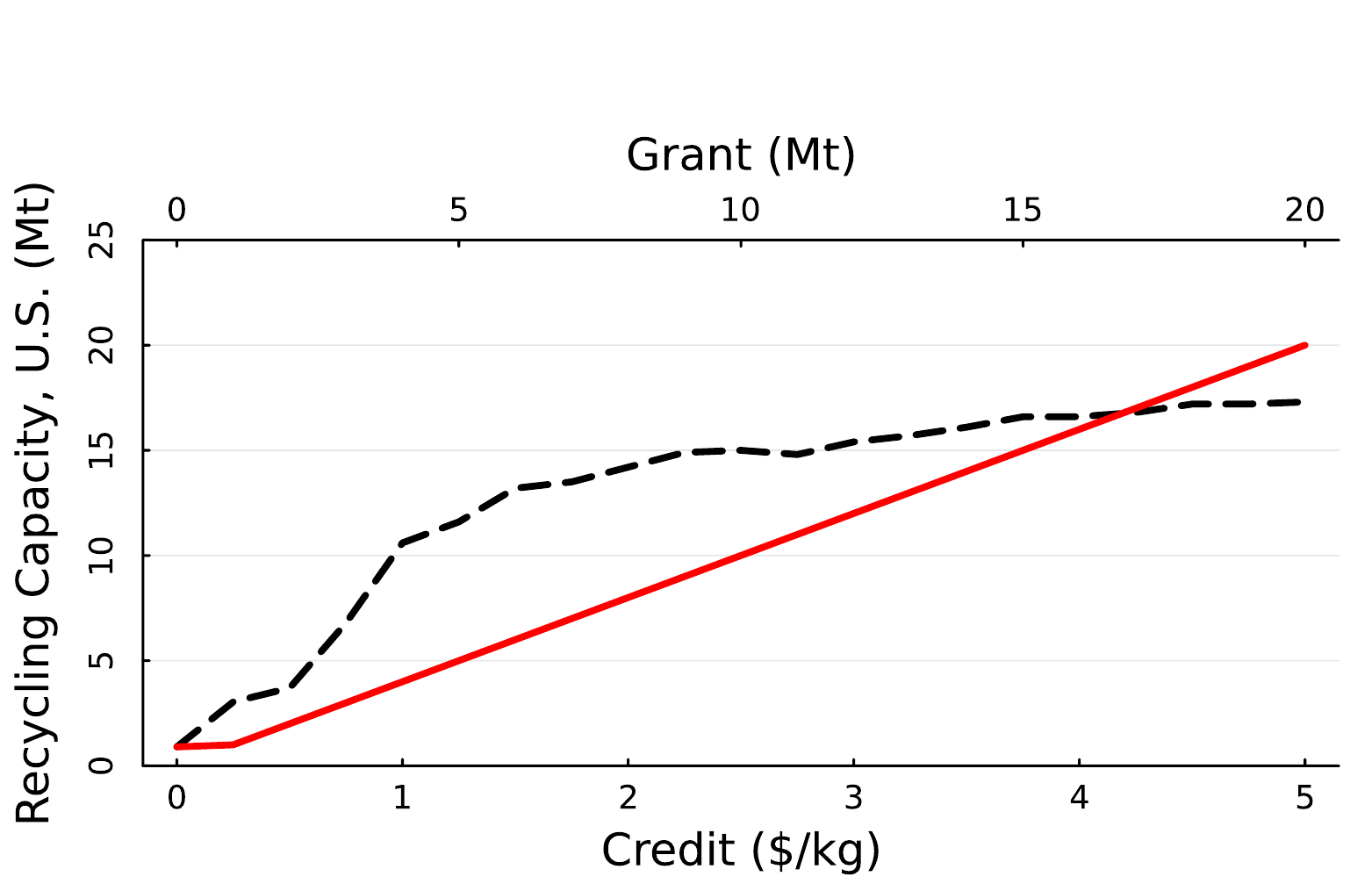}
    \captionsetup{font=footnotesize}
    \caption{U.S. recycling capacity in 2050}
    \label{fig:policyUSCap}
    \end{subfigure}
    \hspace{0.2cm}
    \begin{subfigure}[t]{.48\textwidth} \centering
    \includegraphics[width=\linewidth]{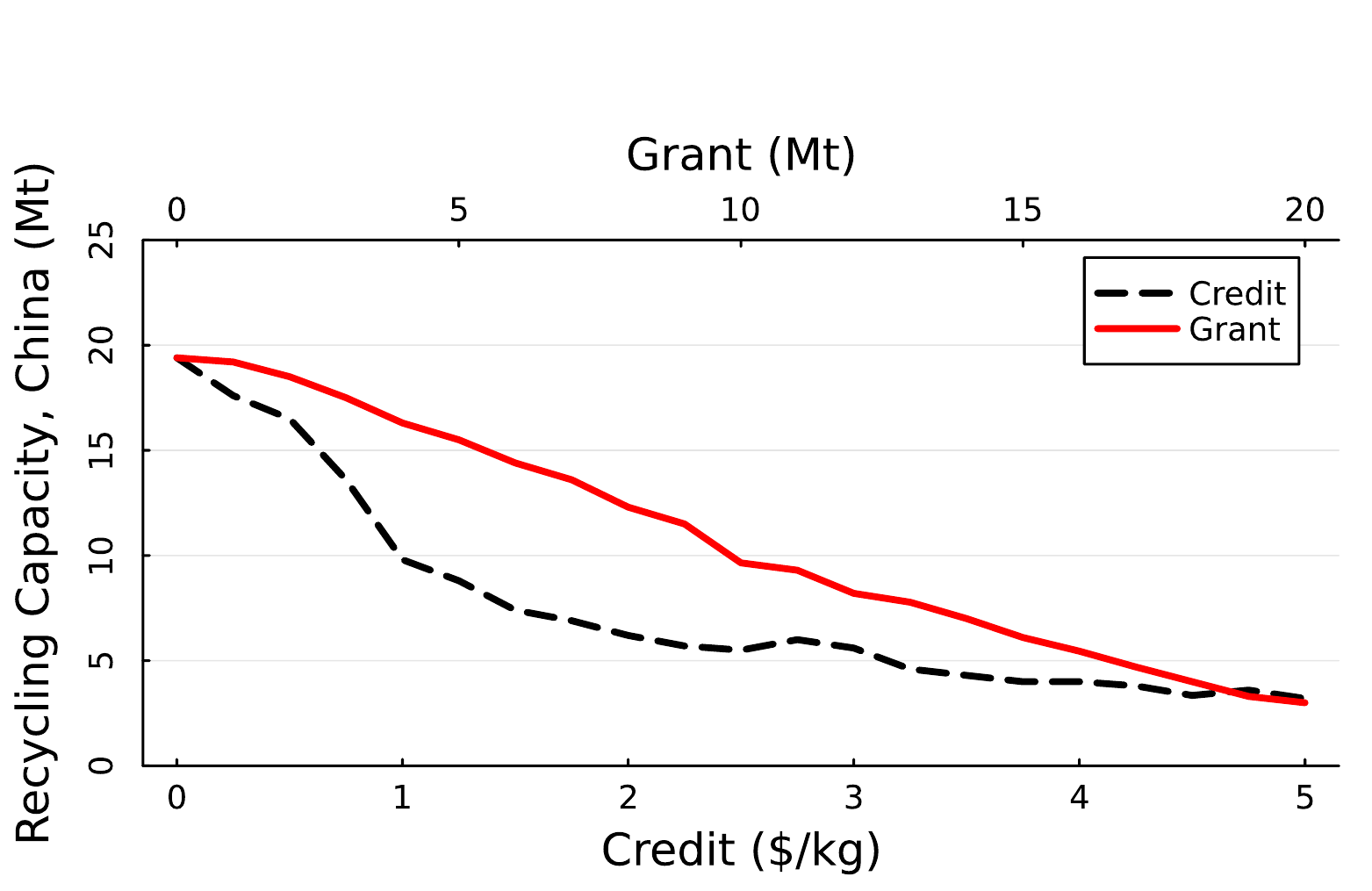}
    \captionsetup{font=footnotesize}
    \caption{China recycling capacity in 2050}
    \label{fig:policyCNCap}
    \end{subfigure}
    \newline
    \begin{subfigure}[t]{.48\textwidth} \centering
    \includegraphics[width=\linewidth]{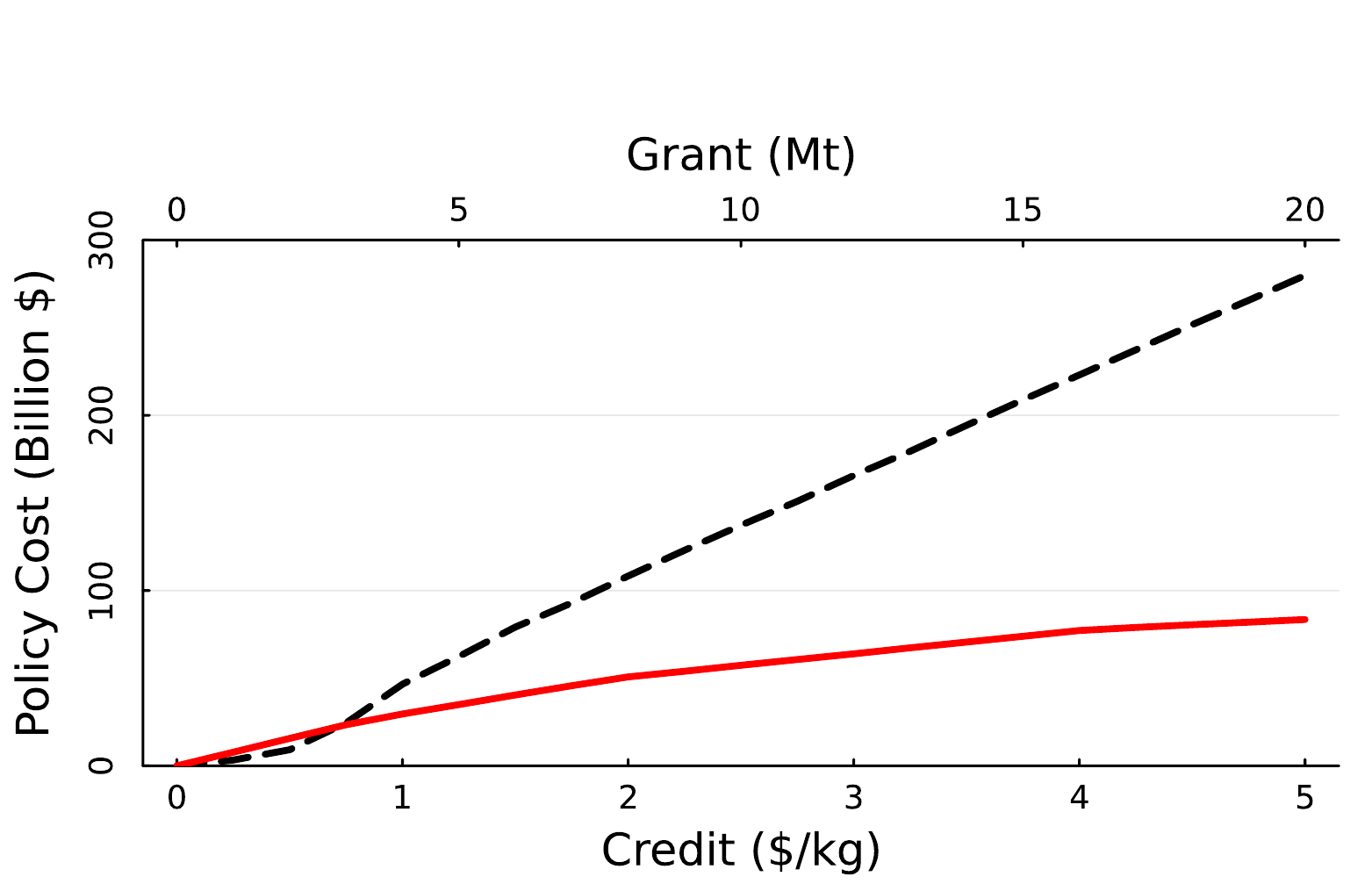}
    \captionsetup{font=footnotesize}
    \caption{Expected policy cost}
    \label{fig:policyCost}
    \end{subfigure}
    \hspace{0.2cm}
    \begin{subfigure}[t]{.48\textwidth} \centering
    \includegraphics[width=\linewidth]{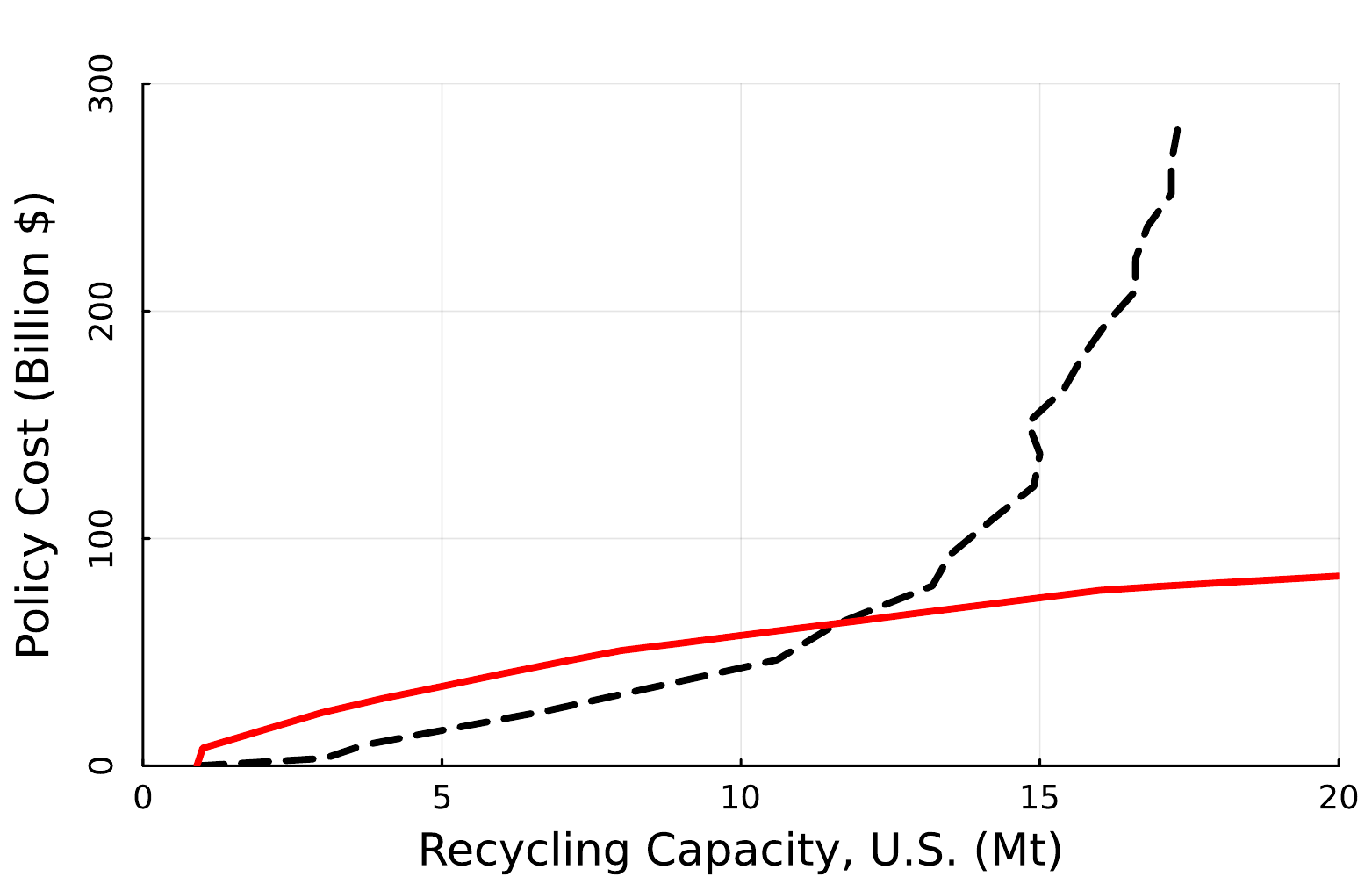}
    \captionsetup{font=footnotesize}
    \caption{Implicit relationship between U.S. recycling capacity in 2050 and expected policy cost}
    \label{fig:policyImplicitCost}
    \end{subfigure}
    \caption{Impact of grant and credit policy level on installed recycling capacity in 2050 (a,b) and expected policy cost (c).  (d) shows implicitly how expenditure on each policy impacts installed U.S. recycling capacity in 2050.} \label{fig:policy}
\end{figure}
%%%%%%%%%%%%%%%%%% Figure %%%%%%%%%%%%%%%%%%

%%%%%%%%%%%%%%%%%% Figure %%%%%%%%%%%%%%%%%%
\begin{table}[tp]
    \centering
    \begin{tabular}{c|c||c|c|c|c|c|c}
         & & \multicolumn{2}{c|}{2050 Recycling (Mt)} & \multicolumn{3}{c|}{Recycling Potential} & Policy Cost or Revenue\\
        Policy & Level & \ \ U.S.\ \ & China & U.S. & China & Aggregate & (billion \$)\\
        \hline
        --- & --- & 0.9 & 19.4 & 1.9\% & 26.7\% & 22.9\% & ---\\
        \hline
        \multirow{3}{*}{Tariff} & 20\% & 1.1 & 19.1 & 1.9\% & 26.0\% & 22.4\% & 0.0\\
         & 50\% & 1.3 & 19.7 & 3.2\% & 26.0\% & 22.7\% & 0.0\\
         & 100\% & 0.9 & 19.7 & 2.4\% & 26.1\% & 22.6\% & 0.0\\
         \hline
        \multirow{3}{*}{Grant (Mt)} & 5.0 & 5.0 & 15.5 & 11.6\% & 22.9\% & 22.6\% & 34.9\\
         & 10.0 & 10.0 & 9.6 & 13.8\% & 21.5\% & 22.4\% & 57.4\\
         & 20.0 & 20.0 & 3.0 & 20.2\% & 18.2\% & 21.6\% & 83.5\\
         \hline
        \multirow{3}{*}{Credit (\$/kg)} & 1 & 10.6 & 9.8 & 24.4\% & 14.0\% & 21.3\% & 46.6\\
         & 3 & 15.4 & 5.6 & 24.1\% & 13.3\% & 22.0\% & 165.6\\
         & 5 & 17.3 & 3.2 & 24.9\% & 12.0\% & 21.1\% & 279.8
    \end{tabular}
    \caption{Impact of tariffs, grants, and credits on recycling metrics under the optimal solution in Case 2.  Metrics include installed recycling capacity in 2050, expected recycling potential \eqref{eq:recyclingPotential}, and the expected cost or revenue incurred by the policymaker.}
    \label{tbl:policyAnalysis}
\end{table}
%%%%%%%%%%%%%%%%%% Figure %%%%%%%%%%%%%%%%%%

We consider the impact of three fundamental policy instruments on the supply chain: \emph{tariffs} on the import of cathode material, \emph{capacity grants} for the installation of recycling capacity, and \emph{production credits} paid for recycling throughput.  We  provide insights on which instrument most effectively shapes U.S. domestic recycling capability.  

The U.S. Department of Energy has provided grants for recycling facilities, including $\$6$ billion of grants announced in 2022 \citep{doe2025grant, doe2025grant2}, and the Inflation Reduction Act of 2022 established credits for the domestic manufacture of critical minerals and battery components, which includes domestic recycling of EV batteries \citep{mehdi2023ira,irs2025credit}.  We study the implementation of these policies in the U.S. under Case~2, with zones for the U.S. and China, and compare their impacts against the imposition of tariffs on U.S. imports of cathode material from China.  Tariffs are modeled by imposing a cost on cathode material imports, added to the operational cost functions $C^{\text{OP}}_{\omega,t}$:
\begin{equation}
    -p^{\mathrm{TAR}} \sum_{k \in \mathcal{K}^{\mathrm{CP}}} c^{\mathrm{NB,NM}}_{\omega,t,k} x^{\mathrm{TR,RM}}_{\omega,\text{China},\text{U.S.},t,k},
\end{equation}
where $p^{\mathrm{TAR}}$ is the tariff rate as a percentage of new cathode material cost.
Capacity grants are modeled as a lower bound $p^{\text{GRANT}}$ on installed recycling capacity in the U.S. by 2031, i.e., the start of the third planning period.  That is, we enforce in model \eqref{eq:reformulatedP} the following constraint:
\begin{equation}
    \label{eq:grantLB}
    \sum_{j \in \mathcal{J}} u^{\text{REC}} y^{\text{REC}}_{\text{U.S.},3,j} + y^{\text{REC}}_{\text{U.S.},3,j,+} \geq p^{\text{GRANT}}.
\end{equation}
Costs associated with the minimum capacity $p^{\text{GRANT}}$ are covered by the grant and are ignored by the recycler.  Additional capacity construction incurs additional cost.  Credits are modeled by adding a credit revenue term to the operational cost functions $C^{\text{OP}}_{\omega,t}$:
\begin{equation}
    -p^{\text{CRED}}  \sum_{j \in \mathcal{J},\ i \in \mathcal{I}} x^{\text{RB,RM}}_{\omega,\text{U.S.},t,i,j},
\end{equation}
where $p^{\text{CRED}}$ gives the recycling credit value per unit mass of throughput.  For reference, the expected undiscounted recycling cost (including planning and operational costs) per unit mass of throughput in Case~2 is $\$1.68$/kg.

We consider several metrics to measure the impact of the policies:
\begin{enumerate}[noitemsep, nolistsep, parsep=0pt]
    \item the installed capacity in each zone in 2050;
    \item the \textit{policy cost or revenue}: planning costs incurred by the grant capacity after 2031 and expected credit costs, or expected revenue received from tariff payments; and
    \item the expected recycling potential for cathode powder \eqref{eq:recyclingPotential}, in aggregate and for materials handled exclusively within a zone.
\end{enumerate}
Figure~\ref{fig:policy} shows how the installed capacity and policy cost vary with different levels of the grant $p^{\text{GRANT}}$ and credit $p^{\text{CRED}}$, and Figure~\ref{fig:policyImplicitCost} shows the implicit relationship between the policy cost at different levels of each policy and the amount of installed capacity in the U.S. in 2050.  That is, Figure~\ref{fig:policyImplicitCost} demonstrates how much domestic recycling capacity results from a given amount of expenditure on each policy.  For a few select levels of each policy, Table~\ref{tbl:policyAnalysis} compares the corresponding value for each metric against a baseline solution without intervention.  

In Table~\ref{tbl:policyAnalysis}, we observe that grants and credits effectively alter the recycler behavior, encouraging larger installations of recycling capacity and increasing the use of recycled material for battery manufacturing in the U.S.  However, tariffs do not substantially impact the optimal solutions, beyond minor variations due to the optimality tolerance.  As the supply of retired batteries lags the demand for new batteries, it is only possible to meet a portion of demand with recycled battery materials.  Under the import tariff, recycled cathode material from facilities in China is used only to manufacture new batteries in China and is never imported to the U.S.  For this reason, tariffs are ineffective at stimulating domestic recycling capacity investment and producing government revenue, without the simultaneous implementation of other policies like minimum recycled material requirements in new batteries.

Figures~\ref{fig:policyUSCap}~and~\ref{fig:policyCNCap} reinforce the conclusion that, in most cases, both grants and credits shift installed recycling capacity from China to the U.S. without drastically changing the total amount of constructed capacity.  However, for a large capacity grants (e.g., $p^{\text{GRANT}} = 20$ Mt), the total recycling capacity may increase sizably, by up to 15\%.  The capacity lower bound \eqref{eq:grantLB} for grants is typically binding, meaning that U.S. capacity grows linearly as a function of the grant capacity.  Alternatively, credits exhibit sublinear growth, so, for large credits, increasing the credit amount does not drastically increase the U.S. recycling capacity.  In Figure~\ref{fig:policyCost} we observe that the marginal policy cost of grant capacity is decreasing, suggesting that less expensive recycling capacity is installed to meet the capacity requirement from large grants.  The expected cost of credits grows superlinearly at low ($<$ \$1/kg) credit levels, then increases linearly. Figure~\ref{fig:policyImplicitCost} shows that, at the same expected cost, credits encourage greater investment in U.S. recycling capacity than grants when the recycling capacity and expected policy costs are low ($< 12$ Mt of U.S. capacity and $< \$65$ billion in policy costs).  At higher capacities and expected policy costs, grants are more cost-effective for developing U.S. capacity.

Table~\ref{tbl:policyAnalysis} also shows how operational recycling decisions respond to the policies.  We find that both grants and credits encourage the use of domestically recycled materials for manufacturing within the U.S. by increasing the recycling potential in the U.S. and causing a corresponding decrease in China.  Credit-based policies encourage higher rates of recycling in the U.S. than grant-based policies.  However, credits slightly decrease the aggregate recycling potential compared to grants and the baseline, likely because they encourage recycling even at times when recycled materials cannot be used effectively in new battery manufacturing. 

Our model of investment behavior under grants and credits relies on a number of assumptions on the design and implementation of the policies.  First, capacity grants are agnostic to the recycling type and cost of the facility, and instead are paid to cover a fixed amount of capacity of any recycling process.  Similarly, total production credits paid to a facility are uncapped and potentially subsidize profitable facilities.  Realistic implementations of these policies may be more nuanced in how payments are awarded.  Second, the supply of retired batteries and demand for new batteries are independent of the policy.  In practice, credits and grants to the battery manufacturing supply chain could reduce EV costs, increasing demand for new batteries.  Production credits (and, similarly, the existence of recycling capacity) increase the value of retired batteries, encouraging earlier retirement of batteries and increasing the retired battery supply.  Finally, policies are not implemented in isolation: grants and credits may be offered simultaneously or in conjunction with other policies, including recycling mandates for retired batteries and recycled material requirements for new batteries.

\subsection{The Role of Uncertainty}
\label{subsec:uncertainty}

In this section, we analyze how firms respond to the amount of uncertainty in the stochastic parameters.  We vary the weight placed on the tails of the EV stock and material cost distributions and evaluate the impact on the firm's optimal investment and operational decisions.  Let parameter $\alpha \in [-1,1]$ denote the level of certainty, where smaller values represent less certainty and higher values represent more certainty.  We update the scenario probabilities by 
\begin{equation}
    p^{\mathrm{EV}}_{\xi}(\alpha) = \begin{cases}
       (1 - |\alpha|) p^{\mathrm{EV}}_{\xi} + \frac{|\alpha|}{n_d} & \quad \alpha < 0,\\
       (1 - |\alpha|) p^{\mathrm{EV}}_{\xi} + \mathbb{I}[\xi = \frac{n_d + 1}{2}]|\alpha| & \quad \alpha \geq 0,
    \end{cases}
\end{equation}
for the case where $n_d$ is odd.  The probabilities $p^{\mathrm{COST}}_{\zeta}(\alpha)$ are defined analogously.  The case $\alpha < 0$ indicates less certainty than the base scenario set by adding a uniform probability of $|\alpha|/n_d$ to each scenario, increasing the probability assigned to unlikely scenarios.  The case $\alpha \geq 0$ indicates greater certainty than the base scenario set by adding a probability of $|\alpha|$ to the ``middle'' scenario ($\xi = \frac{n_d + 1}{2}$), which is the average scenario as constructed in Section~\ref{sec:data}.  This decreases the weight on the tail scenarios.  When $\alpha$ is close to $-1$, the scenario probabilities are nearly uniform; when $\alpha$ is close to $1$, the distribution is nearly deterministic, with a point mass at the mean scenario.  For even $n_d$, a probability of $|\alpha|/2$ is added to the two ``middle'' scenarios when $\alpha \geq 0$.

%%%%%%%%%%%%%%%%%% Figure %%%%%%%%%%%%%%%%%%
\begin{figure}[tp]
    \begin{subfigure}[t]{.48\textwidth}
        \centering
        \includegraphics[width=\linewidth]{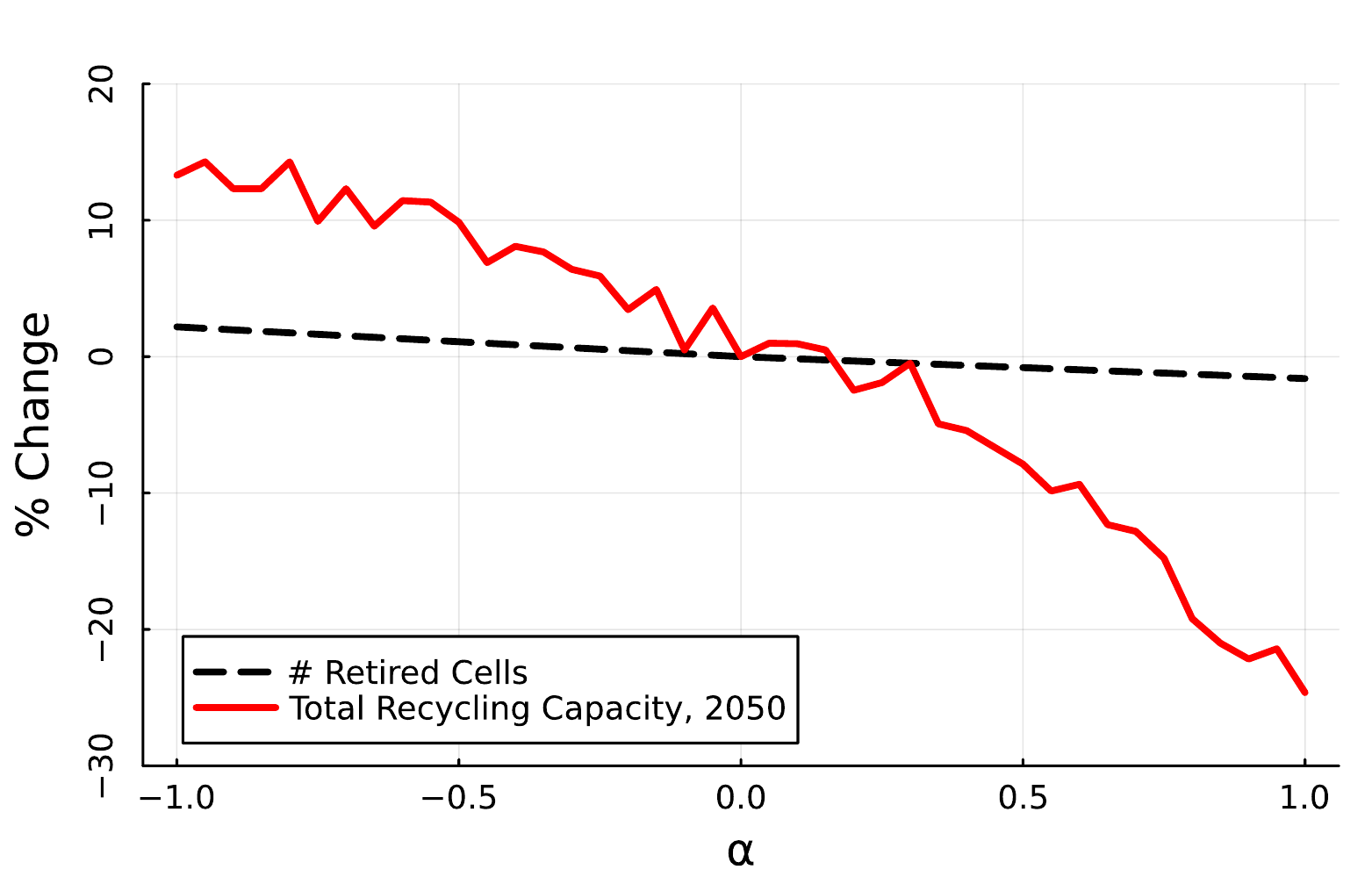}
        \captionsetup{font=footnotesize}
        \caption{}
        \label{fig:uncertaintyCap}
    \end{subfigure}
    \hspace{0.2cm}
    \begin{subfigure}[t]{.48\textwidth}
        \centering
        \includegraphics[width=\linewidth]{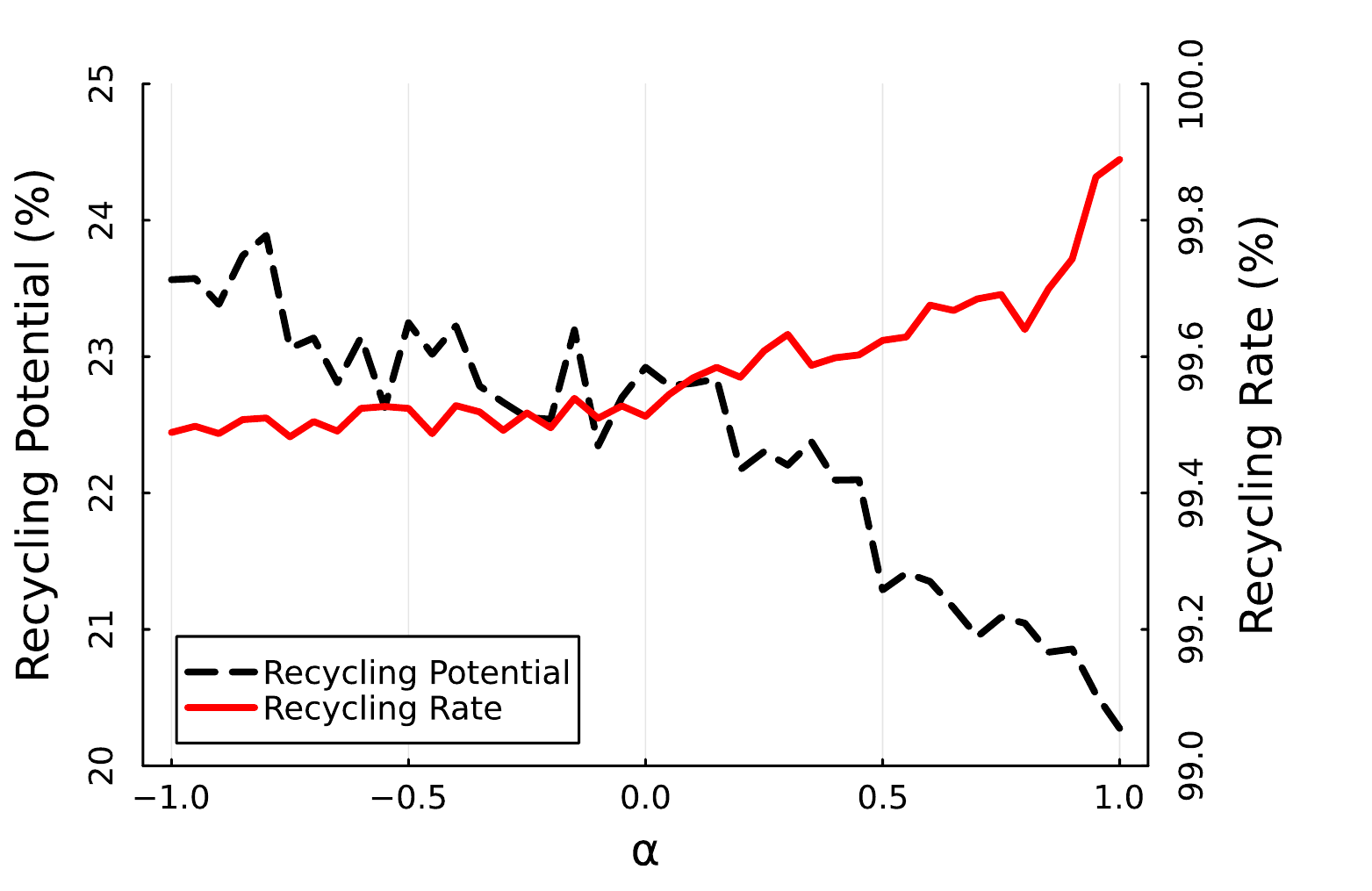}
        \captionsetup{font=footnotesize}
        \caption{}
        \label{fig:uncertaintyRates}
    \end{subfigure}
    \caption{Response of optimal solutions to the level of certainty $\alpha$.  (a) Change in total recycling capacity in 2050 and expected number of retired battery cells, relative to $\alpha = 0$.  (b) Expected recycling potential for cathode material and percentage of retired battery cells that are recycled.} \label{fig:uncertainty}
\end{figure}
%%%%%%%%%%%%%%%%%% Figure %%%%%%%%%%%%%%%%%%

Figure~\ref{fig:uncertainty} shows the response of the optimal solution as the certainty level $\alpha$ varies.  In Figure~\ref{fig:uncertaintyCap}, the total installed recycling capacity in 2050 decreases quickly as certainty increases.  At high certainty ($\alpha = 1$), recycling capacity decreases by up to $25\%$ compared to the nominal solution, while capacity is up to $13\%$ higher at low certainty ($\alpha = -1$). These results suggest that investment in recycling capacity is motivated by extremes of the EV stock and material cost distributions, specifically, by scenarios where many retired batteries are available and material costs are high.  As more uncertainty leads to greater investment, recycling capacity serves as a hedge against high metal prices and as speculation on enthusiastic EV adoption and large retired battery inventories.  These benefits outweigh the costs of investing in capacity that is underutilized when retired batteries are not available or material costs are low.

We note that changes in parameter $\alpha$ slightly bias the expected behavior of retired battery supply and new battery demand; the supply of retired cells varies by up to $2\%$ from the nominal case (see Figure~\ref{fig:uncertaintyCap}).  The relative magnitude of this variation is much smaller than that of the changes in recycling capacity, suggesting that the differences in the optimal solution are not primarily motivated by the trend in expected battery supply and demand.  

Figure~\ref{fig:uncertaintyRates} shows the change in cathode recycling potential and recycling rates of retired battery cells by certainty level.  In high certainty scenarios, the expected proportion of retired battery cells that are recycled increases from $99.5\%$ to $99.9\%$, as scenarios with high battery retirement rates that exceed recycling capacity are less likely.  The expected recycling rate exceeds $99\%$ at all certainty levels, indicating that recycling capacity always meets or exceeds the supply of retired cells in the most likely scenarios.  On the other hand, the expected recycling potential of cathode material decreases from $23\%$ to $20\%$ as $\alpha$ increases, indicating that new battery manufacturing uses a smaller proportion of recycled cathode powder when certainty is high, due to decreased investment in recycling capacity.

%%Conclusion
\section{Conclusion}
\label{sec:conclusion}

In this work, we propose a multistage stochastic optimization model for the design and operation of closed-loop EV battery recycling supply chains.  We transform EverBatt, an evaluative framework for battery recycling cost and environmental impact, into a detailed, high-fidelity decision-making model that determines optimal recycling and cathode production investment strategies.  We construct a procedure for generating multistage scenarios for new battery demand, retired battery supply, and material costs, leveraging state-of-the-art econometric models for critical mineral prices and EV demand.  To solve our model, we describe an algorithm for separable concave minimization over polyhedra that improves performance over existing algorithms and a cut grouping strategy for Benders' decomposition that converges faster than standard approaches.  We find an equivalent formulation of our model which reduces solve times by introducing integrality constraints.  Analysis of the optimal solutions shows large improvements in the cost of manufacturing new batteries when recycling is incorporated into the supply chain, as well as reductions in the environmental impact of battery manufacturing and disposal. We also show that recycling retired batteries is more cost-effective than repurposing them for stationary energy storage.  We quantify the operational value of hydrometallurgical recycling generated by its material recovery efficiency and manufacturing flexibility, and show how recycling investment is motivated by high uncertainty in material costs and retired battery supply.  We provide insights into the impact of recycling capacity grants and production credits on domestic investment in battery recycling, revealing that production credits encourage greater investment at lower policy costs, whereas grants become more cost-effective with higher policy expenditure.  Overall, this methodology provides a comprehensive and flexible quantitative framework for recyclers, battery manufacturers, and policy makers to plan for the integration of recycling technologies into the supply chain.

\theendnotes

\nocite{horst2000introduction}
\bibliographystyle{informs2014}
\bibliography{battopt.bib}

\begin{thebibliography}{3}
\providecommand{\natexlab}[1]{#1}
\providecommand{\url}[1]{\texttt{#1}}
\providecommand{\urlprefix}{URL }

\bibitem[{Dai et~al.(2019)Dai, Spangenberger, Ahmed, Gaines, Kelly,
  \protect\BIBand{} Wang}]{dai2019everbattec}
Dai Q, Spangenberger J, Ahmed S, Gaines L, Kelly JC, Wang M (2019) {EverBatt}:
  A closed-loop battery recycling cost and environmental impacts model.
  Technical report, Argonne National Lab.

\bibitem[{Horst et~al.(2000)Horst, Pardalos, \protect\BIBand{}
  Van~Thoai}]{horst2000introductionec}
Horst R, Pardalos PM, Van~Thoai N (2000) \emph{Introduction to Global
  Optimization} (Springer Science \& Business Media).

\bibitem[{Xu et~al.(2020)Xu, Dai, Gaines, Hu, Tukker, \protect\BIBand{}
  Steubing}]{xu2020futureec}
Xu C, Dai Q, Gaines L, Hu M, Tukker A, Steubing B (2020) Future material demand
  for automotive lithium-based batteries. \emph{Communications Materials} 1:99.

\end{thebibliography}


\begin{thebibliography}{89}
\providecommand{\natexlab}[1]{#1}
\providecommand{\url}[1]{\texttt{#1}}
\providecommand{\urlprefix}{URL }

\bibitem[{{ABI Research}(2024)}]{abi2024capacity}
{ABI Research} (2024) {U.S.} {EV} battery recycling industry faces challenge as
  input supply reaches only a quarter of capacity by 2030. Accessed December
  2024,
  \url{https://www.abiresearch.com/press/us-ev-battery-recycling-industry-faces-challenge-as-input-supply-reaches-only-a-quarter-of-capacity-by-2030}.

\bibitem[{Adulyasak et~al.(2015)Adulyasak, Cordeau, \protect\BIBand{}
  Jans}]{adulyasak2015benders}
Adulyasak Y, Cordeau JF, Jans R (2015) Benders decomposition for production
  routing under demand uncertainty. \emph{Operations Research} 63(4):851--867.

\bibitem[{Allman et~al.(2017)Allman, Daoutidis, Tiffany, \protect\BIBand{}
  Kelley}]{allman2017framework}
Allman A, Daoutidis P, Tiffany D, Kelley S (2017) A framework for ammonia
  supply chain optimization incorporating conventional and renewable
  generation. \emph{AIChE Journal} 63(10):4390--4402.

\bibitem[{Antol{\'\i}n-D{\'\i}az et~al.(2021)Antol{\'\i}n-D{\'\i}az, Petrella,
  \protect\BIBand{} Rubio-Ram{\'\i}rez}]{antolin2021structural}
Antol{\'\i}n-D{\'\i}az J, Petrella I, Rubio-Ram{\'\i}rez JF (2021) Structural
  scenario analysis with {SVARs}. \emph{Journal of Monetary Economics}
  117:798--815.

\bibitem[{Azzi et~al.(2014)Azzi, Battini, Faccio, Persona, \protect\BIBand{}
  Sgarbossa}]{azzi2014inventory}
Azzi A, Battini D, Faccio M, Persona A, Sgarbossa F (2014) Inventory holding
  costs measurement: A multi-case study. \emph{The International Journal of
  Logistics Management} 25(1):109--132.

\bibitem[{BASF(2023)}]{basf2023partnership}
BASF (2023) Partnership between {BASF} and {N}anotech {E}nergy will enable
  production of lithium-ion batteries in {N}orth {A}merica with locally
  recycled content and low {CO2} footprint. Accessed June 2024,
  \url{https://www.basf.com/global/en/media/news-releases/2023/09/p-23-292.html}.

\bibitem[{Bazaraa \protect\BIBand{} Sherali(1982)}]{bazaraa1982use}
Bazaraa MS, Sherali HD (1982) On the use of exact and heuristic cutting plane
  methods for the quadratic assignment problem. \emph{Journal of the
  Operational Research Society} 33(11):991--1003.

\bibitem[{Benders(1962)}]{benders1962partitioning}
Benders JF (1962) Partitioning procedures for solving mixed-variables
  programming problems. \emph{Numerische Mathematik} 4(1):238--252.

\bibitem[{Benson(1985)}]{benson1985finite}
Benson HP (1985) A finite algorithm for concave minimization over a polyhedron.
  \emph{Naval Research Logistics Quarterly} 32(1):165--177.

\bibitem[{Birge \protect\BIBand{} Louveaux(2011)}]{birge2011introduction}
Birge JR, Louveaux F (2011) \emph{Introduction to Stochastic Programming}
  (Springer Science \& Business Media).

\bibitem[{Bloemhof-Ruwaard et~al.(1996)Bloemhof-Ruwaard, Van~Wassenhove, Gabel,
  \protect\BIBand{} Weaver}]{bloemhof1996environmental}
Bloemhof-Ruwaard JM, Van~Wassenhove LN, Gabel HL, Weaver PM (1996) An
  environmental life cycle optimization model for the {European} pulp and paper
  industry. \emph{Omega} 24(6):615--629.

\bibitem[{BMW(2024)}]{bmw2024partnership}
BMW (2024) {BMW} of {N}orth {A}merica and {R}edwood {M}aterials establish
  partnership to recycle lithium-ion batteries. Accessed August 2026,
  \url{https://www.press.bmwgroup.com/usa/article/detail/T0445142EN_US/bmw-of-north-america-and-redwood-materials-establish-partnership-to-recycle-lithium-ion-batteries}.

\bibitem[{Boer et~al.(2024)Boer, Pescatori, \protect\BIBand{}
  Stuermer}]{pescatori2021energy}
Boer L, Pescatori A, Stuermer M (2024) Energy transition metals: Bottleneck for
  net-zero emissions? \emph{Journal of the European Economic Association}
  22(1):200--229.

\bibitem[{Campagnol et~al.(2022)Campagnol, Pfeiffer, \protect\BIBand{}
  Tryggestad}]{campagnol2022valuechain}
Campagnol N, Pfeiffer A, Tryggestad C (2022) Capturing the battery value-chain
  opportunity. Accessed October 2023,
  \url{https://www.mckinsey.com/industries/electric-power-and-natural-gas/our-insights/capturing-the-battery-value-chain-opportunity}.

\bibitem[{Chaudhari et~al.(2021)Chaudhari, Lin, Thompson, Handler, Pearce,
  Caneba, Muhuri, Watkins, \protect\BIBand{} Shonnard}]{chaudhari2021systems}
Chaudhari US, Lin Y, Thompson VS, Handler RM, Pearce JM, Caneba G, Muhuri P,
  Watkins D, Shonnard DR (2021) Systems analysis approach to polyethylene
  terephthalate and olefin plastics supply chains in the circular economy: A
  review of data sets and models. \emph{ACS Sustainable Chemistry \&
  Engineering} 9(22):7403--7421.

\bibitem[{Choi \protect\BIBand{} Fthenakis(2010)}]{choi2010design}
Choi JK, Fthenakis V (2010) Design and optimization of photovoltaics recycling
  infrastructure. \emph{Environmental Science \& Technology} 44(22):8678--8683.

\bibitem[{Contec(2023)}]{contec2023carbonblack}
Contec (2023) On the rise: The price of carbon black isn’t sustainable.
  Accessed December 2023,
  \url{https://contec.tech/price-of-carbon-black-sustainability}.

\bibitem[{{Council of the European Union}(2023)}]{eu2023recyclingrequirement}
{Council of the European Union} (2023) Council adopts new regulation on
  batteries and waste batteries. Accessed October 2023,
  \url{https://www.consilium.europa.eu/en/press/press-releases/2023/07/10/council-adopts-new-regulation-on-batteries-and-waste-batteries}.

\bibitem[{Dai et~al.(2019)Dai, Spangenberger, Ahmed, Gaines, Kelly,
  \protect\BIBand{} Wang}]{dai2019everbatt}
Dai Q, Spangenberger J, Ahmed S, Gaines L, Kelly JC, Wang M (2019) {EverBatt}:
  A closed-loop battery recycling cost and environmental impacts model.
  Technical report, Argonne National Lab.

\bibitem[{D’Ambrosio et~al.(2009)D’Ambrosio, Lee, \protect\BIBand{}
  W{\"a}chter}]{dambrosio2009global}
D’Ambrosio C, Lee J, W{\"a}chter A (2009) A global-optimization algorithm for
  mixed-integer nonlinear programs having separable non-convexity.
  \emph{European Symposium on Algorithms}, 107--118.

\bibitem[{Ejeh et~al.(2023)Ejeh, Martynov, \protect\BIBand{}
  Brown}]{ejeh2023minlp}
Ejeh JO, Martynov SB, Brown SF (2023) An {MINLP} model for the optimal design
  of {CO2} transportation infrastructure in industrial clusters. \emph{Computer
  Aided Chemical Engineering} 52:3085--3090.

\bibitem[{Faessler(2021)}]{faessler2021stationary}
Faessler B (2021) Stationary, second use battery energy storage systems and
  their applications: A research review. \emph{Energies} 14(8):2335.

\bibitem[{Falk \protect\BIBand{} Soland(1969)}]{falk1969algorithm}
Falk JE, Soland RM (1969) An algorithm for separable nonconvex programming
  problems. \emph{Management Science} 15(9):550--569.

\bibitem[{Feldman et~al.(1966)Feldman, Lehrer, \protect\BIBand{}
  Ray}]{feldman1966warehouse}
Feldman E, Lehrer FA, Ray TL (1966) Warehouse location under continuous
  economies of scale. \emph{Management Science} 12(9):670--684.

\bibitem[{Fioriti et~al.(2023)Fioriti, Scarpelli, Pellegrino, Lutzemberger,
  Micolano, \protect\BIBand{} Salamone}]{fioriti2023battery}
Fioriti D, Scarpelli C, Pellegrino L, Lutzemberger G, Micolano E, Salamone S
  (2023) Battery lifetime of electric vehicles by novel rainflow-counting
  algorithm with temperature and {C}-rate dynamics: Effects of fast charging,
  user habits, vehicle-to-grid and climate zones. \emph{Journal of Energy
  Storage} 59:106458.

\bibitem[{Fleischmann et~al.(2001)Fleischmann, Beullens, Bloemhof-Ruwaard,
  \protect\BIBand{} Van~Wassenhove}]{fleischmann2001impact}
Fleischmann M, Beullens P, Bloemhof-Ruwaard JM, Van~Wassenhove LN (2001) The
  impact of product recovery on logistics network design. \emph{Production and
  Operations Management} 10(2):156--173.

\bibitem[{Gemechu et~al.(2017)Gemechu, Sonnemann, \protect\BIBand{}
  Young}]{gemechu2017geopolitical}
Gemechu ED, Sonnemann G, Young SB (2017) Geopolitical-related supply risk
  assessment as a complement to environmental impact assessment: The case of
  electric vehicles. \emph{The International Journal of Life Cycle Assessment}
  22:31--39.

\bibitem[{Georgiadis \protect\BIBand{}
  Athanasiou(2013)}]{georgiadis2013flexible}
Georgiadis P, Athanasiou E (2013) Flexible long-term capacity planning in
  closed-loop supply chains with remanufacturing. \emph{European Journal of
  Operational Research} 225(1):44--58.

\bibitem[{Harper et~al.(2019)Harper, Sommerville, Kendrick, Driscoll, Slater,
  Stolkin, Walton, Christensen, Heidrich, Lambert, Abbott, Ryder, Gaines,
  \protect\BIBand{} Anderson}]{harper2019recycling}
Harper G, Sommerville R, Kendrick E, Driscoll L, Slater P, Stolkin R, Walton A,
  Christensen P, Heidrich O, Lambert S, Abbott A, Ryder K, Gaines L, Anderson P
  (2019) Recycling lithium-ion batteries from electric vehicles. \emph{Nature}
  575:75--86.

\bibitem[{Horst et~al.(2000)Horst, Pardalos, \protect\BIBand{}
  Van~Thoai}]{horst2000introduction}
Horst R, Pardalos PM, Van~Thoai N (2000) \emph{Introduction to Global
  Optimization} (Springer Science \& Business Media).

\bibitem[{Hoyer et~al.(2015)Hoyer, Kieckh{\"a}fer, \protect\BIBand{}
  Spengler}]{hoyer2015technology}
Hoyer C, Kieckh{\"a}fer K, Spengler TS (2015) Technology and capacity planning
  for the recycling of lithium-ion electric vehicle batteries in {G}ermany.
  \emph{Journal of Business Economics} 85:505--544.

\bibitem[{IEA(2020)}]{iea2020scenarios}
IEA (2020) Global {EV} outlook 2020. Technical report, IEA.

\bibitem[{IEA(2021)}]{iea2021minerals}
IEA (2021) The role of critical minerals in clean energy transitions. Technical
  report, IEA.

\bibitem[{IEA(2022)}]{iea2022supplychain}
IEA (2022) Global supply chains of {EV} batteries. Technical report, IEA.

\bibitem[{IEA(2023)}]{iea2023explorer}
IEA (2023) Global {EV} data explorer. Accessed August 2023,
  \url{https://www.iea.org/data-and-statistics/data-tools/global-ev-data-explorer}.

\bibitem[{IMF(2022)}]{imf2023pcps}
IMF (2022) Primary commodity price system. Accessed December 2022,
  \url{https://data.imf.org/commodityprices}.

\bibitem[{{Internal Revenue Service}(2025)}]{irs2025credit}
{Internal Revenue Service} (2025) Advanced manufacturing production credit.
  Accessed June 2025,
  \url{https://www.irs.gov/credits-deductions/advanced-manufacturing-production-credit}.

\bibitem[{Jayaraman et~al.(1999)Jayaraman, Guide~Jr., \protect\BIBand{}
  Srivastava}]{jayaraman1999closed}
Jayaraman V, Guide~Jr VDR, Srivastava R (1999) A closed-loop logistics model
  for remanufacturing. \emph{Journal of the Operational Research Society}
  50(5):497--508.

\bibitem[{Kang et~al.(2013)Kang, Chen, \protect\BIBand{}
  Ogunseitan}]{kang2013potential}
Kang DHP, Chen M, Ogunseitan OA (2013) Potential environmental and human health
  impacts of rechargeable lithium batteries in electronic waste.
  \emph{Environmental Science \& Technology} 47(10):5495--5503.

\bibitem[{Kaya et~al.(2014)Kaya, Bagci, \protect\BIBand{}
  Turkay}]{kaya2014planning}
Kaya O, Bagci F, Turkay M (2014) Planning of capacity, production and inventory
  decisions in a generic reverse supply chain under uncertain demand and
  returns. \emph{International Journal of Production Research} 52(1):270--282.

\bibitem[{Kneifel \protect\BIBand{} Webb(2020)}]{kneifel2020FEMP}
Kneifel J, Webb D (2020) Life cycle costing manual for the {F}ederal {E}nergy
  {M}anagement {P}rogram. \url{https://doi.org/10.6028/NIST.HB.135-2020}.

\bibitem[{Lee \protect\BIBand{} Luss(1987)}]{lee1987multifacility}
Lee SB, Luss H (1987) Multifacility-type capacity expansion planning:
  Algorithms and complexities. \emph{Operations Research} 35(2):249--253.

\bibitem[{Li et~al.(2018)Li, Dababneh, \protect\BIBand{} Zhao}]{li2018cost}
Li L, Dababneh F, Zhao J (2018) Cost-effective supply chain for electric
  vehicle battery remanufacturing. \emph{Applied Energy} 226:277--286.

\bibitem[{Li et~al.(2023)Li, Huai, \protect\BIBand{}
  Jiang}]{li2023optimization}
Li P, Huai Z, Jiang J (2023) Optimization of closed-loop supply chain network
  design under uncertainty: Considering electric vehicle battery recycling.
  \emph{Industrial Engineering and Applications}, volume~35 of \emph{Advances
  in Transdisciplinary Engineering}, 182--190.

\bibitem[{Li \protect\BIBand{} Tirupati(1994)}]{li1994dynamic}
Li S, Tirupati D (1994) Dynamic capacity expansion problem with multiple
  products: Technology selection and timing of capacity additions.
  \emph{Operations Research} 42(5):958--976.

\bibitem[{Lieckens \protect\BIBand{} Vandaele(2007)}]{lieckens2007reverse}
Lieckens K, Vandaele N (2007) Reverse logistics network design with stochastic
  lead times. \emph{Computers \& Operations Research} 34(2):395--416.

\bibitem[{Magnanti \protect\BIBand{} Stratila(2004)}]{magnanti2004separable}
Magnanti TL, Stratila D (2004) Separable concave optimization approximately
  equals piecewise linear optimization. \emph{Integer Programming and
  Combinatorial Optimization}, 234--243.

\bibitem[{Makuza et~al.(2021)Makuza, Tian, Guo, Chattopadhyay,
  \protect\BIBand{} Yu}]{makuza2021pyrometallurgical}
Makuza B, Tian Q, Guo X, Chattopadhyay K, Yu D (2021) Pyrometallurgical options
  for recycling spent lithium-ion batteries: A comprehensive review.
  \emph{Journal of Power Sources} 491:229622.

\bibitem[{Malik et~al.(2022)Malik, Contreras, \protect\BIBand{}
  Vidyarthi}]{malik2022two}
Malik A, Contreras I, Vidyarthi N (2022) Two-level capacitated discrete
  location with concave costs. \emph{Transportation Science} 56(6):1703--1722.

\bibitem[{Mart{\'\i}nez-Costa et~al.(2014)Mart{\'\i}nez-Costa, Mas-Machuca,
  Benedito, \protect\BIBand{} Corominas}]{martinez2014review}
Mart{\'\i}nez-Costa C, Mas-Machuca M, Benedito E, Corominas A (2014) A review
  of mathematical programming models for strategic capacity planning in
  manufacturing. \emph{International Journal of Production Economics}
  153:66--85.

\bibitem[{Mehdi \protect\BIBand{} Moerenhout(2023)}]{mehdi2023ira}
Mehdi A, Moerenhout T (2023) The {IRA} and the {US} battery supply chain: one
  year on. Accessed June 2025,
  \url{https://www.energypolicy.columbia.edu/publications/the-ira-and-the-us-battery-supply-chain-one-year-on}.

\bibitem[{Miller~III \protect\BIBand{} Rice(1983)}]{miller1983discrete}
Miller~III AC, Rice TR (1983) Discrete approximations of probability
  distributions. \emph{Management Science} 29(3):352--362.

\bibitem[{Mrozik et~al.(2021)Mrozik, Rajaeifar, Heidrich, \protect\BIBand{}
  Christensen}]{mrozik2021environmental}
Mrozik W, Rajaeifar MA, Heidrich O, Christensen P (2021) Environmental impacts,
  pollution sources and pathways of spent lithium-ion batteries. \emph{Energy
  \& Environmental Science} 14(12):6099--6121.

\bibitem[{Myung et~al.(2017)Myung, Maglia, Park, Yoon, Lamp, Kim,
  \protect\BIBand{} Sun}]{myung2017nickel}
Myung ST, Maglia F, Park KJ, Yoon CS, Lamp P, Kim SJ, Sun YK (2017) Nickel-rich
  layered cathode materials for automotive lithium-ion batteries: Achievements
  and perspectives. \emph{ACS Energy Letters} 2(1):196--223.

\bibitem[{Neigum \protect\BIBand{} Wang(2024)}]{neigum2024technology}
Neigum K, Wang Z (2024) Technology, economic, and environmental analysis of
  second-life batteries as stationary energy storage: A review. \emph{Journal
  of Energy Storage} 103:114393.

\bibitem[{Nygaard(2023)}]{nygaard2023geopolitical}
Nygaard A (2023) The geopolitical risk and strategic uncertainty of green
  growth after the {Ukraine} invasion: How the circular economy can decrease
  the market power of and resource dependency on critical minerals.
  \emph{Circular Economy and Sustainability} 3:1099--1126.

\bibitem[{Ohnsman(2022)}]{forbes2022redwood}
Ohnsman A (2022) {VW} and {A}udi to recycle {EV} batteries through {T}esla
  cofounder’s company. Accessed June 2024,
  \url{https://www.forbes.com/sites/alanohnsman/2022/07/12/vw-and-audi-to-recycle-ev-batteries-through-tesla-cofounders-company}.

\bibitem[{Pardalos \protect\BIBand{} Romeijn(2002)}]{horst2013handbook}
Pardalos PM, Romeijn HE (2002) \emph{Handbook of Global Optimization}, volume~2
  (Springer).

\bibitem[{Pardalos \protect\BIBand{} Rosen(1986)}]{pardalos1986methods}
Pardalos PM, Rosen JB (1986) Methods for global concave minimization: A
  bibliographic survey. \emph{SIAM Review} 28(3):367--379.

\bibitem[{Pourmohammadi et~al.(2008)Pourmohammadi, Rahimi, \protect\BIBand{}
  Dessouky}]{pourmohammadi2008sustainable}
Pourmohammadi H, Rahimi M, Dessouky M (2008) Sustainable reverse logistics for
  distribution of industrial waste/byproducts: A joint optimization of
  operation and environmental costs. \emph{Supply {C}hain {F}orum: An
  {I}nternational {J}ournal}, volume~9, 2--17.

\bibitem[{Raghavachari(1969)}]{raghavachari1969connections}
Raghavachari M (1969) On connections between zero-one integer programming and
  concave programming under linear constraints. \emph{Operations Research}
  17(4):680--684.

\bibitem[{Rajagopalan et~al.(1998)Rajagopalan, Singh, \protect\BIBand{}
  Morton}]{rajagopalan1998capacity}
Rajagopalan S, Singh MR, Morton TE (1998) Capacity expansion and replacement in
  growing markets with uncertain technological breakthroughs. \emph{Management
  Science} 44(1):12--30.

\bibitem[{Raz \protect\BIBand{} Souza(2018)}]{raz2018recycling}
Raz G, Souza GC (2018) Recycling as a strategic supply source. \emph{Production
  and Operations Management} 27(5):902--916.

\bibitem[{Rosen(1983)}]{rosen1983global}
Rosen JB (1983) Global minimization of a linearly constrained concave function
  by partition of feasible domain. \emph{Mathematics of Operations Research}
  8(2):215--230.

\bibitem[{Rosenberg et~al.(2023)Rosenberg, Gl{\"o}ser-Chahoud, Huster,
  \protect\BIBand{} Schultmann}]{rosenberg2023dynamic}
Rosenberg S, Gl{\"o}ser-Chahoud S, Huster S, Schultmann F (2023) A dynamic
  network design model with capacity expansions for {EoL} traction battery
  recycling -- a case study of an {OEM} in {G}ermany. \emph{Waste Management}
  160:12--22.

\bibitem[{Sahinidis(1996)}]{sahinidis1996baron}
Sahinidis NV (1996) {BARON}: A general purpose global optimization software
  package. \emph{Journal of Global Optimization} 8:201--205.

\bibitem[{Schultmann et~al.(2003)Schultmann, Engels, \protect\BIBand{}
  Rentz}]{schultmann2003closed}
Schultmann F, Engels B, Rentz O (2003) Closed-loop supply chains for spent
  batteries. \emph{Interfaces} 33(6):57--71.

\bibitem[{Shectman \protect\BIBand{} Sahinidis(1998)}]{shectman1998finite}
Shectman JP, Sahinidis NV (1998) A finite algorithm for global minimization of
  separable concave programs. \emph{Journal of Global Optimization} 12:1--36.

\bibitem[{Silverstein(2021)}]{forbes2021tesla}
Silverstein K (2021) What are the keys to {EV} success? {T}esla is focused on
  battery recycling. Accessed August 2026,
  \url{https://www.forbes.com/sites/kensilverstein/2021/08/24/what-are-the-keys-to-ev-success-tesla-is-focused-on-battery-recycling}.

\bibitem[{Sodhi \protect\BIBand{} Reimer(2001)}]{sodhi2001models}
Sodhi MS, Reimer B (2001) Models for recycling electronics end-of-life
  products. \emph{OR-Spektrum} 23:97--115.

\bibitem[{Tadaros et~al.(2022)Tadaros, Migdalas, Samuelsson, \protect\BIBand{}
  Segerstedt}]{tadaros2022location}
Tadaros M, Migdalas A, Samuelsson B, Segerstedt A (2022) Location of facilities
  and network design for reverse logistics of lithium-ion batteries in
  {S}weden. \emph{Operational Research} 22:895--915.

\bibitem[{Tankou \protect\BIBand{} Hall(2023)}]{icct2023capacity}
Tankou A, Hall D (2023) Will the {U.S.} {EV} battery recycling industry be
  ready for millions of end-of-life batteries? Accessed December 2024,
  \url{https://theicct.org/us-ev-battery-recycling-end-of-life-batteries-sept23}.

\bibitem[{Tuy(1964)}]{tuy1964concave}
Tuy H (1964) Concave programming with linear constraints. \emph{Doklady
  Akademii Nauk} 159(1):32--35.

\bibitem[{{U.S. Department of Energy}(2025{\natexlab{a}})}]{doe2025grant}
{US Department of Energy} (2025{\natexlab{a}}) Battery manufacturing and
  recycling grants. Accessed May 2025,
  \url{https://www.energy.gov/mesc/battery-manufacturing-and-recycling-grants}.

\bibitem[{{U.S. Department of Energy}(2025{\natexlab{b}})}]{doe2025grant2}
{US Department of Energy} (2025{\natexlab{b}}) Battery materials processing
  grants. Accessed May 2025,
  \url{https://www.energy.gov/mesc/battery-materials-processing-grants}.

\bibitem[{{U.S. Geological Survey}(2013)}]{usgs2013metalprice}
{US Geological Survey} (2013) Metal prices in the {U}nited {S}tates through
  2010. Technical report, U.S. Geological Survey,
  \url{http://dx.doi.org/10.3133/sir20125188}.

\bibitem[{{U.S. Geological Survey}(2025)}]{usgs2022critical}
{US Geological Survey} (2025) Final 2025 list of critical minerals. Accessed
  July 2026,
  \url{https://www.federalregister.gov/documents/2025/11/07/2025-19813/final-2025-list-of-critical-minerals}.

\bibitem[{Van~Mieghem(2003)}]{van2003commissioned}
Van~Mieghem JA (2003) Capacity management, investment, and hedging: Review and
  recent developments. \emph{Manufacturing \& Service Operations Management}
  5(4):269--302.

\bibitem[{Van~Slyke \protect\BIBand{} Wets(1969)}]{van1969shaped}
Van~Slyke RM, Wets R (1969) L-shaped linear programs with applications to
  optimal control and stochastic programming. \emph{{SIAM} Journal on Applied
  Mathematics} 17(4):638--663.

\bibitem[{Vlachos et~al.(2007)Vlachos, Georgiadis, \protect\BIBand{}
  Iakovou}]{vlachos2007system}
Vlachos D, Georgiadis P, Iakovou E (2007) A system dynamics model for dynamic
  capacity planning of remanufacturing in closed-loop supply chains.
  \emph{Computers \& Operations Research} 34(2):367--394.

\bibitem[{Wang et~al.(2020)Wang, Wang, \protect\BIBand{}
  Yang}]{wang2020optimal}
Wang L, Wang X, Yang W (2020) Optimal design of electric vehicle battery
  recycling network -- from the perspective of electric vehicle manufacturers.
  \emph{Applied Energy} 275:115328.

\bibitem[{Wu et~al.(2023)Wu, Zheng, Liu, Wang, Liu, Nai, Zhang, Zhang,
  \protect\BIBand{} Tao}]{wu2023direct}
Wu J, Zheng M, Liu T, Wang Y, Liu Y, Nai J, Zhang L, Zhang S, Tao X (2023)
  Direct recovery: A sustainable recycling technology for spent lithium-ion
  battery. \emph{Energy Storage Materials} 54:120--134.

\bibitem[{Xu et~al.(2020)Xu, Dai, Gaines, Hu, Tukker, \protect\BIBand{}
  Steubing}]{xu2020future}
Xu C, Dai Q, Gaines L, Hu M, Tukker A, Steubing B (2020) Future material demand
  for automotive lithium-based batteries. \emph{Communications Materials} 1:99.

\bibitem[{Yao et~al.(2018)Yao, Zhu, Zhao, Tong, Fan, \protect\BIBand{}
  Hua}]{yao2018hydrometallurgical}
Yao Y, Zhu M, Zhao Z, Tong B, Fan Y, Hua Z (2018) Hydrometallurgical processes
  for recycling spent lithium-ion batteries: A critical review. \emph{ACS
  Sustainable Chemistry \& Engineering} 6(11):13611--13627.

\bibitem[{You \protect\BIBand{} Grossmann(2011)}]{you2011stochastic}
You F, Grossmann IE (2011) Stochastic inventory management for tactical process
  planning under uncertainties: {MINLP} models and algorithms. \emph{AIChE
  Journal} 57(5):1250--1277.

\bibitem[{Zeballos et~al.(2014)Zeballos, M{\'e}ndez, Barbosa-Povoa,
  \protect\BIBand{} Novais}]{zeballos2014multi}
Zeballos LJ, M{\'e}ndez CA, Barbosa-Povoa AP, Novais AQ (2014) Multi-period
  design and planning of closed-loop supply chains with uncertain supply and
  demand. \emph{Computers \& Chemical Engineering} 66:151--164.

\bibitem[{Zhang \protect\BIBand{} Sahinidis(2025)}]{zhang2025solving}
Zhang Y, Sahinidis NV (2025) Solving continuous and discrete nonlinear programs
  with {BARON}. \emph{Computational Optimization and Applications}
  92:1123--1161.

\bibitem[{Zhen et~al.(2018)Zhen, Wu, Wang, Hu, \protect\BIBand{}
  Yi}]{zhen2018capacitated}
Zhen L, Wu Y, Wang S, Hu Y, Yi W (2018) Capacitated closed-loop supply chain
  network design under uncertainty. \emph{Advanced Engineering Informatics}
  38:306--315.

\bibitem[{Zhu \protect\BIBand{} Chen(2020)}]{zhu2020research}
Zhu L, Chen M (2020) Research on spent {LiFePO4} electric vehicle battery
  disposal and its life cycle inventory collection in {China}.
  \emph{International Journal of Environmental Research and Public Health}
  17(23):8828.

\end{thebibliography}

\newpage

%%Appendix
\appendix

\section{Appendix}
This appendix is organized as follows.  Section~\ref{sec:ecModel} contains tables that summarize the data and variables used in our models.  Section~\ref{sec:ecCCParameter} describes the computation of the average battery mass parameter used in demand scenario generation.  Section~\ref{sec:ecZoneCost} describes how cost parameters vary across zones and conducts a sensitivity analysis on these parameters.  Section~\ref{sec:ecProofs} provides proofs of our results.  Section~\ref{sec:ecExtensions} gives rigorous justification of the modeling extensions.

\subsection{Summary of Model Data}
\label{sec:ecModel}
Tables~\ref{tbl:modelSets}, \ref{tbl:modelParams}, and \ref{tbl:modelVars} give descriptions of the index sets, constant parameters, and variables used to define the models \eqref{eq:deterministicP} and \eqref{eq:reformulatedP}.
\begin{xltabular}{\linewidth}{l|X}
    \caption{Sets}\label{tbl:modelSets} \\
    \hline
    $\mathcal{I}$ & Set of battery chemistries\\
    $\mathcal{J}$ & Set of recycling processes\\
    $\mathcal{K}$ & Set of materials encountered during recycling and manufacturing\\
    $\mathcal{K}^{\text{CP}}$ & Set of cathode powders\\
    $\mathcal{L}$ & Set of planning periods in the model horizon\\
    $\mathcal{N}^{\text{CP}}_{l,k}$ & Set of cathode production line indices used to produce material $k \in \mathcal{K}^{\text{CP}}$ during planning period $l \in \mathcal{L}$\\
    $\mathcal{N}^{\text{REC}}_l$ & Set of recycling facility indices during planning period $l \in \mathcal{L}$\\
    $\mathcal{S}$ & Set of stages in the multistage scenario tree\\
    $\mathcal{T}$ & Set of time periods in the model horizon\\
    $\mathcal{T}_l$ & Set of time periods associated with planning period $l \in \mathcal{L}$\\
    $\mathcal{Z}$ & Set of zones in which production facilities can be located\\
    $\Omega_{\sigma}$ & Set of multistage scenario tree nodes at stage $\sigma \in \mathcal{S}$\\
    \hline
\end{xltabular}

\begin{xltabular}{\linewidth}{l|X}
    \caption{Parameters}\label{tbl:modelParams} \\
    \hline
    $a_\omega(t)$ & Ancestor node of node $\omega \in \cup_{\sigma \in \mathcal{S}} \Omega_{\sigma}$ at period $t \in \mathcal{T}$\\
    $c^{\text{CP}}_{\omega,z,t,k}$ & Variable cost of producing one unit of material $k \in \mathcal{K}^{\text{CP}}$ via cathode production in time period $t \in \mathcal{T}$, zone $z \in \mathcal{Z}$ and node $\omega \in \Omega_{\sigma_t}$\\
    $c^{\text{CP,NM}}_{\omega,t,k}$ & Cost per unit of material $k \in \mathcal{K} \setminus \mathcal{K}^{\text{CP}}$ purchased new to manufacture cathode powder in time period $t \in \mathcal{T}$ and node $\omega \in \Omega_{\sigma_t}$\\
    $c^{\text{MC}}_{\omega,z,t,k}$ & Variable cost for generating one unit of material $k \in \mathcal{K} \setminus \mathcal{K}^{\text{CP}}$ via material conversion in time period $t \in \mathcal{T}$, zone $z \in \mathcal{Z}$ and node $\omega \in \Omega_{\sigma_t}$\\
    $c^{\text{NB,NM}}_{\omega,t,k}$ & Cost per unit of material $k \in \mathcal{K}$ purchased new to manufacture batteries in time period $t \in \mathcal{T}$ and node $\omega \in \Omega_{\sigma_t}$\\
    $c^{\text{REC}}_{\omega,z,t,i,j}$ & Variable cost per unit mass of battery of chemistry $i \in \mathcal{I}$ recycled with process $j \in \mathcal{J}$ in time period $t \in \mathcal{T}$, zone $z \in \mathcal{Z}$ and node $\omega \in \Omega_{\sigma_t}$\\
    $c^{\text{TR,RB}}_{\omega,t,z,z'}$ & Cost per unit mass of  battery transported from zone $z \in \mathcal{Z}$ to zone $z' \in \mathcal{Z} \setminus \{z\}$ in time period $t \in \mathcal{T}$ and node $\omega \in \Omega_{\sigma_t}$\\
    $c^{\text{TR,RM}}_{\omega,t,z,z'}$ & Cost to transport a unit of recycled material from zone $z \in \mathcal{Z}$ to zone $z' \in \mathcal{Z} \setminus \{z\}$ in time period $t \in \mathcal{T}$ and node $\omega \in \Omega_{\sigma_t}$\\
    $C^{\text{OP}}_{\omega,t}$ & Operational costs as a function of operational decisions in time period $t \in \mathcal{T}$ and node $\omega \in \Omega_{\sigma_t}$\\
    $C^{\text{PL}}_t$ & Planning costs as a function of capacity decisions in time period $t \in \mathcal{T}$\\
    $\overline{C}^{\text{PL}}_t$ & Planning costs as a function of capacity decisions in time period $t \in \mathcal{T}$ for the reformulated capacity variable structure\\
    $d_{\omega,z,t,i}$ & Demand for batteries of chemistry $i \in \mathcal{I}$ in time period $t \in \mathcal{T}$, zone $z \in \mathcal{Z}$ and node $\omega \in \Omega_{\sigma_t}$\\
    $f^{\text{CP}}_{z,k}$ & Concave function mapping cathode production capacity for a line producing material $k \in \mathcal{K}^{\text{CP}}$ in zone $z \in \mathcal{Z}$ to annual labor and capital cost\\
    $f^{\text{REC}}_{z,j}$ & Concave function mapping recycling capacity for process $j \in \mathcal{J}$ in zone $z \in \mathcal{Z}$ to annual labor and capital cost\\
    $l_t$ & Planning period associated with time period $t \in \mathcal{T}$\\
    $p_\omega$ & Probability associated with node $\omega \in \cup_{\sigma \in \mathcal{S}} \Omega_\sigma$\\
    $s_{\omega,z,t,i}$ & Supply of batteries of chemistry $i \in \mathcal{I}$ in time period $t \in \mathcal{T}$, zone $z \in \mathcal{Z}$ and node $\omega \in \Omega_{\sigma_t}$\\
    $u^{\text{CP}}$ & Maximum capacity of a cathode production line\\
    $u^{\text{REC}}$ & Maximum capacity of a recycling process at each facility\\
    $v_{\omega,t,k}$ & Value per unit of material $k \in \mathcal{K}$ in time period $t \in \mathcal{T}$ and node $\omega \in \Omega_{\sigma_t}$\\
    $\Delta^{\text{CP}}_{k',k}$ & Amount of material $k \in \mathcal{K} \setminus \mathcal{K}^{\text{CP}}$ used to produce one unit of material $k' \in \mathcal{K}^{\text{CP}}$ via cathode production\\
    $\Delta^{\text{MC}}_{k',k}$ & Amount of material $k \in \mathcal{K} \setminus \mathcal{K}^{\text{CP}}$ used to produce one unit of material $k' \in \mathcal{K} \setminus \mathcal{K}^{\text{CP}}$ via material conversion\\
    $\Delta^{\text{NB}}_{i,k}$ & Amount of material $k \in \mathcal{K}$ used to produce a unit mass of battery of chemistry $i \in \mathcal{I}$ \\
    $\Delta^{\text{REC}}_{k,i,j}$ & Amount of material $k \in \mathcal{K}$ generated by recycling a unit mass of battery of chemistry $i \in \mathcal{I}$ with process $j \in \mathcal{J}$\\
    $\gamma$ & Annual discount factor\\
    $\eta$ & Proportion of material value recovered when selling recycled material to the market\\
    $\rho$ & Proportion of material value incurred as annual inventory cost\\
    $\sigma_t$ & Stage associated with time period $t \in \mathcal{T} \cup \{0\}$\\
    \hline
\end{xltabular}

\begin{xltabular}{\linewidth}{l|X}
    \caption{Variables}\label{tbl:modelVars} \\
    \hline
    $x^{\text{CP,INV}}_{\omega,z,t,k}$ & Amount of output material $k \in \mathcal{K}^{\text{CP}}$ from cathode production put in inventory in time period $t \in \mathcal{T}$, zone $z \in \mathcal{Z}$ and node $\omega \in \Omega_{\sigma_t}$\\
    $x^{\text{INV}}_{\omega,z,t,k}$ & Amount of recycled material $k \in \mathcal{K}$ in inventory in time period $t \in \mathcal{T}$, zone $z \in \mathcal{Z}$ and node $\omega \in \Omega_{\sigma_t}$\\
    $x^{\text{INV,MC}}_{\omega,z,t,k}$ & Amount of recycled material $k \in \mathcal{K} \setminus \mathcal{K}^{\text{CP}}$ used in material conversion in time period $t \in \mathcal{T}$, zone $z \in \mathcal{Z}$ and node $\omega \in \Omega_{\sigma_t}$\\
    $x^{\text{INV,NB}}_{\omega,z,t,k}$ & Amount of recycled material $k \in \mathcal{K}^{\text{CP}}$ used to produce new batteries in time period $t \in \mathcal{T}$, zone $z \in \mathcal{Z}$ and node $\omega \in \Omega_{\sigma_t}$\\
    $x^{\text{MC,CP}}_{\omega,z,t,k}$ & Amount of output material $k \in \mathcal{K} \setminus \mathcal{K}^{\text{CP}}$ from material conversion used for cathode production in time period $t \in \mathcal{T}$, zone $z \in \mathcal{Z}$ and node $\omega \in \Omega_{\sigma_t}$\\
    $x^{\text{NM,CP}}_{\omega,z,t,k}$ & Amount of new material $k \in \mathcal{K} \setminus \mathcal{K}^{\text{CP}}$ used in cathode production in time period $t \in \mathcal{T}$, zone $z \in \mathcal{Z}$ and node $\omega \in \Omega_{\sigma_t}$\\
    $x^{\text{NM,NB}}_{\omega,z,t,k}$ & Amount of new material $k \in \mathcal{K}$ used to produce new batteries in time period $t \in \mathcal{T}$, zone $z \in \mathcal{Z}$ and node $\omega \in \Omega_{\sigma_t}$\\
    $x^{\text{RB}}_{\omega,z,t,i}$ & Amount of retired battery of chemistry $i \in \mathcal{I}$ in inventory in time period $t \in \mathcal{T}$, zone $z \in \mathcal{Z}$ and node $\omega \in \Omega_{\sigma_t}$\\
    $x^{\text{RB,RM}}_{\omega,z,t,i,j}$ & Amount of retired battery of chemistry $i \in \mathcal{I}$ recycled with process $j \in \mathcal{J}$ in time period $t \in \mathcal{T}$, zone $z \in \mathcal{Z}$ and node $\omega \in \Omega_{\sigma_t}$\\
    $x^{\text{RM,INV}}_{\omega,z,t,k}$ & Amount of output material $k \in \mathcal{K}$ from battery recycling put in inventory in time period $t \in \mathcal{T}$, zone $z \in \mathcal{Z}$ and node $\omega \in \Omega_{\sigma_t}$\\
    $x^{\text{RM,S}}_{\omega,z,t,k}$ & Amount of output material $k \in \mathcal{K}$ from battery recycling sold to the market in time period $t \in \mathcal{T}$, zone $z \in \mathcal{Z}$ and node $\omega \in \Omega_{\sigma_t}$\\
    $x^{\text{TR,RB}}_{\omega,z,z',t,i}$ & Amount of retired battery with chemistry $i \in \mathcal{I}$ transported from zone $z \in \mathcal{Z}$ to zone $z' \in \mathcal{Z} \setminus \{z\}$ in time period $t \in \mathcal{T}$ and node $\omega \in \Omega_{\sigma_t}$\\
    $x^{\text{TR,RM}}_{\omega,z,z',t,k}$ & Amount of recycled material $k \in \mathcal{K}$ transported from zone $z \in \mathcal{Z}$ to zone $z' \in \mathcal{Z} \setminus \{z\}$ in time period $t \in \mathcal{T}$ and node $\omega \in \Omega_{\sigma_t}$\\
    $y^{\text{CP}}_{z,l,k,n}$ & Annual capacity of cathode production line $n \in \mathcal{N}^{\text{CP}}_{l,k}$ producing material $k \in \mathcal{K}^{\text{CP}}$ in zone $z \in \mathcal{Z}$ during planning period $l \in \mathcal{L}$ \\
    $y^{\text{REC}}_{z,l,j,n}$ & Annual capacity of recycling process $j \in \mathcal{J}$ at facility $n \in \mathcal{N}^{\text{REC}}_l$ in zone $z \in \mathcal{Z}$ during planning period $l \in \mathcal{L}$\\
    $y^{\text{CP}}_{z,l,k}$ & Number of cathode production lines producing material $k \in \mathcal{K}^{\text{CP}}$ in zone $z \in \mathcal{Z}$ during planning period $l \in \mathcal{L}$ with capacity at the upper bound $u^{\text{CP}}$\\
    $y^{\text{REC}}_{z,l,j}$ & Number of facilities in zone $z \in \mathcal{Z}$ during planning period $l \in \mathcal{L}$ with capacity for recycling process $j \in \mathcal{J}$ at the upper bound $u^{\text{REC}}$\\
    $y^{\text{CP}}_{z,l,k,+}$ & Annual capacity of the cathode production line producing material $k \in \mathcal{K}^{\text{CP}}$ in zone $z \in \mathcal{Z}$ during planning period $l \in \mathcal{L}$ that does not have  capacity at its upper or lower bound\\
    $y^{\text{REC}}_{z,l,j,+}$ & Annual capacity of the recycling facility for process $j \in \mathcal{J}$ in zone $z \in \mathcal{Z}$ during planning period $l \in \mathcal{L}$ that does not have capacity at its upper or lower bound\\
    \hline
\end{xltabular}

\subsection{Battery Mass Parameter Calculation}
\label{sec:ecCCParameter}

The parameter $m_{\xi,t,i}$ gives the average mass of a new battery by year, demand scenario, and chemistry, accounting for differences in battery size between BEVs and PHEVs.  Let $BEV^{\text{SDS}}_t$ and $BEV^{\text{STEPS}}_t$ represent the proportion of EV stock composed of BEVs under the SDS and STEPS projections, respectively, using data from \citeappendix{xu2020futureec}.  Note that $1 - BEV^{\text{SDS}}_t$ gives the proportion composed of PHEVs.
Similarly to the scenario-adjusted cost distribution parameters $(v^{\text{MED}}_{\xi,t,m},v^{\text{LB}}_{\xi,t,m},v^{\text{UB}}_{\xi,t,m})$, we compute samples for $BEV_t$ from the scenario multipliers $\lambda_{\xi,t}$:
$$BEV_{\xi,t} = BEV_t^{\text{SDS}} \lambda_{\xi,t} + BEV_t^{\text{STEPS}} (1 - \lambda_{\xi,t}).$$  We then convert the proportion of battery stock $BEV_{\xi,t}$ to a proportion of new battery sales $\overline{BEV}_{\xi,t}$ recursively by
$$\overline{BEV}_{\xi,t} = \frac{BEV_{\xi,t} EV_{\xi,t} - BEV_{\xi,t-1} EV_{\xi,t-1} + \sum_{a=1}^A \overline{BEV}_{\xi,t-a} LS_a AB_{\xi,t,a}}{AB_{\xi,t,0}}.$$
Note that historical BEV proportions are available through year $-A$.  Let $m^{\text{BEV}}_i$ and $m^{\text{PHEV}}_i$ be the average mass of BEV and PHEV batteries for each chemistry \citepappendix{xu2020futureec}.  The average mass of a new battery is then computed
$$m_{\xi,t,i} = m^{\text{BEV}}_i \overline{BEV}_{\xi,t} + m^{\text{PHEV}}_i (1 - \overline{BEV}_{\xi,t}).$$

\subsection{Sensitivity to Zonal Cost Parameters}
\label{sec:ecZoneCost}

\begin{table}[htp]
    \centering
    \begin{tabular}{r|c|c}
        & U.S. & China\\
        \hline
        Equipment Cost Multiplier (\%) & 100 & 60\\
        Labor (\$/hr) & 20.00 & 3.00\\
        Electricity (\$/MWh) & 68.00 & 88.00\\
        Natural Gas (\$/MMBtu) & 3.84 & 12.00\\
        Water (\$/1000 gal) & 5.00 & 2.00\\
        Landfill (\$/ton) & 55.36 & 10.00\\
        Wastewater (\$/1000 gal) & 7.00 & 3.00
    \end{tabular}
    \caption{Baseline values for location-dependent cost parameters by zone.
    }
    \label{tbl:defaultLocationValues}
\end{table}

Cost parameters vary across zone due to differences in seven location-dependent costs: labor, electricity, natural gas, water, landfill, wastewater, and an equipment cost multiplier.  All planning and operational cost objects that are indexed by zone are functions of these values.  Table~\ref{tbl:defaultLocationValues} gives the baseline values for these parameters from the EverBatt model \citepappendix{dai2019everbattec}.

To investigate the robustness of our model and reported solutions against these parameters, we perform a sensitivity analysis on Case~2, variation DR1.  We perturb the baseline value for each parameter in the U.S. and China individually by $\pm10\%$, $\pm20\%$, and $\pm40\%$, and report the relative change in objective value and in installed capacity for recycling and cathode production facilities.  For baseline solution $\overline{y}$ and solution $y$ from the perturbed model, the relative change in recycling capacity for zone $z$ is computed by 
$$100 \times \frac{\sum_{j \in \mathcal{J},\, l \in \mathcal{L}} u^{\text{REC}} (\overline{y}^{\text{REC}}_{z,l,j} - y^{\text{REC}}_{z,l,j}) + \overline{y}^{\text{REC}}_{z,l,j,+} - y^{\text{REC}}_{z,l,j,+} }{\sum_{j \in \mathcal{J},\, l \in \mathcal{L}, z' \in \mathcal{Z}} u^{\text{REC}} \overline{y}^{\text{REC}}_{z',l,j} + \overline{y}^{\text{REC}}_{z',l,j,+}}.$$
Relative changes in cathode production capacity are computed analogously.

Table~\ref{tbl:sensitivityObjective} reports the relative change in objective value for parameters that yield changes over $0.05\%$ for some tested perturbation.  The optimal objective is very stable against perturbations to the location-dependent cost parameters.  In fact, perturbations to only two of the 14 parameters cause objective changes above $0.05\%$, namely, the labor cost and equipment cost multiplier in China.  The optimal objective is most sensitive to these parameters, although the relative change never exceeds $\pm1\%$.

\begin{table}[p]
    \centering
    \begin{tabular}{c|r|c|c|c|c|c|c}
        & & \multicolumn{6}{c}{Perturbation}\\
        Zone & Parameter & -40\% & -20\% & -10\% & +10\% & +20\% & +40\%\\
        \hline
        \multirow{2}{*}{China}
        & Equipment & -0.7\% & -0.3\% & -0.1\% & 0.1\% & 0.2\% & 0.4\%\\
        & Labor & -0.1\% & 0.0\% & 0.0\% & 0.0\% & 0.0\% & 0.1\%\\
    \end{tabular}
    \caption{Relative change in optimal objective value under select location-dependent cost parameter perturbations.   All other perturbations yield relative changes under $0.05\%$.
    }
    \label{tbl:sensitivityObjective}
\end{table}

Tables~\ref{tbl:sensitivityRecycling}~and~\ref{tbl:sensitivityCapacity} give the relative change in recycling and cathode production capacity, respectively, under each parameter perturbation.  Some differences in the reported optimal solutions may be attributed to the use of a $0.01\%$ relative optimality tolerance in the objective value, which potentially enforces a weaker tolerance in the solution space.  In general, the solutions show a decrease in U.S. capacity and an increase in China capacity when cost parameters increase in the U.S., and the opposite trends when cost parameters increase in China.  The recycling solution is more stable against perturbations than the cathode production solution.  Relative changes in recycling capacity generally do not exceed $10\%$, even when the cost parameters are perturbed by $\pm 40\%$.  Cathode production capacity never changes more than $0.5\%$ in the U.S., meaning that under all parameter perturbations, almost no U.S. cathode production capacity is constructed.  On the other hand, cathode production capacity in China exhibits higher sensitivity, varying by up to $67\%$ from the baseline solution.  This result suggests that, as the objective value remains stable, there are many near-optimal cathode production solutions that yield similar objective values.  Both recycling and cathode production solutions are most sensitive to labor cost and the equipment cost multiplier.  Perturbations to the other cost parameters never cause large changes in the capacity solutions.  Overall, these results show that our recycling capacity solution is resilient to changes in location-dependent cost parameters, and the cathode production capacity solution is resilient to changes in all location-dependent cost parameters except the equipment cost multiplier in China.

\begin{table}[p]
    \centering
    \begin{subtable}[t]{\textwidth}
        \begin{tabular}{c|r|cc|cc|cc|cc|cc|cc}
        & & \multicolumn{12}{c}{Perturbation}\\
        & & \multicolumn{2}{c|}{-40\%} & \multicolumn{2}{c|}{-20\%} & \multicolumn{2}{c|}{-10\%} & \multicolumn{2}{c|}{+10\%} & \multicolumn{2}{c|}{+20\%} & \multicolumn{2}{c}{+40\%}\\
        \hline
        Zone & Parameter & U.S. & C.N. & U.S. & C.N. & U.S. & C.N. & U.S. & C.N. & U.S. & C.N. & U.S. & C.N.\\
        \hline
        \multirow{7}{*}{U.S.}
        & Equipment & 9.7 & -8.6 & 5.5 & -4.4 & 3.8 & -2.7 & -2.4 & 3.2 & -2.4 & 3.3 & -2.4 & 3.7\\
        & Labor & 10.0 & -9.8 & 3.5 & -2.5 & 3.4 & -2.7 & -2.3 & 3.4 & -2.4 & 3.6 & -2.4 & 3.7\\
        & Electricity & 3.9 & -2.3 & 5.1 & -4.6 & 5.4 & -2.2 & 0.2 & -0.6 & -2.4 & 2.7 & -2.4 & 3.0\\
        & Natural Gas & 0.5 & 0.7 & 1.3 & 0.5 & 0.0 & 1.3 & 1.0 & 0.3 & 0.5 & 0.6 & 1.1 & 0.0\\
        & Water & -1.3 & 1.0 & 0.0 & 2.6 & 0.3 & 0.2 & 0.8 & -0.1 & 1.3 & -0.4 & -0.8 & 1.8\\
        & Landfill & 0.2 & 1.4 & 2.3 & -1.1 & 0.5 & 1.2 & 3.2 & -2.4 & 1.6 & -0.1 & -1.3 & 2.2\\
        & Wastewater & 2.4 & -0.8 & 1.9 & -0.7 & 2.4 & -2.7 & -0.8 & 1.7 & -0.1 & 1.5 & 0.3 & 1.0\\
        \hline
        \multirow{7}{*}{China}
        & Equipment & -2.4 & 6.4 & -2.4 & 6.9 & -1.3 & 3.4 & 4.9 & -4.9 & 5.6 & -7.5 & 8.9 & -11.2\\
        & Labor & -2.4 & 3.1 & 2.1 & -1.2 & 0.9 & 1.3 & 3.0 & -2.9 & 0.3 & 0.8 & 1.3 & -2.7\\
        & Electricity & -2.4 & 3.4 & -2.1 & 3.6 & -0.3 & 1.7 & 2.1 & -2.8 & 4.8 & -4.5 & 5.4 & -4.6\\
        & Natural Gas & 0.3 & 1.4 & 0.3 & 0.5 & 0.5 & 0.7 & 2.5 & -1.8 & 4.6 & -3.4 & -1.1 & 2.2\\
        & Water & 1.5 & -0.3 & 1.2 & -1.3 & 0.5 & 0.8 & 0.8 & -0.8 & 3.0 & -2.1 & 0.8 & 1.3\\
        & Landfill & 0.3 & 2.4 & -1.3 & 2.6 & 0.0 & 2.0 & 2.6 & -2.1 & 0.5 & -0.7 & 1.6 & -1.2\\
        & Wastewater & 0.0 & -0.6 & 0.8 & -1.5 & 1.1 & 0.0 & 2.7 & -1.4 & 0.3 & 1.4 & 1.3 & 0.0\\
    \end{tabular}
    \captionsetup{font=footnotesize}
    \caption{Relative change (\%) to recycling capacity under cost parameter perturbations.
    }
    \label{tbl:sensitivityRecycling}
    \end{subtable}
    \newline
    \vspace{0.6em}
    \begin{subtable}[t]{\textwidth}
    \begin{tabular}{c|r|cc|cc|cc|cc|cc|cc}
        & & \multicolumn{12}{c}{Perturbation}\\
        & & \multicolumn{2}{c|}{-40\%} & \multicolumn{2}{c|}{-20\%} & \multicolumn{2}{c|}{-10\%} & \multicolumn{2}{c|}{+10\%} & \multicolumn{2}{c|}{+20\%} & \multicolumn{2}{c}{+40\%}\\
        \hline
        Zone & Parameter & U.S. & C.N. & U.S. & C.N. & U.S. & C.N. & U.S. & C.N. & U.S. & C.N. & U.S. & C.N.\\
        \hline
        \multirow{7}{*}{U.S.}
        & Equipment & 0.3 & 0.4 & 0.0 & -3.8 & 0.0 & -1.4 & 0.0 & -3.8 & 0.0 & -2.1 & 0.0 & -5.4\\
        & Labor & 0.0 & -2.4 & 0.0 & -3.4 & 0.0 & -3.7 & 0.0 & -3.5 & 0.0 & -4.6 & 0.0 & 5.8\\
        & Electricity & 0.0 & -4.5 & 0.0 & -1.1 & 0.3 & -3.7 & 0.0 & -0.7 & 0.0 & 3.4 & 0.0 & 1.4\\
        & Natural Gas & 0.0 & -4.7 & 0.0 & -1.1 & 0.0 & -6.8 & 0.0 & -3.0 & 0.0 & -2.1 & 0.3 & -4.8\\
        & Water & 0.0 & -0.3 & 0.0 & 1.4 & 0.0 & 0.4 & 0.0 & 0.8 & 0.0 & -4.9 & 0.0 & 3.8\\
        & Landfill & 0.0 & 1.6 & 0.0 & -4.1 & 0.1 & -5.0 & 0.0 & -2.3 & 0.0 & -1.2 & 0.0 & -0.2\\
        & Wastewater & 0.1 & -1.4 & 0.0 & -2.4 & 0.0 & 0.9 & 0.0 & 2.7 & 0.0 & 6.2 & 0.0 & 0.0\\
        \hline
        \multirow{7}{*}{China}
        & Equipment & 0.0 & 66.1 & 0.0 & 30.1 & 0.0 & 13.0 & 0.0 & -19.2 & 0.0 & -27.6 & 0.0 & -67.1\\
        & Labor & 0.0 & 9.1 & 0.0 & 3.1 & 0.3 & 0.8 & 0.0 & -0.6 & 0.0 & 2.7 & 0.0 & -10.5\\
        & Electricity & 0.0 & -1.2 & 0.0 & 0.1 & 0.0 & -3.0 & 0.0 & -1.7 & 0.0 & -1.2 & 0.0 & -1.4\\
        & Natural Gas & 0.0 & -1.0 & 0.0 & 1.7 & 0.3 & 2.8 & 0.0 & -0.7 & 0.0 & -4.0 & 0.0 & -6.0\\
        & Water & 0.0 & 0.2 & 0.0 & -0.3 & 0.0 & -1.7 & 0.0 & 2.1 & 0.0 & -4.1 & 0.0 & -1.4\\
        & Landfill & 0.0 & 0.7 & 0.3 & 2.8 & 0.0 & -2.3 & 0.0 & -0.7 & 0.0 & -4.0 & 0.0 & 0.7\\
        & Wastewater & 0.0 & 1.3 & 0.0 & -3.8 & 0.0 & -4.9 & 0.4 & -0.5 & 0.0 & -5.6 & 0.0 & 2.8\\
    \end{tabular}
    \captionsetup{font=footnotesize}
    \caption{Relative change (\%) to cathode production capacity under cost parameter perturbations.
    }
    \label{tbl:sensitivityCapacity}
    \end{subtable}
    \vspace{0.4em}
    \caption{Sensitivity of optimal facility capacity in the U.S. and China (C.N.) to perturbations in zonal cost parameters.  Each row gives results for perturbations to one parameter in one zone.}
\end{table}

\subsection{Proofs of Results}
\label{sec:ecProofs}

\setcounter{repeatproposition}{0}
\begin{repeatproposition}
    For any $\overline{y} \geq 0$ that satisfies \eqref{eq:constrCapRECIncrease}-\eqref{eq:constrCapCPUB}, $\{(y,x) \geq 0 : \eqref{eq:constrProd},\ \eqref{eq:constrMatFlow},\ \eqref{eq:constrCap},\ y = \overline{y}\} \neq \emptyset$.
\end{repeatproposition}

\begin{proof}
    Fix $y \geq 0$ so that \eqref{eq:constrCapRECIncrease}-\eqref{eq:constrCapCPUB} are satisfied.  We construct a corresponding feasible $x$ that only buys new materials to manufacture batteries and does not recycle.  Take $x^{\text{NM,NB}}_{\omega,z,t,k} = \sum_{i \in \mathcal{I}} \Delta^{\text{NB}}_{i,k} d_{\omega,z,t,i}$.  Further, set $x^{\text{RB}}_{\omega,z,t,i} = \sum_{t' = 1}^t s_{a_\omega(t'),z,t',i}$ and $x^{\text{RB}}_{\omega,z,0,i} = 0$.  Note that $a_\omega(t) = \omega$ for $\omega \in \Omega_{\sigma_t}$.  Set all other components of $x$ to $0$.  Because battery supply $s$, demand $d$, and material requirements $\Delta$ are nonnegative under Assumption~\ref{assump:dataProperties}, $x \geq 0$.

    The construction of $x^{\text{NM,NB}}$ ensures that \eqref{eq:constrNBProdCP} and \eqref{eq:constrNBProdNoCP} are satisfied, and that of $x^{\text{RB}}$ ensures that \eqref{eq:constrRBMatFlow} is satisfied.  The remainder of constraints in \eqref{eq:constrProd} and \eqref{eq:constrMatFlow} only involve components of $x$ that are $0$ and are thus satisfied.  As $y \geq 0$, the solution $x$ is feasible for \eqref{eq:constrCapREC} and \eqref{eq:constrCapCP}.  Thus, given $y$, we can construct $x \geq 0$ so that constraints \eqref{eq:constrProd}-\eqref{eq:constrCap} are satisfied. 
\end{proof}

\begin{lemma}[\citeappendix{horst2000introductionec}, Theorem 1.19]
\label{lemma:horstExtPt}
    Let $\mathcal{X}$ be a bounded polytope and $f\,:\,\mathcal{X} \rightarrow \mathbb{R}$ be a concave function.  Then, $f$ attains its minimum at an extreme point of $\mathcal{X}$.
\end{lemma}

\begin{proof}
    Let $\{x^{(j)}\}_{j=1}^J$ be the extreme points of $\mathcal{X}$.  As $\mathcal{X}$ is bounded, $\mathcal{X} = \conv \{x^{(j)}\}_{j=1}^J$.  Then, for any $x \in \mathcal{X}$, there is some $\lambda \geq 0$ with $\sum_{j=1}^J \lambda_j = 1$ and $\sum_{j=1}^J \lambda_j x^{(j)} = x$.  By concavity of $f$,
    \begin{equation*}
        \begin{aligned}
            f(x) & = f \left (\sum_{j=1}^J \lambda_j x^{(j)} \right ) \geq  \sum_{j=1}^J \lambda_j f(x^{(j)}) \geq \min_{j \in \lBrack J \rBrack} f(x^{(j)}) \left ( \sum_{j=1}^J \lambda_j \right ) = \min_{j \in \lBrack J \rBrack} f(x^{(j)}).
        \end{aligned}
    \end{equation*}
    Therefore, $f$ achieves its minimum on $\mathcal{X}$ at some element of $\{x^{(j)}\}_{j=1}^J$, which is an extreme point of $\mathcal{X}$.
\end{proof}

\begin{lemma}
    \label{lemma:concMinFinite}
    Let $f$ be a concave function and $\mathcal{X} \subseteq \mathbb{R}^p_+$ be a nonempty polyhedron.  Then, if $\inf_{x \in \mathcal{X}} f(x) > - \infty$, it attains a finite optimal solution at an extreme point of $\mathcal{X}$.
\end{lemma}

\begin{proof}
Suppose $\inf_{x \in \mathcal{X}} f(x) > - \infty$.  Let $\{x^{(j)}\}_{j=1}^J$ and $\{r^{(k)}\}_{k=1}^K$ be the extreme points and extreme rays of $\mathcal{X}$, and denote $\epsilon = \inf_{x \in \mathcal{X}} f(x)$, which is finite as the problem is feasible.  Note that 
$$\mathcal{X} = \conv \{x^{(j)}\}_{j=1}^J + \cone \{r^{(k)}\}_{k=1}^K.$$
Further, as $\mathcal{X} \subseteq \mathbb{R}^p_+$, it does not contain a line, which implies that it has as least one extreme point.

Select any $x \in \conv \{x^{(j)}\}_{j=1}^J$ and $r \in \cone \{r^{(k)}\}_{k=1}^K$. To prove by contradiction, suppose that $f(x + r) < f(x)$.  Take $\delta = \frac{f(x) - \epsilon + 1}{f(x) - f(x + r)}$ and note that $\delta \geq 1$ as $\epsilon \leq f(x + r)$.  Then, by concavity of $f$,
$$f(x + r) = f \left ( \frac{1}{\delta} (x + \delta r) + \frac{\delta - 1}{\delta} x \right ) \geq \frac{1}{\delta} f(x + \delta r) + \frac{\delta - 1}{\delta} f(x),$$
and
$$f(x + \delta r) \leq \delta \left ( f(x+r) - f(x) \right ) + f(x) = \epsilon - 1.$$
As $x + \delta r \in \mathcal{X}$, this contradicts the definition of $\epsilon$, and thus $f(x) \leq f(x + r)$.  This shows that 
$$\inf_{x \in \mathcal{X}} f(x) \quad = \quad \inf_{x \in \conv \{x^{(j)}\}_{j=1}^J} f(x).$$
The latter is a concave minimization over a bounded polytope, and thus attains its optimal value $\epsilon$ at an extreme point of $\conv \{x^{(j)}\}_{j=1}^J$ by Lemma~\ref{lemma:horstExtPt}. Extreme points of the convex hull of a discrete set are elements of the set, and these elements are extreme points of $\mathcal{X}$.  Thus, the optimal value $\epsilon$ is obtained at an extreme point of $\mathcal{X}$.
\end{proof}

\setcounter{repeatproposition}{1}
\begin{repeatproposition}
   The model \eqref{eq:deterministicP} is feasible and has an optimal solution at an extreme point of its feasible region.
\end{repeatproposition}

\begin{proof}
    We first establish that \eqref{eq:deterministicP} is feasible.  Take $y = 0$.  As $u^{\text{REC}} > 0$ and $u^{\text{CP}} > 0$ under Assumption~\ref{assump:dataProperties}, $y$ is feasible for \eqref{eq:constrCapRECIncrease}-\eqref{eq:constrCapCPUB}.  By Proposition~\ref{prop:fullRecourse}, there is some $x \geq 0$ so that $(y,x)$ satisfies \eqref{eq:constrProd}, \eqref{eq:constrMatFlow}, and \eqref{eq:constrCap}.  Then, $(y,x)$ is feasible for \eqref{eq:deterministicP}.
    
    Next, consider the relaxation of \eqref{eq:deterministicP}:
    \begin{equation}
        \label{eq:relaxPrimal}
        \begin{aligned}
            \min_{x,y} \quad & \sum_{t \in \mathcal{T}} (1 - \gamma)^{t-1} \left ( C^{\text{PL}}_t(y) + \sum_{\omega \in \Omega_{\sigma_t}} p_\omega C^{\text{OP}}_{\omega,t}(x) \right ) \\
            \text{s.t.} \quad & \eqref{eq:constrRECProd},\ \eqref{eq:constrRBMatFlowInit},\ \eqref{eq:constrRBMatFlow}\\
            & x \geq 0,\ y \geq 0.
        \end{aligned}
    \end{equation}
    Any feasible solution for \eqref{eq:deterministicP} is feasible for \eqref{eq:relaxPrimal} with the same objective value.  Note that, as constraints \eqref{eq:constrCap} are relaxed, there are no constraints coupling variables $x$ and $y$.  Thus, \eqref{eq:relaxPrimal} is equivalent to
    \begin{equation}
        \label{eq:relaxedPrimalDecomposed}
        \begin{aligned}
            \min_{y \geq 0} \quad & \sum_{t \in \mathcal{T}} (1 - \gamma)^{t-1} C^{\text{PL}}_t(y) \quad + \quad & \min_{x \geq 0} \quad & \sum_{t \in \mathcal{T}} (1 - \gamma)^{t-1} \sum_{\omega \in \Omega_{\sigma_t}} p_\omega C^{\text{OP}}_{\omega,t}(x) \\
            & & \text{s.t.} \quad & \eqref{eq:constrRECProd},\ \eqref{eq:constrRBMatFlowInit},\ \eqref{eq:constrRBMatFlow}.
        \end{aligned}
    \end{equation} 
    As $C^{\text{PL}}_t$ is monotonic increasing in each variable $y$ by Assumption~\ref{assump:dataProperties}, $y = 0$ is optimal for the first problem of \eqref{eq:relaxedPrimalDecomposed}.
    
    Consider any direction $x \geq 0$ in which the feasible region of the second problem of \eqref{eq:relaxedPrimalDecomposed} is unbounded.  By \eqref{eq:constrRBMatFlow}, 
    $$x^{\text{RB}}_{\omega,z,t,i} - x^{\text{RB}}_{a_{\omega}(t-1),z,t-1,i} + \sum_{j \in \mathcal{J}} x^{\text{RB,RM}}_{\omega,z,t,i,j} - \sum_{z' \in \mathcal{Z} \setminus \{z\}} \left ( x^{\text{TR,RB}}_{\omega,z',z,t,i} - x^{\text{TR,RB}}_{\omega,z,z',t,i} \right ) = 0.$$  Summing over $z$ yields $\sum_{z \in \mathcal{Z}} \left ( x^{\text{RB}}_{\omega,z,t,i} - x^{\text{RB}}_{a_{\omega}(t-1),z,t-1,i} + \sum_{j \in \mathcal{J}} x^{\text{RB,RM}}_{\omega,z,t,i,j} \right ) = 0$.  As $x^{\text{RB}}_{\omega,z,0,i} = 0$ and $x \geq 0$, by induction $x^{\text{RB}} = 0$ and $x^{\text{RB,RM}} = 0$.  Then, \eqref{eq:constrRECProd} gives $x^{\text{RM,INV}}_{\omega,z,t,k} + x^{\text{RM,S}}_{\omega,z,t,k} - \sum_{i \in \mathcal{I},\ j \in \mathcal{J}} \Delta^{\text{REC}}_{k,i,j} x^{\text{RB,RM}}_{\omega,z,t,i,j} = 0$, so $x^{\text{RM,INV}} = 0$ and $x^{\text{RM,S}} = 0$ as $x \geq 0$.

    Then, consider the change in objective in direction $x$, given by
    $\sum_{t \in \mathcal{T}} (1 - \gamma)^{t-1} \sum_{\omega \in \Omega_{\sigma_t}} p_\omega C^{\text{OP}}_{\omega,t}(x).$
    As $\gamma < 1$, $c \geq 0$, $v \geq 0$, $p \geq 0$, $x \geq 0$, and $x^{\text{RM,S}} = 0$, it holds that $\sum_{t \in \mathcal{T}} (1 - \gamma)^{t-1} \sum_{\omega \in \Omega_{\sigma_t}} p_\omega C^{\text{OP}}_{\omega,t}(x) \geq 0$, and the direction does not improve the objective.  As the second problem of \eqref{eq:relaxedPrimalDecomposed} is a linear program, this proves that its objective is bounded.  Thus, neither problem in \eqref{eq:relaxedPrimalDecomposed} is unbounded, so \eqref{eq:relaxPrimal} and \eqref{eq:deterministicP} are bounded in objective.  As \eqref{eq:deterministicP} is a concave minimization over a polyhedron contained in the nonnegative orthant and is bounded in objective, it attains its optimal value at an extreme point of its feasible region by Lemma~\ref{lemma:concMinFinite}.
\end{proof}

\setcounter{repeattheorem}{0}
\begin{repeattheorem}
    Let $(\tilde{y},\tilde{x})$ be an extreme point of $\{(y,x) \geq 0\ :\ \eqref{eq:constrProd},\ \eqref{eq:constrMatFlow},\ \eqref{eq:constrCap}\}$.  Then, $\tilde{y}^{\text{REC}}$ satisfies \eqref{eq:recExtPtProperty} and $\tilde{y}^{\text{CP}}$ satisfies the corresponding property.
\end{repeattheorem}

\begin{proof}
    The comparable solution structure for $\tilde{y}^{\text{CP}}$ with total capacity $Y_{z,l,k} = \sum_{n \in \mathcal{N}^{\text{CP}}_{l,k}} \tilde{y}^{\text{CP}}_{z,l,k,n}$ is given by
    \begin{equation}
        \label{eq:CPExtPtProperty}
        \begin{aligned}
            \mathcal{C}(\{\tilde{y}^{\text{CP}}_{z,l,k,n}\}_{n \in \mathcal{N}^{\text{CP}}_{l,k}};u^{\text{CP}}) & \geq \left \lceil \frac{Y_{z,l,k}}{u^{\text{CP}}} \right \rceil - 1;\\
            \mathcal{C}(\{\tilde{y}^{\text{CP}}_{z,l,k,n}\}_{n \in \mathcal{N}^{\text{CP}}_{l,k}};0) & \geq |\mathcal{N}^{\text{CP}}_{l,k}| - \left \lceil \frac{Y_{z,l,k}}{u^{\text{CP}}} \right \rceil \quad \forall k \in \mathcal{K}^{\text{CP}},\ l \in \mathcal{L},\ z \in \mathcal{Z}.
        \end{aligned}
    \end{equation}
    We demonstrate the result only for the property of the recycling capacity solution $\tilde{y}^{\text{REC}}$, as the same arguments apply to the cathode production capacity solution $\tilde{y}^{\text{CP}}$.  The result is shown by contraposition.  Denote the polyhedron $\mathcal{X} = \{(y,x) \geq 0\ :\ \eqref{eq:constrProd},\ \eqref{eq:constrMatFlow},\ \eqref{eq:constrCap}\}$ and consider any $(y,x) \in \mathcal{X}$.  Fix some $z,l,j$ and let the total recycling capacity be $Y_{z,l,j} = \sum_{n \in \mathcal{N}^{\text{REC}}_l} y^{\text{REC}}_{z,l,j,n}$.  Suppose that \eqref{eq:recExtPtProperty} does not hold for $y^{\text{REC}}$.
    
    First, consider the case where $\mathcal{C}(\{y^{\text{REC}}_{z,l,j,n}\}_{n \in \mathcal{N}^{\text{REC}}_l};u^{\text{REC}}) \leq \left \lceil \frac{Y_{z,l,j}}{u^{\text{REC}}} \right \rceil - 2$.  Suppose for contradiction that there are one or fewer indices $n$ with $u^{\text{REC}} > y^{\text{REC}}_{z,l,j,n} > 0$.  Then, the total number of indices corresponding to nonzero elements is bounded by $\left \lceil \frac{Y_{z,l,j}}{u^{\text{REC}}} \right \rceil - 1$, and
    $$\sum_{n \in \mathcal{N}^{\text{REC}}_l} y^{\text{REC}}_{z,l,j,n} \leq u^{\text{REC}} \left ( \left \lceil \frac{Y_{z,l,j}}{u^{\text{REC}}} \right \rceil - 1 \right ) < Y_{z,l,j} = \sum_{n \in \mathcal{N}^{\text{REC}}_l} y^{\text{REC}}_{z,l,j,n}.$$
    By this contradiction, there must be at least two indices $n_1$, $n_2$ with $u^{\text{REC}} > y^{\text{REC}}_{z,l,j,n} > 0$ for $n \in \{n_1,n_2\}$.  
    
    Second, consider the case where $\mathcal{C}(\{y^{\text{REC}}_{z,l,j,n}\}_{n \in \mathcal{N}^{\text{REC}}_l};0) \leq |\mathcal{N}^{\text{REC}}_l| - \left \lceil \frac{Y_{z,l,j}}{u^{\text{REC}}} \right \rceil - 1$.  Again, suppose for contradiction that there are one or fewer indices $n$ with $u^{\text{REC}} > y^{\text{REC}}_{z,l,j,n} > 0$.  Then, there are at least $\left \lceil \frac{Y_{z,l,j}}{u^{\text{REC}}} \right \rceil$ elements at the upper bound $u^{\text{REC}}$ and one additional nonzero element (that may also take the value $u^{\text{REC}}$), so 
    $$\sum_{n \in \mathcal{N}^{\text{REC}}_l} y^{\text{REC}}_{z,l,j,n} > u^{\text{REC}} \left \lceil \frac{Y_{z,l,j}}{u^{\text{REC}}} \right \rceil \geq Y_{z,l,j} = \sum_{n \in \mathcal{N}^{\text{REC}}_l} y^{\text{REC}}_{z,l,j,n}.$$  By this contradiction, there must again be at least two indices $n_1$, $n_2$ with $u^{\text{REC}} > y^{\text{REC}}_{z,l,j,n} > 0$ for $n \in \{n_1,n_2\}$.
    
    In either scenario, select such indices $n_1$, $n_2$, and denote 
    $$\delta = \min \{y^{\text{REC}}_{z,l,j,n_1}, y^{\text{REC}}_{z,l,j,n_2}, u^{\text{REC}} - y^{\text{REC}}_{z,l,j,n_1}, u^{\text{REC}} - y^{\text{REC}}_{z,l,j,n_2}\}.$$  By construction, $\delta > 0$.  We construct perturbed solutions
    $$\overline{y}^{\text{REC}}_{z,l,j,n}(1) = \begin{cases}
    y^{\text{REC}}_{z,l,j,n} & n \not \in \{n_1, n_2\}\\
    y^{\text{REC}}_{z,l,j,n} + \delta & n = n_1\\
    y^{\text{REC}}_{z,l,j,n} - \delta & n = n_2,
    \end{cases}$$
    and
    $$\overline{y}^{\text{REC}}_{z,l,j,n}(2) = \begin{cases}
    y^{\text{REC}}_{z,l,j,n} & n \not \in \{n_1, n_2\}\\
    y^{\text{REC}}_{z,l,j,n} - \delta & n = n_1\\
    y^{\text{REC}}_{z,l,j,n} + \delta & n = n_2.
    \end{cases}$$
    Let $\overline{y}^{\text{CP}}(1) = \overline{y}^{\text{CP}}(2) = y^{\text{CP}}$.
    By construction, $0 \leq \overline{y}^{\text{REC}}_{z,l,j,n}(1) \leq u^\text{REC}$ and $\sum_{n \in \mathcal{N}^{\text{REC}}_l} \overline{y}^{\text{REC}}_{z,l,j,n}(1) = \sum_{n \in \mathcal{N}^{\text{REC}}_l} y^{\text{REC}}_{z,l,j,n}$, so $(\overline{y}(1), x)$ satisfies all constraints of $\mathcal{X}$.  Similarly, $(\overline{y}(2),x) \in \mathcal{X}$.  Finally, we note that $(y,x) = \frac{1}{2} (\overline{y}(1),x) + \frac{1}{2} (\overline{y}(2),x)$, so $(y,x)$ is a convex combination of other points in $\mathcal{X}$, and thus is not an extreme point.

    This shows that, if a solution $y^{\text{REC}}$ does not satisfy \eqref{eq:recExtPtProperty}, then the solution $(y,x)$ is not an extreme point of $\mathcal{X}$.  Thus, the contrapositive of this result is also true, namely, an extreme point $(\tilde{y},\tilde{x})$ of $\mathcal{X}$ must satisfy \eqref{eq:recExtPtProperty}.  By similar logic, it can be shown that an extreme point must also satisfy \eqref{eq:CPExtPtProperty}.
\end{proof}

\setcounter{repeattheorem}{1}
\begin{repeattheorem}
    The models \eqref{eq:deterministicP} and \eqref{eq:reformulatedP} have the same optimal objective value.
\end{repeattheorem}

\begin{proof}
    We first demonstrate that a feasible solution for \eqref{eq:reformulatedP} can be mapped to an equivalent feasible solution for \eqref{eq:deterministicP}.  Let $(\overline{y},\overline{x})$ be feasible for \eqref{eq:reformulatedP}.  We define $\tilde{y}$ by
    \begin{equation*}
        \begin{aligned}
            \tilde{y}^{\text{REC}}_{z,l,j,n} & = \begin{cases}
                u^{\text{REC}} & n \leq \overline{y}^{\text{REC}}_{z,l,j}\\
                \overline{y}^{\text{REC}}_{z,l,j,+} & n = \overline{y}^{\text{REC}}_{z,l,j} + 1\\
                0 & \text{otherwise},
            \end{cases}\\
            \tilde{y}^{\text{CP}}_{z,l,k,n} & = \begin{cases}
                u^{\text{CP}} & n \leq \overline{y}^{\text{CP}}_{z,l,k}\\
                \overline{y}^{\text{CP}}_{z,l,k,+} & n = \overline{y}^{\text{CP}}_{z,l,k} + 1\\
                0 & \text{otherwise}.
            \end{cases}
        \end{aligned}
    \end{equation*}
    By construction, $0 \leq \tilde{y}^{\text{REC}}_{z,l,j,n} \leq u^{\text{REC}}$ and $0 \leq \tilde{y}^{\text{CP}}_{z,l,k,n} \leq u^{\text{CP}}$.  Further, $\sum_{n \in \mathcal{N}^{\text{REC}}_{l}} \tilde{y}^{\text{REC}}_{z,l,j,n} = u^{\text{REC}} \overline{y}^{\text{REC}}_{z,l,j} + \overline{y}^{\text{REC}}_{z,l,j,+}$ and similarly $\sum_{n \in \mathcal{N}^{\text{CP}}_{l,k}} \tilde{y}^{\text{CP}}_{z,l,k,n} = u^{\text{CP}} \overline{y}^{\text{CP}}_{z,l,k} + \overline{y}^{\text{CP}}_{z,l,k,+}$.  Due to feasibility of $(\overline{y},\overline{x})$ for \eqref{eq:constrCapRef}, the solution $(\tilde{y},\overline{x})$ is feasible for constraints \eqref{eq:constrCap}.  Further, $\overline{x}$ remains feasible for \eqref{eq:constrProd} and \eqref{eq:constrMatFlow}. Thus, $(\tilde{y},\overline{x})$ is feasible for \eqref{eq:deterministicP}.  The objective values satisfy
    \begin{equation*}
        \begin{aligned}
            C^{\text{PL}}_t(\tilde{y}) & = \sum_{z \in \mathcal{Z}} \left ( \sum_{j \in \mathcal{J}} \sum_{n \in \mathcal{N}^{\text{REC}}_{l_t}} f^{\text{REC}}_{z,j}(\tilde{y}^{\text{REC}}_{z,l_t,j,n}) + \sum_{k \in \mathcal{K}^{\text{CP}}} \sum_{n \in \mathcal{N}^{\text{CP}}_{l_t,k}} f^{\text{CP}}_{z,k}(\tilde{y}^{\text{CP}}_{z,l_t,k,n}) \right )\\
            & = \sum_{z \in \mathcal{Z}} \left ( \sum_{j \in \mathcal{J}} \left ( \overline{y}^{\text{REC}}_{z,l_t,j} f^{\text{REC}}_{z,j} (u^{\text{REC}}) + f^{\text{REC}}_{z,j}(\overline{y}^{\text{REC}}_{z,l_t,j,+}) \right ) + \sum_{k \in \mathcal{K}^{\text{CP}}} \left ( \overline{y}^{\text{CP}}_{z,l_t,k} f^{\text{CP}}_{z,k}(u^{\text{CP}}) + f^{\text{CP}}_{z,k}(\overline{y}^{\text{CP}}_{z,l_t,k,+}) \right ) \right )\\
            & = \overline{C}^{\text{PL}}_t(\overline{y}),
        \end{aligned}
    \end{equation*}
    as $f(0) = 0$.  Thus $(\tilde{y},\overline{x})$ has the same objective value in \eqref{eq:deterministicP} as $(\overline{y},\overline{x})$ in \eqref{eq:reformulatedP}.

In the opposite direction, we consider an optimal solution for \eqref{eq:deterministicP} that occurs at an extreme point of its feasible region, which exists by Proposition~\ref{prop:finiteP}.  Denote this optimal extreme point $(\tilde{y},\tilde{x})$.  By Theorem~\ref{thm:extPtStructure}, this solution satisfies \eqref{eq:recExtPtProperty} and \eqref{eq:CPExtPtProperty}.

    As $u^{\text{REC}} > 0$ under Assumption~\ref{assump:dataProperties}, each index $n \in \mathcal{N}^{\text{REC}}_l$ can be counted by at most one of $\mathcal{C}(\{\tilde{y}^{\text{REC}}_{z,l,j,n}\}_{n \in \mathcal{N}^{\text{REC}}_l};0)$ and $\mathcal{C}(\{\tilde{y}^{\text{REC}}_{z,l,j,n}\}_{n \in \mathcal{N}^{\text{REC}}_l};u^{\text{REC}})$.  By this property and the structure \eqref{eq:recExtPtProperty}, 
    $$|\mathcal{N}^{\text{REC}}_l| \geq \mathcal{C}(\{\tilde{y}^{\text{REC}}_{z,l,j,n}\}_{n \in \mathcal{N}^{\text{REC}}_l};0) + \mathcal{C}(\{\tilde{y}^{\text{REC}}_{z,l,j,n}\}_{n \in \mathcal{N}^{\text{REC}}_l};u^{\text{REC}}) \geq |\mathcal{N}^{\text{REC}}| - 1.$$
    If the lower bound holds with equality, we can select the unique $\overline{n}_{z,l,j}$ so that $0 < \tilde{y}^{\text{REC}}_{z,l,j,(\overline{n}_{z,l,j})} < u^{\text{REC}}$.  Otherwise, we select arbitrary $\overline{n}_{z,l,j}$.  The index has the property that $\tilde{y}^{\text{REC}}_{z,l,j,n} \in \{0,u^{\text{REC}}\}$ for any $n \neq \overline{n}_{z,l,j}$. We can similarly select $\overline{n}_{z,l,k}$ for $\tilde{y}^{\text{CP}}$.
    Now, we define 
    \begin{equation*}
        \begin{aligned}
            & \overline{y}^{\text{REC}}_{z,l,j} = \mathcal{C}(\{\tilde{y}^{\text{REC}}_{z,l,j,n}\}_{n \in \mathcal{N}^{\text{REC}}_l \setminus \{\overline{n}_{z,l,j}\}};u^{\text{REC}});\\
            & \overline{y}^{\text{REC}}_{z,l,j,+} = \tilde{y}^{\text{REC}}_{z,l,j,\overline{n}_{z,l,j}};\\
            & \overline{y}^{\text{CP}}_{z,l,k} = \mathcal{C}(\{\tilde{y}^{\text{CP}}_{z,l,k,n}\}_{n \in \mathcal{N}^{\text{CP}}_{l,k} \setminus \{\overline{n}_{z,l,k}\}};u^{\text{REC}}); \text{ and}\\
            & \overline{y}^{\text{CP}}_{z,l,k,+} = \tilde{y}^{\text{CP}}_{z,l,k,\overline{n}_{z,l,k}}.\\
        \end{aligned}
    \end{equation*}
    By construction, $\overline{y}^{\text{REC}}_{z,l,j}$ and $\overline{y}^{\text{CP}}_{z,l,k}$ satisfy the integrality constraints \eqref{eq:constrCapRECInt} and \eqref{eq:constrCapCPInt}, and further $\overline{y}^{\text{REC}}_{z,l,j,+}$ and $\overline{y}^{\text{CP}}_{z,l,k,+}$ satisfy nonnegativity and the upper bounds \eqref{eq:constrCapRECUBRef} and \eqref{eq:constrCapCPUBRef} by the feasibility of $\tilde{y}$ for \eqref{eq:deterministicP}.  By definition of $\overline{n}$,
    $$\sum_{n \in \mathcal{N}^{\text{REC}}_{l}} \tilde{y}^{\text{REC}}_{z,l,j,n} = u^{\text{REC}} \mathcal{C}(\{\tilde{y}^{\text{REC}}_{z,l,j,n}\}_{n \in \mathcal{N}^{\text{REC}}_l \setminus \{\overline{n}_{z,l,j}\}};u^{\text{REC}}) + \tilde{y}^{\text{REC}}_{z,l,j,(\overline{n}_{z,l,j})} = u^{\text{REC}} \overline{y}^{\text{REC}}_{z,l,j} + \overline{y}^{\text{REC}}_{z,l,j,+},$$
    and similarly $$\sum_{n \in \mathcal{N}^{\text{CP}}_{l,k}} \tilde{y}^{\text{CP}}_{z,l,k,n} = u^{\text{CP}} \overline{y}^{\text{CP}}_{z,l,k} + \overline{y}^{\text{CP}}_{z,l,k,+}.$$
    Thus, feasibility of $(\tilde{y},\tilde{x})$ for \eqref{eq:constrCap} implies that $(\overline{y},\tilde{x})$ is feasible for \eqref{eq:constrCapRECRef}-\eqref{eq:constrCapCPIncreaseRef}, and $\tilde{x}$ remains feasible for \eqref{eq:constrProd} and \eqref{eq:constrMatFlow}.  Thus, $(\overline{y},\tilde{x})$ is feasible for \eqref{eq:reformulatedP}.  Considering the objective,
    \begin{equation*}
        \begin{aligned}
            \sum_{n \in \mathcal{N}^{\text{REC}}_{l}} f^{\text{REC}}_{z,j}(\tilde{y}^{\text{REC}}_{z,l,j,n}) & = \mathcal{C}(\{\tilde{y}^{\text{REC}}_{z,l,j,n}\}_{n \in \mathcal{N}^{\text{REC}}_l \setminus \{\overline{n}_{z,l,j}\}};u^{\text{REC}}) f^{\text{REC}}_{z,j}(u^{\text{REC}}) + f^{\text{REC}}_{z,j}(\tilde{y}^{\text{REC}}_{z,l,j,(\overline{n}_{z,l,j})})\\
            & = \overline{y}^{\text{REC}}_{z,l,j} f^{\text{REC}}_{z,j}(u^{\text{REC}}) + f^{\text{REC}}_{z,j}(\overline{y}^{\text{REC}}_{z,l,j,+}),
        \end{aligned}
    \end{equation*}
    as $f^{\text{REC}}_{z,j}(0) = 0$ by Assumption~\ref{assump:dataProperties}.  Again, a similar result holds for $\tilde{y}^{\text{CP}}$.  Therefore, the objectives satisfy
    \begin{equation*}
        \begin{aligned}
            C^{\text{PL}}_t(\tilde{y}) & = \sum_{z \in \mathcal{Z}} \left ( \sum_{j \in \mathcal{J}} \sum_{n \in \mathcal{N}^{\text{REC}}_{l_t}} f^{\text{REC}}_{z,j}(\tilde{y}^{\text{REC}}_{z,l_t,j,n}) + \sum_{k \in \mathcal{K}^{\text{CP}}} \sum_{n \in \mathcal{N}^{\text{CP}}_{l_t,k}} f^{\text{CP}}_{z,k}(\tilde{y}^{\text{CP}}_{z,l_t,k,n}) \right )\\
            & = \sum_{z \in \mathcal{Z}} \left ( \sum_{j \in \mathcal{J}} \overline{y}^{\text{REC}}_{z,l_t,j} f^{\text{REC}}_{z,j}(u^{\text{REC}}) + f^{\text{REC}}_{z,j}(\overline{y}^{\text{REC}}_{z,l_t,j,+}) + \sum_{k \in \mathcal{K}^{\text{CP}}} \overline{y}^{\text{CP}}_{z,l_t,k} f^{\text{CP}}_{z,k}(u^{\text{CP}}) + f^{\text{CP}}_{z,k}(\overline{y}^{\text{CP}}_{z,l_t,k,+}) \right )\\
            & = \overline{C}^{\text{PL}}_t(\overline{y}),
        \end{aligned}
    \end{equation*}
    and $(\overline{y},\tilde{x})$ has the same objective value in \eqref{eq:reformulatedP} as $(\tilde{y},\tilde{x})$ has in \eqref{eq:deterministicP}.  We have already shown that any feasible solution for \eqref{eq:reformulatedP} has a corresponding feasible solution for \eqref{eq:deterministicP} with the same objective value, so $(\overline{y},\tilde{x})$ is optimal for \eqref{eq:reformulatedP}, as otherwise there must also be a better solution than $(\tilde{y},\tilde{x})$ for \eqref{eq:deterministicP}.  Therefore, as $(\tilde{y},\tilde{x})$ is optimal for \eqref{eq:deterministicP} and $(\overline{y},\tilde{x})$ is optimal for \eqref{eq:reformulatedP} with the same objective value, the problems have the same optimal objective.
\end{proof}

\begin{lemma}
    \label{lemma:pwlFuncProperties}
    The functions $\overline{f}_i$defined in \eqref{eq:underapproximatorConstruction} have the following properties:
    \begin{enumerate}
        \item $\overline{f}_i(y) \leq f_i(y)$ for $y \in [y_i^{(1)}, y_i^{(k_i)}]$;
        \item $\overline{f}_i$ is concave and continuous; and
        \item $\overline{f}_i(y_i^{(j)}) = f_i(y_i^{(j)})$ for all $j \in \{1,\dots,k_i\}$.
    \end{enumerate}
\end{lemma}

\begin{proof}
1.  Select $j$ so that $y \in [y_i^{(j)}, y_i^{(j+1)})$.  Then, taking $\lambda = \frac{y - y_i^{(j)}}{y_i^{(j+1)} - y_i^{(j)}} \in [0,1)$, 
\begin{equation*}
\begin{aligned}
        \overline{f}_i(y) & = \frac{y - y_i^{(j)}}{y_i^{(j+1)} - y_i^{(j)}} f_i(y_i^{(j+1)}) + \frac{y_i^{(j+1)} - y}{y_i^{(j+1)} - y_i^{(j)}} f_i(y_i^{(j)}) = \lambda f_i(y_i^{(j+1)}) + (1 - \lambda) f_i(y_i^{(j)}) \\
        & \leq f_i(\lambda y_i^{(j+1)} + (1 - \lambda) y_i^{(j)}) = f_i(y),
\end{aligned}
\end{equation*}
by the concavity of $f_i$.

2.  Denote $\tilde{f}_i(y) = \underset{j \in \{1,\dots,k_i - 1\}}{\min} \left \{ \frac{f_i(y_i^{(j+1)}) - f_i(y_i^{(j)})}{y_i^{(j+1)} - y_i^{(j)}} (y - y_i^{(j)}) + f_i(y_i^{(j)}) \right \}$.  Consider $y \in [y_i^{(1)},y_i^{(k_i)}]$ and any index $j$.  If $y \geq y_i^{(j+1)}$, let $\lambda = \frac{y_i^{(j+1)} - y_i^{(j)}}{y - y_i^{(j)}} \in (0,1]$ so that $y_i^{(j+1)} = \lambda y + (1 - \lambda) y_i^{(j)}$.  Then, by concavity of $f_i$,
\begin{equation*}
    f_i(y_i^{(j+1)}) \geq \lambda f_i(y) + (1 - \lambda) f_i(y_i^{(j)}),
\end{equation*}
and
\begin{equation*}
    \frac{f_i(y_i^{(j+1)}) - f_i(y_i^{(j)})}{y_i^{(j+1)} - y_i^{(j)}} (y - y_i^{(j)}) + f_i(y_i^{(j)}) \geq f_i(y) \geq \overline{f}_i(y).
\end{equation*}

Otherwise, if $y < y_i^{(j)}$, let $\lambda = \frac{y_i^{(j)} - y}{y_i^{(j+1)} - y} \in (0,1]$ so that $y_i^{(j)} = \lambda y_i^{(j+1)} + (1 - \lambda) y$.  By concavity of $f_i$,
\begin{equation*}
    f_i(y_i^{(j)}) \geq \lambda f_i(y_i^{(j+1)}) + (1 - \lambda) f_i(y),
\end{equation*}
and
\begin{equation*}
    \frac{f_i(y_i^{(j+1)}) - f_i(y_i^{(j)})}{y_i^{(j+1)} - y_i^{(j)}} (y - y_i^{(j)}) + f_i(y_i^{(j)}) \geq f_i(y) \geq \overline{f}_i(y).
\end{equation*}
Finally, if $y \in [y_i^{(j)},y_i^{(j+1)})$, 
\begin{equation*}
    \frac{f_i(y_i^{(j+1)}) - f_i(y_i^{(j)})}{y_i^{(j+1)} - y_i^{(j)}} (y - y_i^{(j)}) + f_i(y_i^{(j)}) = \overline{f}_i(y).
\end{equation*}
Thus, we have that $\overline{f}_i(y) \leq \tilde{f}_i(y)$.  Further, $\overline{f}_i(y) \in \left \{ \frac{f_i(y_i^{(j+1)}) - f_i(y_i^{(j)})}{y_i^{(j+1)} - y_i^{(j)}} (y - y_i^{(j)}) + f_i(y_i^{(j)}) \right \}_{j = 1}^{k_i - 1}$, so $\overline{f}_i(y) \geq \tilde{f}_i(y)$, and $\overline{f}_i(y) = \tilde{f}_i(y)$.  Therefore, as the function $\overline{f}_i$ is a minimum of linear functions, it is continuous and concave.

3. The equality follows from direct evaluation.
\end{proof}

\begin{corollary}
\label{corr:succPWL}
    Let $\overline{f}_i^{1}$ be the piecewise linear function associated with breakpoints $\{y_i^{(j)}\}_{j=1}^{k_i}$ and $\overline{f}_i^{2}$ the function with breakpoints $\{y_i^{(j)}\}_{j=1}^{k_i} \setminus \{y_i^{(r)}\}$ for some $r \not \in \{1,k_i\}$.  Then, $\overline{f}_i^{1}(y) \geq \overline{f}_i^{2}(y)$ for $y \in [y_i^{(1)}, y_i^{(k_i)}]$.
\end{corollary}

\begin{proof}
By Lemma~\ref{lemma:pwlFuncProperties}, $\overline{f}_i^{1}$ is concave.  Let $j_1$ and $j_2$ be consecutive breakpoints for $\overline{f}_i^{2}$, so $$j_2 = \begin{cases}
j_1 + 1 & j_1 \neq r - 1\\
j_1 + 2 & \text{otherwise}.
\end{cases}$$  Then, by Lemma~\ref{lemma:pwlFuncProperties},
$$\frac{f_i(y_i^{(j_2)}) - f_i(y_i^{(j_1)})}{y_i^{(j_2)} - y_i^{(j_1)}} (y_i - y_i^{(j_1)}) + f_i(y_i^{(j_1)}) = \frac{\overline{f}^{1}_i(y_i^{(j_2)}) - \overline{f}^{1}_i(y_i^{(j_1)})}{y_i^{(j_2)} - y_i^{(j_1)}} (y_i - y_i^{(j_1)}) + \overline{f}^{1}_i(y_i^{(j_1)})$$
for any $y_i \in [y^{(1)}_i,y^{(k_i)}_i]$.  Thus, $\overline{f}_i^{2}$ can be equivalently constructed in the form \eqref{eq:underapproximatorConstruction} by using the concave function $\overline{f}_i^{1}$ in place of $f_i$.  It then holds that $\overline{f}_i^{2}$ is a piecewise under-approximator of $\overline{f}_i^{1}$ and, by Lemma~\ref{lemma:pwlFuncProperties}, $\overline{f}_i^{1}(y) \geq \overline{f}_i^{2}(y)$ for $y \in [y_i^{(1)}, y_i^{(k_i)}]$.
\end{proof}

\setcounter{repeatlemma}{0}
\begin{repeatlemma}
    Let $(y^*,x^*) \in \mathcal{X}$ be an optimal solution to \eqref{eq:SCPSP}.  Then, there is an optimal solution that is an extreme point of $\mathcal{X}(y^*)$.
\end{repeatlemma}

\begin{proof}
Denote the optimal cost for \eqref{eq:SCPSP} by $z^{\text{SP}} = \sum_{i=1}^n \overline{f}_i(y^*_i) + \sum_{\omega \in \Omega_1} p_\omega c_\omega^T x^*_\omega$.  Consider the restricted problem
\begin{equation*}
    \tag{RSP}
    \label{eq:RSPLemma}
    \begin{aligned}
        z^{\text{RP}} = \min_{(y,x) \in \mathcal{X}(y^*)} \quad & \sum_{i = 1}^n \overline{f}_i(y_i) + \sum_{\omega \in \Omega_1} p_\omega c_\omega^T x_\omega.\\
    \end{aligned}
\end{equation*}
As $\mathcal{X}(y^*) \subseteq \mathcal{X}$, \eqref{eq:RSPLemma} is a restriction of \eqref{eq:SCPSP}, and $z^{\text{SP}} \leq z^{\text{RP}}$.  However, $(y^*,x^*)$ is feasible for \eqref{eq:RSPLemma}, so $z^{\text{RP}} \leq z^{\text{SP}}$.  Thus, the optimal values $z^{\text{RP}} = z^{\text{SP}}$, and an optimal solution for \eqref{eq:RSPLemma} is also optimal for \eqref{eq:SCPSP}.

Further, 
\begin{equation*}
\begin{aligned}
    \mathcal{X}(y^*) = \{(y,\{x_\omega\}_{\omega \in \Omega_1}) \geq 0\ :\ y_i = y^*_i\ \forall i \in \lBrack n_{\text{I}} \rBrack,\ A y = b,\ B_\omega x_\omega + D y = d_\omega\ \forall \omega \in \Omega_1\}
\end{aligned}
\end{equation*}
is a polytope under the assumption that $\mathcal{X}$ is bounded.  Lemma~\ref{lemma:pwlFuncProperties} gives that the functions $\overline{f}_i$ are concave, so the objective function of \eqref{eq:RSPLemma} is concave and \eqref{eq:RSPLemma} is a concave minimization over a polytope. Thus, \eqref{eq:RSPLemma} has an optimal solution at an extreme point of $\mathcal{X}(y^*)$ by Lemma~\ref{lemma:horstExtPt}.  This solution is also optimal for \eqref{eq:SCPSP}, and therefore, there is an extreme point of $\mathcal{X}(y^*)$ that is optimal for \eqref{eq:SCPSP}.
\end{proof}

\setcounter{repeattheorem}{2}
\begin{repeattheorem}
    Algorithm~\ref{alg:PWL} terminates finitely with a global optimum of \eqref{eq:SCP}.
\end{repeattheorem}

\begin{proof}
    Consider the set of feasible realizations of the integer components of $y$, 
    $$\mathcal{Y} = \{z \in \mathbb{Z}_+^{n_{\text{I}}}\ :\ \exists (y,x) \in \mathcal{X},\ y_i = z_i\ \forall i \in \lBrack n_{\text{I}} \rBrack\} \subseteq \{z \in \mathbb{Z}_+^{n_{\text{I}}}\ :\ \underline{y}_i \leq z_i \leq \overline{y}_i \  \forall i \in \lBrack n_{\text{I}} \rBrack \}.$$
    Clearly, $\mathcal{Y}$ is finite.  
    We construct the set of points that are extreme points of $\mathcal{X}(z)$ for some $z \in \mathcal{Y}$,
    $$\mathcal{E} = \bigcup_{z \in \mathcal{Y}} \{(y,x)\ :\ (y,x) \text{ extreme point of } \mathcal{X}(z)\}.$$
    This set is also finite, as there are finitely many extreme points of any polytope $\mathcal{X}(z)$ and finitely many such $z \in \mathcal{Y}$.
    We next construct the set of unique capacity solutions with a corresponding extreme point:
    $$\overline{\mathcal{E}} = \bigcup_{(y,x) \in \mathcal{E}} \{y\}.$$
    Then, $|\overline{\mathcal{E}}| \leq |\mathcal{E}|$, so $\overline{\mathcal{E}}$ is finite.

    For any $y$ feasible for \eqref{eq:SCPSP} or \eqref{eq:SCP} and any set of breakpoints $\{y_i^{(j)}\}_{j=1}^{k_i}$ generated by Algorithm~\ref{alg:PWL}, 
    $$y_i^{(1)} = \underline{y}_i \leq y_i \leq \overline{y}_i = y_i^{(k_i)}$$ by the initialization in step 1 of the algorithm, so $y_i \in [y_i^{(1)},y_i^{(k_i)}]$.  As a result, the first property of Lemma~\ref{lemma:pwlFuncProperties} will apply for any feasible solution $y$ to these problems.  Further, $y_i^{(1)}$ and $y_i^{(k_i)}$ maintain the same values across iterations by step 5 of the algorithm.

    Suppose that Algorithm~\ref{alg:PWL} produces a sequence of iterates $(\tilde{y}^{s},\{\tilde{x}^{s}_\omega\}_{\omega \in \Omega_1})$ of length at least $|\overline{\mathcal{E}}| + 1$.  We note that \eqref{eq:SCPSP} has a finite optimal solution by the boundedness assumption on $\mathcal{X}$ (Lemma~\ref{lemma:horstExtPt}).  Under Assumption~\ref{assump:extPoint}, at every iteration $s$, the iterate $(\tilde{y}^{s},\{\tilde{x}^{s}_\omega\}_{\omega \in \Omega_1})$ is an extreme point of $\mathcal{X}(\tilde{y}^{s})$, so $\tilde{y}^{s} \in \overline{\mathcal{E}}$ for all $s$.

    As the number of iterates is strictly larger than $|\overline{\mathcal{E}}|$, some capacity solution must be revisited, so $\tilde{y}^{s_1} = \tilde{y}^{s_2}$ for some $s_1 < s_2 \leq |\overline{\mathcal{E}}| + 1$.
    Denote the set of breakpoints at the start of iteration $s_2$ by $\{y_i^{(j)}\}_{j=1}^{k_i}$.
    Note that the breakpoints which generate $\overline{f}_i^{s_1}$ are a subset of those which generate $\overline{f}_i^{s_2}$ and include $\{y_i^{(1)},y_i^{(k_i)}\}$, so $\overline{f}_i^{s_2}(\tilde{y}_i^{s_1}) \geq \overline{f}_i^{s_1}(\tilde{y}_i^{s_1})$ as a consequence of Corollary~\ref{corr:succPWL}.  From step 5 of the algorithm, either $\tilde{y}_i^{s_1} \in \{y_i^{(j)}\}_{j=1}^{k_i}$ or $\overline{f}^{s_1}_i(\tilde{y}_i^{s_1}) \geq f_i(\tilde{y}_i^{s_1})$ for each $i$.  In the former case, $\overline{f}_i^{s_2}(\tilde{y}_i^{s_2}) = f_i(\tilde{y}_i^{s_2})$ by Lemma~\ref{lemma:pwlFuncProperties}.  In the latter case, 
    $$\overline{f}_i^{s_2}(\tilde{y}_i^{s_2}) = \overline{f}_i^{s_2}(\tilde{y}_i^{s_1}) \geq \overline{f}_i^{s_1}(\tilde{y}_i^{s_1}) \geq f_i(\tilde{y}_i^{s_1}) = f_i(\tilde{y}_i^{s_2}).$$
    Thus, the condition in step 4 of the algorithm is met at iteration $s_2$, and Algorithm~\ref{alg:PWL} terminates finitely within $|\overline{\mathcal{E}}| + 1$ iterations.

    Let the algorithm terminate at iteration $s$ with iterate $(\tilde{y}^s,\{\tilde{x}^s_\omega\}_{\omega \in \Omega_1})$.  Then,
    \begin{equation*}
        \begin{aligned}
            \sum_{i = 1}^n \overline{f}_i^s (\tilde{y}^s_i) + \sum_{\omega \in \Omega_1} p_\omega c_\omega^T \tilde{x}_\omega^s & = \underset{(y,\{x_\omega\}_{\omega \in \Omega_1}) \in \mathcal{X}}{\min} \sum_{i = 1}^n \overline{f}_i^s (y_i) + \sum_{\omega \in \Omega_1} p_\omega c_\omega^T x_\omega \\
            & \leq \underset{(y,\{x_\omega\}_{\omega \in \Omega_1}) \in \mathcal{X}}{\min} \sum_{i = 1}^n f_i(y_i) + \sum_{\omega \in \Omega_1} p_\omega c_\omega^T x_\omega \\
            & \leq \sum_{i = 1}^n f_i(\tilde{y}^s_i) + \sum_{\omega \in \Omega_1} p_\omega c_\omega^T \tilde{x}^s_\omega \\
            & \leq \sum_{i = 1}^n \overline{f}^s_i(\tilde{y}^s_i) + \sum_{\omega \in \Omega_1} p_\omega c_\omega^T \tilde{x}^s_\omega,
        \end{aligned}
    \end{equation*}
    where the first inequality follows from Lemma~\ref{lemma:pwlFuncProperties} and the third from the termination criterion in step 4 of the algorithm.  Therefore,
    $$\underset{(y,\{x_\omega\}_{\omega \in \Omega_1}) \in \mathcal{X}}{\min} \sum_{i = 1}^n f_i(y_i) + \sum_{\omega \in \Omega_1} p_\omega c_\omega^T x_\omega = \sum_{i = 1}^n f_i(\tilde{y}^s_i) + \sum_{\omega \in \Omega_1} p_\omega c_\omega^T \tilde{x}^s_\omega,$$
    and $(\tilde{y}^s,\{\tilde{x}^s_\omega\}_{\omega \in \Omega_1})$ is a global optimum for \eqref{eq:SCP}.
\end{proof}

\subsection{Proofs of Modeling Extension Reformulations}
\label{sec:ecExtensions}

In this section, we rigorously justify the reformulations of the model extensions provided in Section~\ref{sec:extension}.  Section~\ref{subsec:ecPWCExtension} considers the reformulation for piecewise concave cost functions, and Section~\ref{subsec:ecFCExtension} considers the reformulation for fixed costs of facility expansion.  We prove that these reformulations are equivalent to the original problems by demonstrating that the reformulation and original form have the same optimal objective value, using strategies seen in Theorem~\ref{thm:reformulationEquivalent} for the reformulation~\eqref{eq:reformulatedP}.  For these extensions, we focus only on the recycling capacity solution $y^{\mathrm{REC}}$ and ignore the reformulation of the cathode production capacity $y^{\mathrm{CP}}$; equivalent reformulations for these variables follow from the same arguments.

\subsubsection{Piecewise Concave Cost Functions}
\label{subsec:ecPWCExtension}
We consider the reformulation described in Section~\ref{subsec:extensionPWC} for the case where the facility cost functions are piecewise concave.  
Let $(\overline{y},\overline{x})$ be a feasible solution for this reformulation.  We define $\tilde{y}$ by 
\begin{equation*}
    \begin{aligned}
        \tilde{y}^{\text{REC}}_{z,l,j,n} & = \begin{cases}
            u^{\text{REC}}_q & \sum_{q'=1}^{q-1} \overline{y}^{\mathrm{REC}}_{z,l,j,q'} < n \leq \sum_{q'=1}^{q} \overline{y}^{\mathrm{REC}}_{z,l,j,q'}\\
            \sum_{q'=1}^Q \overline{y}^{\text{REC}}_{z,l,j,q',+} - u^{\mathrm{REC}}_{q-1} (1-\overline{y}^{\mathrm{REC,I}}_{z,l,j,q}) & n = 1 + \sum_{q'=1}^{Q} \overline{y}^{\mathrm{REC}}_{z,l,j,q'}\\
            0 & \text{otherwise},
        \end{cases}
    \end{aligned}
\end{equation*}
so that the number of facilities with capacity $u^{\mathrm{REC}}_q$ is the same in both solutions, and some facility matches the capacity from the continuous facility.  By~\eqref{eq:PWCreformulationIndicDef}, $\overline{y}^{\mathrm{REC,I}}_{z,l,j,q} = 0$ implies that $\overline{y}^{\text{REC}}_{z,l,j,q,+} = u^{\mathrm{REC}}_{q-1}$, and otherwise $\overline{y}^{\mathrm{REC}}_{z,l,j,q,+} \in [u^{\mathrm{REC}}_{q-1},u^{\mathrm{REC}}_q]$.  Thus, we have
\begin{equation*}
    \sum_{q'=1}^Q \overline{y}^{\text{REC}}_{z,l,j,q',+} - u^{\mathrm{REC}}_{q-1} (1-\overline{y}^{\mathrm{REC,I}}_{z,l,j,q}) \in [u^{\mathrm{REC}}_0,u^{\mathrm{REC}}_Q],
\end{equation*}
and $\tilde{y}^{\mathrm{REC}}_{z,l,j,n} \in [u^{\mathrm{REC}}_0,u^{\mathrm{REC}}_Q]$ for all $n$.
It also holds that $$\sum_{n \in \mathcal{N}^{\mathrm{REC}}_l} \tilde{y}^{\mathrm{REC}}_{z,l,j,n} = \sum_{q=1}^Q u^{\mathrm{REC}}_q \overline{y}^{\mathrm{REC}}_{z,l,j,q} + \overline{y}^{\mathrm{REC}}_{z,l,j,q,+} - u^{\mathrm{REC}}_{q-1} (1-\overline{y}^{\mathrm{REC,I}}_{z,l,j,q}),$$ and the recycling capacities of the solutions $\overline{y}$ and $\tilde{y}$ are equal.
As these solutions share the same total capacities and the throughput and monotonicity constraints (\eqref{eq:constrCapREC}~and~\eqref{eq:constrCapRECIncrease}) are satisfied by $(\overline{y},\overline{x}$), the solution $(\tilde{y},\overline{x})$ is feasible for constraints~\eqref{eq:constrCap} and also for~\eqref{eq:deterministicP} (as in the proof of Theorem~\ref{thm:reformulationEquivalent}).  Again by~\eqref{eq:PWCreformulationIndicDef}, the relevant objective terms are also equal:
\begin{equation*}
\begin{aligned}
        & \sum_{q=1}^Q \left ( \overline{y}^{\mathrm{REC}}_{z,l,j,q} f^{\mathrm{REC}}_{z,j,q}(u^{\mathrm{REC}}_q) + f^{\mathrm{REC}}_{z,j,q} (\overline{y}^{\mathrm{REC}}_{z,l,j,q,+}) - (1 - \overline{y}^{\mathrm{REC,I}}_{z,l,j,q}) f^{\mathrm{REC}}_{z,j,q}(u^{\mathrm{REC}}_{q-1})\right )\\
        =\ & \sum_{n \in \mathcal{N}^{\mathrm{REC}}_{l}} \sum_{q=1}^Q \mathbb{I}[y^{\mathrm{REC}}_{z,l,j,n} \in (u^{\mathrm{REC}}_{q-1},u^{\mathrm{REC}}_q]] f^{\mathrm{REC}}_{z,j,q}(y^{\mathrm{REC}}_{z,l,j,n})\\
        =\ & \sum_{n \in \mathcal{N}^{\mathrm{REC}}_{l}} f^{\mathrm{REC}}_{z,j}(y^{\mathrm{REC}}_{z,l,j,n}),
\end{aligned}
\end{equation*}
and the planning cost of $\overline{y}^{\mathrm{REC}}$ in the reformulation is the same as that of $\tilde{y}^{\mathrm{REC}}$ in~\eqref{eq:deterministicP}.  Thus, $(\tilde{y},\overline{x})$ has the same objective value as $(\overline{y},\overline{x})$.

In the opposite direction, let $(\hat{y},\hat{x})$ be an optimal solution for~\eqref{eq:deterministicP} with piecewise concave cost functions. As only the cost function associated with variables $y$ has changed, \eqref{eq:deterministicP} remains feasible and bounded (see Proposition~\ref{prop:finiteP}).  Consider the optimization problem
\begin{equation}
    \label{eq:restrictedPWC}
    \begin{aligned}
        \min_{x,y} \quad & \sum_{t \in \mathcal{T}} (1 - \gamma)^{t-1} \left ( C^{\text{PL}}_t(y) + \sum_{\omega \in \Omega_{\sigma_t}} p_\omega C^{\text{OP}}_{\omega,t}(x) \right ) \\
        \text{s.t.} \quad & \eqref{eq:constrProd},\ \eqref{eq:constrMatFlow},\ \eqref{eq:constrCap}\\
        & u^{\mathrm{REC},-}_{z,l,j,n} \leq y^{\mathrm{REC}}_{z,l,j,n} \leq u^{\mathrm{REC},+}_{z,l,j,n} \quad \forall j \in \mathcal{J},\ l \in \mathcal{L},\ z \in \mathcal{Z},\ n \in \mathcal{N}^{\text{REC}}_l\\
        & x \geq 0,\ y \geq 0,
    \end{aligned}
\end{equation}
where $C^{\mathrm{PL}}_t$ uses the piecewise concave cost functions, $u^{\mathrm{REC},+}_{z,l,j,n} := \sum_{q=1}^Q u^{\mathrm{REC}}_{q} \mathbb{I}[\hat{y}^{\mathrm{REC}}_{z,l,j,n} \in (u^{\mathrm{REC}}_{q-1},u^{\mathrm{REC}}_q]]$, and $u^{\mathrm{REC},-}_{z,l,j,n} := \sum_{q=1}^Q u^{\mathrm{REC}}_{q-1} \mathbb{I}[\hat{y}^{\mathrm{REC}}_{z,l,j,n} \in (u^{\mathrm{REC}}_{q-1},u^{\mathrm{REC}}_q]]$.  
As the interval containing the capacity variables $y^{\mathrm{REC}}_{z,l,j,n}$ is fixed and the function values at the interval endpoints match (i.e., $f^{\mathrm{REC}}_{z,j,q}(u^{\mathrm{REC}}_q) = f^{\mathrm{REC}}_{z,j,q+1}(u^{\mathrm{REC}}_q)$), the objective function is concave over a polyhedral feasible region.  By the same logic as in Proposition~\ref{prop:finiteP}, the problem~\eqref{eq:restrictedPWC} has an optimal solution at an extreme point of its feasible region.  Let $(\tilde{y},\tilde{x})$ be such a solution.  Problem~\eqref{eq:restrictedPWC} is a restriction of~\eqref{eq:deterministicP} and the solution $(\hat{y},\hat{x})$ remains feasible for~\eqref{eq:restrictedPWC}, so $(\hat{y},\hat{x})$ is an optimal solution to~\eqref{eq:restrictedPWC}.  Then, $(\tilde{y},\tilde{x})$ must have the same objective value in~\eqref{eq:restrictedPWC} as $(\hat{y},\hat{x})$, and thus $(\tilde{y},\tilde{x})$ and $(\hat{y},\hat{x})$ have the same objective value in~\eqref{eq:deterministicP}.  Therefore, $(\tilde{y},\tilde{x})$ is optimal for~\eqref{eq:deterministicP}.

We next show that the extreme point $(\tilde{y},\tilde{x})$ satisfies the following property:
\begin{equation}
    \label{eq:PWCextPt}
    \sum_{q=0}^Q \mathcal{C}(\{\tilde{y}^{\mathrm{REC}}_{z,l,j,n}\}_{n \in \mathcal{N}^{\mathrm{REC}}_l}; u^{\mathrm{REC}}_q) \geq |\mathcal{N}^{\mathrm{REC}}_l| - 1.
\end{equation}
Suppose otherwise.  Then, there must be two indices $n_1, n_2$ with $\tilde{y}^{\mathrm{REC}}_{z,l,j,n} \not \in \{u^{\mathrm{REC}}_q\}_{q=0}^Q$ for $n \in \{n_1,n_2\}$.  Let $q^-_1$ be the largest index such that $\tilde{y}^{\mathrm{REC}}_{z,l,j,n_1} > u^{\mathrm{REC}}_{q^-_1}$, and $q^+_1$ be the smallest index such that $\tilde{y}^{\mathrm{REC}}_{z,l,j,n_1} < u^{\mathrm{REC}}_{q^+_1}$.  Let $q^-_2$ and $q^+_2$ be constructed similarly for index $n_2$.  Define
\begin{equation*}
    \delta = \min \{\tilde{y}^{\mathrm{REC}}_{z,l,j,n_1} - u^{\mathrm{REC}}_{q^-_1},\ \tilde{y}^{\mathrm{REC}}_{z,l,j,n_2} - u^{\mathrm{REC}}_{q^-_2},\ u^{\mathrm{REC}}_{q^+_1} -\tilde{y}^{\mathrm{REC}}_{z,l,j,n_1},\ u^{\mathrm{REC}}_{q^+_2} -\tilde{y}^{\mathrm{REC}}_{z,l,j,n_2}\}.
\end{equation*}
By construction, $\delta > 0$.  We define perturbed solutions
\begin{equation*}
    y^{\mathrm{REC}}_{z,l,j,n}(1) = \begin{cases}
        \tilde{y}^{\mathrm{REC}}_{z,l,j,n} \quad & n \not \in \{n_1,n_2\}\\
        \tilde{y}^{\mathrm{REC}}_{z,l,j,n} + \delta \quad & n = n_1\\
        \tilde{y}^{\mathrm{REC}}_{z,l,j,n} - \delta \quad & n = n_2,
    \end{cases}
\end{equation*}
and
\begin{equation*}
    y^{\mathrm{REC}}_{z,l,j,n}(2) = \begin{cases}
        \tilde{y}^{\mathrm{REC}}_{z,l,j,n} \quad & n \not \in \{n_1,n_2\}\\
        \tilde{y}^{\mathrm{REC}}_{z,l,j,n} - \delta \quad & n = n_1\\
        \tilde{y}^{\mathrm{REC}}_{z,l,j,n} + \delta \quad & n = n_2.
    \end{cases}
\end{equation*}
Observe that $\sum_{n \in \mathcal{N}^{\mathrm{REC}}_l} y^{\mathrm{REC}}_{z,l,j,n}(1) = \sum_{n \in \mathcal{N}^{\mathrm{REC}}} \tilde{y}^{\mathrm{REC}}_{z,l,j,n}$, so the capacity remains the same in both solutions.  Also, by definition of $\delta$, $q^-_1$, and $q^+_1$, we have that $u^{\mathrm{REC},-}_{z,l,j,n} \leq y^{\mathrm{REC}}_{z,l,j,n}(1) \leq u^{\mathrm{REC},+}_{z,l,j,n}$ for all $n$.  Therefore, $(y(1),\tilde{x})$ is feasible for~\eqref{eq:restrictedPWC}.  Similarly, $(y(2),\tilde{x})$ is also feasible for~\eqref{eq:restrictedPWC}.  Finally, we note that $(\tilde{y},\tilde{x}) = \frac{1}{2}(y(1),\tilde{x}) + \frac{1}{2} (y(2),\tilde{x})$, so $(\tilde{y},\tilde{x})$ is a convex combination of feasible points for~\eqref{eq:restrictedPWC}, and is not an extreme point of the feasible set.  By this contradiction, we find that the condition~\eqref{eq:PWCextPt} is satisfied by $(\tilde{y},\tilde{x})$.

We now construct an equivalent feasible solution for the reformulated constraints \eqref{eq:PWCreformulation}.  If the lower bound of the property~\eqref{eq:PWCextPt} holds with equality, we can select the unique index $\overline{n}_{z,l,j}$ such that $\tilde{y}^{\mathrm{REC}}_{z,l,j,\overline{n}_{z,l,j}} \not \in \{u^{\mathrm{REC}}_q\}_{q=0}^Q$.  Otherwise, we select arbitrary $\overline{n}_{z,l,j}$.  By~\eqref{eq:PWCextPt}, any $n \neq \overline{n}_{z,l,j}$ has $\tilde{y}^{\mathrm{REC}}_{z,l,j,n} \in \{u^{\mathrm{REC}}_q\}_{q=0}^Q$.  Now, we define
\begin{equation*}
    \begin{aligned}
        \overline{y}^{\mathrm{REC}}_{z,l,j,q} & = \mathcal{C}(\{\tilde{y}^{\mathrm{REC}}_{z,l,j,n}\}_{n \in \mathcal{N}^{\mathrm{REC}}_l \setminus \{\overline{n}_{z,l,j}\}}; u^{\mathrm{REC}}_q);\\
        \overline{y}^{\mathrm{REC,I}}_{z,l,j,q} & = \mathbb{I}[\tilde{y}^{\mathrm{REC}}_{z,l,j,\overline{n}_{z,l,j}} \in (u^{\mathrm{REC}}_{q-1},u^{\mathrm{REC}}_q]];\ \text{and}\\
        \overline{y}^{\mathrm{REC}}_{z,l,j,q,+} & = \tilde{y}^{\mathrm{REC}}_{z,l,j,\overline{n}_{z,l,j}} \mathbb{I}[\tilde{y}^{\mathrm{REC}}_{z,l,j,\overline{n}_{z,l,j}} \in (u^{\mathrm{REC}}_{q-1},u^{\mathrm{REC}}_q]] + u^{\mathrm{REC}}_{q-1} \left ( 1 - \mathbb{I}[\tilde{y}^{\mathrm{REC}}_{z,l,j,\overline{n}_{z,l,j}} \in (u^{\mathrm{REC}}_{q-1},u^{\mathrm{REC}}_q]] \right ).
    \end{aligned}
\end{equation*}
By construction, $(\overline{y}^{\mathrm{REC}}_{z,l,j,q},\overline{y}^{\mathrm{REC,I}}_{z,l,j,q},\overline{y}^{\mathrm{REC}}_{z,l,j,q,+})$ satisfies the domain constraints~\eqref{eq:PWCreformulationDomain}.  As $u^{\mathrm{REC}}_q$ are unique and the intervals $(u^{\mathrm{REC}}_{q-1},u^{\mathrm{REC}}_q]$ are disjoint, the solution also satisfies constraints~\eqref{eq:PWCreformulationCount}~and~\eqref{eq:PWCreformulationIndicCount}.  Finally, as $\overline{y}^{\mathrm{REC}}_{z,l,j,q,+} = u^{\mathrm{REC}}_{q-1}$ if $\overline{y}^{\mathrm{REC,I}}_{z,l,j,q} = 0$, and otherwise takes the value $\tilde{y}^{\mathrm{REC}}_{z,l,j,\overline{n}_{z,l,j}} \in (u^{\mathrm{REC}}_{q-1},u^{\mathrm{REC}}_q]$, constraints~\eqref{eq:PWCreformulationIndicDef} are satisfied.  Therefore, the constructed solution is feasible for the reformulation~\eqref{eq:PWCreformulation}.  The total capacity of both solutions are equal:
\begin{equation*}
    \begin{aligned}
        \sum_{n \in \mathcal{N}^{\mathrm{REC}}_l} \tilde{y}^{\mathrm{REC}}_{z,l,j,n} & = \tilde{y}^{\mathrm{REC}}_{z,l,j,\overline{n}_{z,l,j}} + \sum_{q=1}^Q u^{\mathrm{REC}}_q \mathcal{C}(\{\tilde{y}^{\mathrm{REC}}_{z,l,j,n}\}_{n \in \mathcal{N}^{\mathrm{REC}}_l \setminus \{\overline{n}_{z,l,j}\}}; u^{\mathrm{REC}}_q) \\
        & = \sum_{q=1}^Q u^{\mathrm{REC}}_q \overline{y}^{\mathrm{REC}}_{z,l,j,q} + \overline{y}^{\mathrm{REC}}_{z,l,j,q,+} - u^{\mathrm{REC}}_{q-1}(1 - \overline{y}^{\mathrm{REC,I}}_{z,l,j,q}).    
    \end{aligned}
\end{equation*}
Then, the feasibility of $(\tilde{y},\tilde{x})$ for~\eqref{eq:constrProd}-\eqref{eq:constrCap} implies that the solution $(\overline{y},\tilde{x})$ remains feasible under the reformulation.  Considering the objective,
\begin{equation*}
    \begin{aligned}
        &\sum_{n \in \mathcal{N}^{\mathrm{REC}}} \sum_{q = 1}^Q \mathbb{I}[\tilde{y}^{\mathrm{REC}}_{z,l,j,n} \in (u^{\mathrm{REC}}_{q-1}, u^{\mathrm{REC}}_{q}]] f^{\mathrm{REC}}_{z,j,q}(\tilde{y}^{\mathrm{REC}}_{z,l,j,n})\\
        =\ & \sum_{q=1}^Q \overline{y}^{\mathrm{REC}}_{z,l,j,q}f^{\mathrm{REC}}_{z,j,q}(u^{\mathrm{REC}}_q) + f^{\mathrm{REC}}_{z,j,q}(\overline{y}^{\mathrm{REC}}_{z,l,j,q,+}) - (1-\overline{y}^{\mathrm{REC,I}}_{z,l,j,q}) f^{\mathrm{REC}}_{z,j,q}(u^{\mathrm{REC}}_{q-1}).
    \end{aligned}
\end{equation*}
Then, $(\overline{y},\tilde{x})$ has the same objective value in the reformulation as $(\tilde{y},\tilde{x})$ has in~\eqref{eq:deterministicP}.

To conclude, as any feasible solution for the reformulation has a corresponding feasible solution for~\eqref{eq:deterministicP} with the same objective value, $(\overline{y},\tilde{x})$ is optimal for the reformulation by the optimality of $(\tilde{y},\tilde{x})$ for~\eqref{eq:deterministicP}.  Therefore, both solutions are optimal for their respective models with the same objective values, so~\eqref{eq:deterministicP} and the reformulation have the same optimal objective.

\subsubsection{Fixed Cost of Expansion}
\label{subsec:ecFCExtension}
We consider the reformulation described in Section~\ref{subsec:extensionFixedC}, where expanding a facility incurs a fixed cost.  Let $(\overline{y},\overline{x})$ be an optimal solution for this reformulation.  We define $\tilde{y}$ by 
\begin{equation*}
    \tilde{y}^{\mathrm{REC}}_{z,l,j,n} = \begin{cases}
        u^{\mathrm{REC}} \quad & n \leq \overline{y}^{\mathrm{REC}}_{z,l,j}\\
        \overline{y}^{\mathrm{REC}}_{z,l,l',j,+} \quad & n = |\mathcal{N}^{\mathrm{REC}}_l| + l' - L,\ l' \in \mathcal{L}\\
        0 \quad & \text{otherwise}.
    \end{cases}
\end{equation*}
Then, $\sum_{n \in \mathcal{N}^{\mathrm{REC}}_l} \tilde{y}^{\mathrm{REC}}_{z,l,j,n} = u^{\mathrm{REC}} \overline{y}^{\mathrm{REC}}_{z,l,j} + \sum_{l' \in \mathcal{L}} \overline{y}^{\mathrm{REC}}_{z,l,l',j,+}$ and the recycling capacities of the solutions $\overline{y}$ and $\tilde{y}$ are equal.  As these solutions share the same total capacities and the throughput and monotonicity constraints (\eqref{eq:constrCapREC}~and~\eqref{eq:constrCapRECIncrease}) are satisfied by $(\overline{y},\overline{x})$, the solution $(\tilde{y},\overline{x})$ is feasible for constraints~\eqref{eq:constrCap} and also for~\eqref{eq:deterministicP} (as in the proof of Theorem~\ref{thm:reformulationEquivalent}).  

By the optimality of the solution, the indicator $y^{\mathrm{REC,EXP}}_{z,l,l',j} = 1$ if and only if $y^{\mathrm{REC}}_{z,l,l',j,+} > y^{\mathrm{REC}}_{z,l-1,l',j,+}$, from constraint~\eqref{eq:FCreformulationIndUB} and the positivity of $f^{\mathrm{EXP}}_{z,l,j}$ in the objective function.  Similarly, if $\overline{y}^{\mathrm{REC}}_{z,l,j} > \overline{y}^{\mathrm{REC}}_{z,l-1,j}$, there are $\overline{y}^{\mathrm{REC}}_{z,l,j} - \overline{y}^{\mathrm{REC}}_{z,l-1,j}$ indices $n$ with $\tilde{y}^{\mathrm{REC}}_{z,l,j,n} > \tilde{y}^{\mathrm{REC}}_{z,l-1,j,n}$, and otherwise no facilities with index $n \leq |\mathcal{N}^{\mathrm{REC}}_l| - L$ incur expansion costs.  Then, the relevant objective terms are equal:
\begin{equation*}
    f^{\mathrm{EXP}}_{z,l,j} \left [ \overline{y}^{\mathrm{REC}}_{z,l,j} - \overline{y}^{\mathrm{REC}}_{z,l-1,j} \right ]^+ + f^{\mathrm{EXP}}_{z,l,j} \overline{y}^{\mathrm{REC,EXP}}_{z,l,l',j} = \sum_{n \in \mathcal{N}^{\mathrm{REC}}_{l}} f^{\mathrm{EXP}}_{z,l,j} \mathbb{I}[\tilde{y}^{\mathrm{REC}}_{z,l,j,n} > \tilde{y}^{\mathrm{REC}}_{z,l-1,j,n}]
\end{equation*}
and
\begin{equation*}
\begin{aligned}
    &\overline{y}^{\mathrm{REC}}_{z,l,j} f^{\mathrm{REC}}_{z,j}(u^{\mathrm{REC}})  + \sum_{l' \in \mathcal{L}} f^{\mathrm{REC}}_{z,j}(\overline{y}^{\mathrm{REC}}_{z,l,l',j,+}) = \sum_{n \in \mathcal{N}^{\mathrm{REC}}_{l}} f^{\mathrm{REC}}_{z,j} (\tilde{y}^{\mathrm{REC}}_{z,l,j,n}).
\end{aligned}
\end{equation*}
The planning cost of $\overline{y}^{\mathrm{REC}}$ in the reformulation is the same as that of $\tilde{y}^{\mathrm{REC}}$ in~\eqref{eq:deterministicP}.  Thus, $(\tilde{y},\overline{x})$ has the same objective value in~\eqref{eq:deterministicP} as $(\overline{y},\overline{x})$ has in the reformulation.

In the opposite direction, let $(\hat{y},\hat{x})$ be an optimal solution for~\eqref{eq:deterministicP} with fixed costs of expansion included in the objective function.  As only the cost function associated with variables $y$ has changed,~\eqref{eq:deterministicP} remains feasible and bounded.  Consider the optimization problem
\begin{equation}
    \label{eq:restrictedFC}
    \begin{aligned}
        \min_{x,y} \quad & \sum_{t \in \mathcal{T}} (1 - \gamma)^{t-1} \left ( C^{\text{PL}}_t(y) + \sum_{\omega \in \Omega_{\sigma_t}} p_\omega C^{\text{OP}}_{\omega,t}(x) \right ) + \sum_{z \in \mathcal{Z}} \sum_{j \in \mathcal{J}} \sum_{l \in \mathcal{L}} f^{\mathrm{EXP}}_{z,l,j} |\overline{\mathcal{N}}^{\mathrm{REC}}_{z,l,j}|\\
        \text{s.t.} \quad & \eqref{eq:constrProd},\ \eqref{eq:constrMatFlow},\ \eqref{eq:constrCap}\\
        & y^{\mathrm{REC}}_{z,l,j,n} \leq y^{\mathrm{REC}}_{z,l-1,j,n} \quad \forall j \in \mathcal{J},\ l \in \mathcal{L},\ z \in \mathcal{Z},\ n \in \mathcal{N}^{\mathrm{REC}}_l \setminus \overline{\mathcal{N}}^{\text{REC}}_{z,l,j}\\
        & x \geq 0,\ y \geq 0,
    \end{aligned}
\end{equation}
where $\overline{\mathcal{N}}^{\mathrm{REC}}_{z,l,j} := \{n \in \mathcal{N}^{\mathrm{REC}}_{l} \,:\, \hat{y}^{\mathrm{REC}}_{z,l,j,n} > \hat{y}^{\mathrm{REC}}_{z,l-1,j,n}\}$.  We continue to define $y^{\mathrm{REC}}_{z,0,j,n} = 0$.  As the indicators in the expansion cost expression do not appear in~\eqref{eq:restrictedFC}, the objective function is concave over a polyhedral feasible region.  By the same logic as in Proposition~\ref{prop:finiteP}, the problem~\eqref{eq:restrictedFC} has an optimal solution at an extreme point of its feasible region. Let $(\tilde{y},\tilde{x})$ be such a solution.
Problem~\eqref{eq:restrictedFC} is a restriction of~\eqref{eq:deterministicP} with an objective function that overestimates that of~\eqref{eq:deterministicP} on its feasible set by incurring the expansion costs for all facilities that can feasibly expand.  We define $\hat{Z}^{\mathrm{P}}$ as the optimal objective value of~\eqref{eq:deterministicP}, $\tilde{Z}^{\mathrm{P}}$ as the objective value of $(\tilde{y},\tilde{x})$ in~\eqref{eq:deterministicP}, $\hat{Z}$ as the objective value of $(\hat{y},\hat{x})$ in~\eqref{eq:restrictedFC}, and $\tilde{Z}$ as the objective value of $(\tilde{y},\tilde{x})$ in~\eqref{eq:restrictedFC}.
The solution $(\hat{y},\hat{x})$ remains feasible for~\eqref{eq:restrictedFC} with the same objective value as in~\eqref{eq:deterministicP}, so $\hat{Z}^{\mathrm{P}} = \hat{Z}$.  As~\eqref{eq:restrictedFC} is a restriction of~\eqref{eq:deterministicP}, $(\hat{y},\hat{x})$ is an optimal solution to~\eqref{eq:restrictedFC} and $\hat{Z} = \tilde{Z}$.  By the overestimation property of the objective function of \eqref{eq:restrictedFC}, $\tilde{Z} \geq \tilde{Z}^{\mathrm{P}}$.  Finally, $\tilde{Z}^{\mathrm{P}} \geq \hat{Z}^{\mathrm{P}}$ by optimality for \eqref{eq:deterministicP}.  Therefore, $\tilde{Z} \geq \tilde{Z}^{\mathrm{P}} \geq \hat{Z}^{\mathrm{P}} = \tilde{Z}$, so $\hat{Z}^{\mathrm{P}} = \tilde{Z}^{\mathrm{P}}$ and $(\tilde{y},\tilde{z})$ is optimal for~\eqref{eq:deterministicP}.

For each $(z,l,j)$, there is at most one index $\overline{n}_{z,l,j}$ for which the following property is not satisfied:
\begin{equation}
    \label{eq:FCextPt}
    \tilde{y}^{\mathrm{REC}}_{z,l,j,n} \in \begin{cases}
        \{0,u^{\mathrm{REC}}\} \quad & n \in \overline{\mathcal{N}}^{\text{REC}}_{z,l,j}\\
        \{0,\tilde{y}^{\mathrm{REC}}_{z,l-1,j,n}\} \quad & n \in \mathcal{N}^{\mathrm{REC}}_l \setminus \overline{\mathcal{N}}^{\text{REC}}_{z,l,j}.
    \end{cases}
\end{equation}
Suppose otherwise.  Then, there must be two indices $n_1,n_2$ that violate~\eqref{eq:FCextPt}.  Define
\begin{equation*}
    \begin{aligned}
        \delta_1 & = \begin{cases}
            \min \{\tilde{y}^{\mathrm{REC}}_{z,l,j,n_1}, u^{\mathrm{REC}} - \tilde{y}^{\mathrm{REC}}_{z,l,j,n_1} \} \quad & n_1 \in \overline{\mathcal{N}}^{\text{REC}}_{z,l,j}\\
            \min \{\tilde{y}^{\mathrm{REC}}_{z,l,j,n_1}, \tilde{y}^{\mathrm{REC}}_{z,l-1,j,n_1} - \tilde{y}^{\mathrm{REC}}_{z,l,j,n_1}\} \quad & n_1 \in \mathcal{N}^{\mathrm{REC}}_l \setminus \overline{\mathcal{N}}^{\text{REC}}_{z,l,j};
        \end{cases}\\
        \delta_2 & = \begin{cases}
            \min \{\tilde{y}^{\mathrm{REC}}_{z,l,j,n_2}, u^{\mathrm{REC}} - \tilde{y}^{\mathrm{REC}}_{z,l,j,n_2} \} \quad & n_2 \in \overline{\mathcal{N}}^{\text{REC}}_{z,l,j}\\
            \min \{\tilde{y}^{\mathrm{REC}}_{z,l,j,n_2}, \tilde{y}^{\mathrm{REC}}_{z,l-1,j,n_2} - \tilde{y}^{\mathrm{REC}}_{z,l,j,n_2}\} \quad & n_2 \in \mathcal{N}^{\mathrm{REC}}_l \setminus \overline{\mathcal{N}}^{\text{REC}}_{z,l,j}.
        \end{cases}
    \end{aligned}
\end{equation*}
Let $\delta = \min\{\delta_1, \delta_2\}$.  By construction, $\delta > 0$.  We define perturbed solutions 
\begin{equation*}
    y^{\mathrm{REC}}_{z,l,j,n}(1) = \begin{cases}
        \tilde{y}^{\mathrm{REC}}_{z,l,j,n} \quad & n \not \in \{n_1,n_2\}\\
        \tilde{y}^{\mathrm{REC}}_{z,l,j,n} + \delta \quad & n = n_1\\
        \tilde{y}^{\mathrm{REC}}_{z,l,j,n} - \delta \quad & n = n_2,
    \end{cases}
\end{equation*}
and
\begin{equation*}
    y^{\mathrm{REC}}_{z,l,j,n}(2) = \begin{cases}
        \tilde{y}^{\mathrm{REC}}_{z,l,j,n} \quad & n \not \in \{n_1,n_2\}\\
        \tilde{y}^{\mathrm{REC}}_{z,l,j,n} - \delta \quad & n = n_1\\
        \tilde{y}^{\mathrm{REC}}_{z,l,j,n} + \delta \quad & n = n_2.
    \end{cases}
\end{equation*}
Observe that $\sum_{n \in \mathcal{N}^{\mathrm{REC}}_l} y^{\mathrm{REC}}_{z,l,j,n}(1) = \sum_{n \in \mathcal{N}^{\mathrm{REC}}} \tilde{y}^{\mathrm{REC}}_{z,l,j,n}$, so the capacity remains the same in both solutions.  Also, by definition of $\delta$, we have that $0 \leq y^{\mathrm{REC}}_{z,l,j,n}(1) \leq u^{\mathrm{REC}}$ for all $n \in \overline{\mathcal{N}}^{\text{REC}}_{z,l,j}$, and $0 \leq y^{\mathrm{REC}}_{z,l,j,n}(1) \leq \tilde{y}^{\mathrm{REC}}_{z,l-1,j,n} = y^{\mathrm{REC}}_{z,l-1,j,n}(1)$ for all $n \in \mathcal{N}^{\mathrm{REC}}_l \setminus \overline{\mathcal{N}}^{\text{REC}}_{z,l,j}$.  Therefore, $(y(1),\tilde{x})$ is feasible for~\eqref{eq:restrictedFC}.  Similarly, $(y(2),\tilde{x})$ is also feasible for~\eqref{eq:restrictedFC}.  Finally, we note that $(\tilde{y},\tilde{x}) = \frac{1}{2}(y(1),\tilde{x}) + \frac{1}{2} (y(2),\tilde{x})$, so $(\tilde{y},\tilde{x})$ is a convex combination of feasible points for~\eqref{eq:restrictedFC}, and is not an extreme point of the feasible set.  By this contradiction, there is at most one index $\overline{n}_{z,l,j}$ for which condition~\eqref{eq:FCextPt} is not satisfied.  If no such index exists, we select one arbitrarily.

We next prove that $\tilde{y}^\mathrm{REC}_{z,l,j,n} \in \{0,u^{\mathrm{REC}}\}$ for all $n \in \mathcal{N}^{\mathrm{REC}}_l \setminus \{\overline{n}_{z,l',j}\}_{l' \in \mathcal{L}}$.  We proceed inductively.  As a base case, when $l = 0$, we have that $\tilde{y}^{\mathrm{REC}}_{z,0,j,n} = 0$ for all $n$.  As an inductive hypothesis, suppose that $\tilde{y}^\mathrm{REC}_{z,l-1,j,n} \in \{0,u^{\mathrm{REC}}\}$ for all $n \in \mathcal{N}^{\mathrm{REC}}_{l-1} \setminus \{\overline{n}_{z,l',j}\}_{l'=1}^{l-1}$ at period $l-1$.  
For $n \in \overline{\mathcal{N}}^{\mathrm{REC}}_{z,l,j} \setminus \{\overline{n}_{z,l,j}\}$, condition~\eqref{eq:FCextPt} gives that $\tilde{y}^{\mathrm{REC}}_{z,l,j,n} \in \{0,u^{\mathrm{REC}}\}$.  Otherwise, for $n \in \mathcal{N}^{\mathrm{REC}}_l \setminus \overline{\mathcal{N}}^{\text{REC}}_{z,l,j}$ with $n \neq \overline{n}_{z,l,j}$, the condition gives that $\tilde{y}^{\mathrm{REC}}_{z,l,j,n} \in \{0,\tilde{y}^{\mathrm{REC}}_{z,l-1,j,n}\}$.  By the inductive hypothesis, $\tilde{y}^{\mathrm{REC}}_{z,l-1,j,n} \in \{0,u^{\mathrm{REC}}\}$ for $n \not \in \{\overline{n}_{z,l',j}\}_{l'=1}^{l-1}$.  Thus, $\tilde{y}^\mathrm{REC}_{z,l,j,n} \in \{0,u^{\mathrm{REC}}\}$ for all $n \in \mathcal{N}^{\mathrm{REC}}_l \setminus \{\overline{n}_{z,l',j}\}_{l'=1}^l$, and the inductive hypothesis is satisfied at period $l$.  Therefore,
\begin{equation}
    \label{eq:FCextPtFull}
        \tilde{y}^\mathrm{REC}_{z,l,j,n} \in \{0,u^{\mathrm{REC}}\} \quad \forall n \in \mathcal{N}^{\mathrm{REC}}_l \setminus \{\overline{n}_{z,l',j}\}_{l' \in \mathcal{L}}.
\end{equation}

We use property~\eqref{eq:FCextPtFull} to construct an equivalent feasible solution for the reformulated constraints~\eqref{eq:FCreformulation}.  Define
\begin{equation*}
    \begin{aligned}
        \overline{y}^{\mathrm{REC}}_{z,l,j} & = \mathcal{C}(\{\tilde{y}^{\mathrm{REC}}_{z,l,j,n}\}_{n \in \mathcal{N}^{\mathrm{REC}}_{l} \setminus \{\overline{n}_{z,l,j}\}_{l \in \mathcal{L}}}; u^{\mathrm{REC}});\\
        \overline{y}^{\mathrm{REC}}_{z,l,l',j,+} & = \tilde{y}^{\mathrm{REC}}_{z,l,j,\overline{n}_{z,l',j}};\ \text{and}\\
        \overline{y}^{\mathrm{REC,EXP}}_{z,l,l',j} & = \mathbb{I}[\tilde{y}^{\mathrm{REC}}_{z,l,j,\overline{n}_{z,l',j}} > \tilde{y}^{\mathrm{REC}}_{z,l-1,j,\overline{n}_{z,l',j}}].
    \end{aligned}
\end{equation*}
By construction, $(\overline{y}^{\mathrm{REC}}_{z,l,j},\overline{y}^{\mathrm{REC}}_{z,l,l',j,+},\overline{y}^{\mathrm{REC,EXP}}_{z,l,l',j})$ satisfies the domain constraints~\eqref{eq:FCreformulationIntegerDom},~\eqref{eq:FCreformulationIndDom}~and~\eqref{eq:FCreformulationContinuousDom}.  The expansion indicator constraint~\eqref{eq:FCreformulationIndUB} is also satisfied by definition as $\overline{y}^{\mathrm{REC}}_{z,l,l',j,+} \in [0,u^{\mathrm{REC}}]$.  Therefore, the constructed solution is feasible for the reformulation~\eqref{eq:FCreformulation}.  By property~\eqref{eq:FCextPtFull}, the total capacity of both solutions are equal:
\begin{equation*}
    \sum_{n \in \mathcal{N}^{\mathrm{REC}}_l} \tilde{y}^{\mathrm{REC}}_{z,l,j,n} = u^{\mathrm{REC}} \mathcal{C}(\{\tilde{y}^{\mathrm{REC}}_{z,l,j,n}\}_{n \in \mathcal{N}^{\mathrm{REC}}_{l} \setminus \{\overline{n}_{z,l',j}\}_{l' \in \mathcal{L}}}; u^{\mathrm{REC}}) + \sum_{l' \in \mathcal{L}} \tilde{y}^{\mathrm{REC}}_{z,l,j,\overline{n}_{z,l',j}} = u^{\mathrm{REC}} \overline{y}^{\mathrm{REC}}_{z,l,j} + \sum_{l' \in \mathcal{L}} \overline{y}^{\mathrm{REC}}_{z,l,l',j,+}.
\end{equation*}
Then, the feasibility of $(\tilde{y},\tilde{x})$ for~\eqref{eq:constrProd}-\eqref{eq:constrCap} implies that the solution $(\overline{y},\tilde{x})$ remains feasible for the reformulation.  By~\eqref{eq:FCextPtFull}, the indicators in the objective satisfy 
\begin{equation*}
    \begin{aligned}
        \sum_{n \in \mathcal{N}^{\mathrm{REC}}_l \setminus \{\overline{n}_{z,l',j}\}_{l' \in \mathcal{L}}} \mathbb{I}[\tilde{y}^{\mathrm{REC}}_{z,l,j,n} > \tilde{y}^{\mathrm{REC}}_{z,l-1,j,n}] & = \sum_{n \in \mathcal{N}^{\mathrm{REC}}_l \setminus \{\overline{n}_{z,l',j}\}_{l' \in \mathcal{L}}} \mathbb{I}[\tilde{y}^{\mathrm{REC}}_{z,l,j,n} = u^{\mathrm{REC}},\ \tilde{y}^{\mathrm{REC}}_{z,l-1,j,n} \neq u^{\mathrm{REC}}]\\
        & = \overline{y}^{\mathrm{REC}}_{z,l,j} - \sum_{n \in \mathcal{N}^{\mathrm{REC}}_l \setminus \{\overline{n}_{z,l',j}\}_{l' \in \mathcal{L}}}  \mathbb{I}[\tilde{y}^{\mathrm{REC}}_{z,l,j,n} = u^{\mathrm{REC}},\ \tilde{y}^{\mathrm{REC}}_{z,l-1,j,n} = u^{\mathrm{REC}}]\\
        & \geq \overline{y}^{\mathrm{REC}}_{z,l,j} - \overline{y}^{\mathrm{REC}}_{z,l-1,j}.
    \end{aligned}
\end{equation*}
As the indicators are also nonnegative, it follows that
\begin{equation*}
\begin{aligned}
    \sum_{n \in \mathcal{N}^{\mathrm{REC}}_l} f^{\mathrm{EXP}}_{z,l,j} \mathbb{I}[\tilde{y}^{\mathrm{REC}}_{z,l,j,n} > \tilde{y}^{\mathrm{REC}}_{z,l-1,j,n}] \geq f^{\mathrm{EXP}}_{z,l,j} \left[ \overline{y}^{\mathrm{REC}}_{z,l,j} - \overline{y}^{\mathrm{REC}}_{z,l-1,j} \right ] ^+ + \sum_{l' \in \mathcal{L}} f^{\mathrm{EXP}}_{z,l,j} \overline{y}^{\mathrm{REC,EXP}}_{z,l,l',j}.
\end{aligned}
\end{equation*}
We also have
\begin{equation*}
    \sum_{n \in \mathcal{N}^{\mathrm{REC}}_l} f^{\mathrm{REC}}_{z,j}(\tilde{y}^{\mathrm{REC}}_{z,l,j,n}) = \overline{y}^{\mathrm{REC}}_{z,l,j} f^{\mathrm{REC}}_{z,j}(u^{\mathrm{REC}}) + \sum_{l' \in \mathcal{L}}f^{\mathrm{REC}}_{z,j}(\overline{y}^{\mathrm{REC}}_{z,l,l',j,+})
\end{equation*}
Then, $(\overline{y},\tilde{x})$ has an objective value in the reformulation no larger than $(\tilde{y},\tilde{x})$ has in~\eqref{eq:deterministicP}.

To conclude, we first observed that any optimal solution for the reformulation has a corresponding feasible solution for~\eqref{eq:deterministicP} with the same objective value.  Second, we have that the solution $(\overline{y},\tilde{x})$ is feasible for the reformulation with an objective value no larger than that of $(\tilde{y},\tilde{x})$, which is the optimal objective value for~\eqref{eq:deterministicP}.  If $(\overline{y},\tilde{x})$ is not optimal for the reformulation, the reformulation has an optimal solution with a smaller objective than that of $(\tilde{y},\tilde{x})$, and there is a feasible solution for~\eqref{eq:deterministicP} with a smaller objective value than $(\tilde{y},\tilde{x})$, which contradicts its optimality.  Therefore, $(\overline{y},\tilde{x})$ is optimal for the reformulation. Then, there exists a feasible solution for~\eqref{eq:deterministicP} with the same objective value as $(\overline{y},\tilde{x})$ in the reformulation; this objective value is no larger than that of $(\tilde{y},\tilde{x})$. As $(\tilde{y},\tilde{x})$ is optimal for~\eqref{eq:deterministicP}, this objective value also cannot be smaller than that of $(\tilde{y},\tilde{x})$.  Thus, the objective values of $(\overline{y},\tilde{x})$ and $(\tilde{y},\tilde{x})$ are equal, and both solutions are optimal for their respective models with the same objective values. Therefore,~\eqref{eq:deterministicP} and the reformulation have the same optimal objective.

\bibliographystyleappendix{informs2014}
\bibliographyappendix{appendix.bib}

\end{document}